\documentclass[letterpaper,10pt]{article}

\usepackage[margin=0.5in]{geometry}

\usepackage{fullpage}
\usepackage{enumerate}
\usepackage{authblk}
\usepackage[colorinlistoftodos]{todonotes}
\usepackage{verbatim}
\usepackage{section, amsthm, textcase, setspace, amssymb, lineno, amsmath, amssymb, amsfonts, latexsym, fancyhdr, longtable, ulem, mathtools}
\usepackage{epsfig, graphicx, pstricks,pst-grad,pst-text,tikz,colortbl}
\usepackage{epsf}
\usepackage{graphicx, color}
\usepackage{float}
\usepackage[rflt]{floatflt}
\usepackage{amsfonts}
\usepackage{latexsym,enumitem}
\usetikzlibrary{fit,matrix,positioning}
\usepackage{pdflscape}
\usetikzlibrary{decorations.pathreplacing}
\usepackage{mathrsfs}
\usepackage{faktor}

\definecolor{darkgreen}{rgb}{0, 0.5, 0}

\newtheorem{theorem}{Theorem}
\newtheorem{lemma}[theorem]{Lemma}
\newtheorem{corollary}[theorem]{Corollary}

\newtheorem{conj}[theorem]{Conjecture}

\newtheorem{definition}[theorem]{Definition}

\newtheorem{Ex}[theorem]{Example}

\newtheorem*{theorem*}{Theorem}
\newtheorem{remark}[theorem]{Remark}

\newcommand{\bl}[1]{{\color{blue} #1}}

\newcommand{\ind}{{\rm ind \hspace{.0cm}}}

\DeclareMathOperator{\ad}{ad}

\begin{document}

\title{On toral posets and contact Lie algebras}

\author[*]{Nicholas W. Mayers}
\author[**]{Nicholas Russoniello}

\affil[*]{Department of Mathematics, North Carolina State University, Raleigh, NC 27695}
\affil[**]{Department of Mathematics, College of William \& Mary, Williamsburg, VA 23185}

\maketitle


\begin{abstract}
\noindent
A $(2k+1)-$dimensional Lie algebra is called contact if it admits a one-form $\varphi$ such that $\varphi\wedge(d\varphi)^k\neq 0.$ Here, we extend recent work to describe a combinatorial procedure for generating contact, type-A Lie poset algebras whose associated posets have chains of arbitrary cardinality, and we conjecture that our construction leads to a complete characterization.
\end{abstract}

\noindent
\textit{Mathematics Subject Classification 2020}: 17B30, 05E16

\noindent 
\textit{Key Words and Phrases}: contact Lie algebra,  Lie poset algebra, index, contactization

\section{Introduction}

Type-A Lie poset algebras are subalgebras of $\mathfrak{sl}(n)$ which are defined by an associated poset. Recent work has focused on studying the structure of these algebras combinatorially \textbf{\cite{CG, SeriesA, Binary,ContactLiePoset, breadth}}. Of interest here is the work of (\textbf{\cite{ContactLiePoset}}, 2021), where the authors give a complete characterization of those posets with chains of cardinality at most three whose type-A Lie poset algebras are ``contact." Building on the work in both (\textbf{\cite{Binary}}, 2021) and \textbf{\cite{ContactLiePoset}}, we extend the constructions given in \textbf{\cite{ContactLiePoset}} to include posets with chains of arbitrary size, and we conjecture a complete description of all posets associated to contact, type-A Lie poset algebras.
\bigskip

A ($2k+1)-$dimensional Lie algebra $\mathfrak{g}$ is said to be \textit{contact} if there exists $\varphi\in\mathfrak{g}^*$ satisfying $\varphi\wedge (d\varphi)^k\ne 0$, where $d\varphi(x,y)=-\varphi([x,y])$, for all $x,y\in\mathfrak{g}.$ The one-form $\varphi$ is called a \textit{contact form}, $\varphi\wedge (d\varphi)^k$ is a \textit{volume form} on the underlying Lie group $G,$ and $\varphi$ defines a \textit{contact structure}, i.e., a smooth, maximally non-integrable, codimension-one distribution on $G.$ While contact manifolds and their applications have been extensively explored for a number of years \textbf{\cite{Boyer,Geiges,Ghosh,Herczeg,Machon}}, the study of contact structures arising from Lie algebras, i.e., \textit{left-invariant} contact structures on Lie groups, has only recently garnered significant attention \textbf{\cite{Alvarez,Diatta,GR,Khakimdjanov,RS,InvCon}}. Included in such works are various methods for constructing and classifying families of contact Lie algebras; however, most approaches are geometric or algebraic in nature. Here, as in \textbf{\cite{ContactLiePoset}} and (\textbf{\cite{seaweedA}}, 2022), our aim is to construct an infinite family of contact Lie algebras combinatorially. Conducive to our approach are recent combinatorial results regarding a Lie-algebraic invariant called the ``index."

The \textit{index} of a Lie algebra was first introduced by Dixmier (\textbf{\cite{D}}, 1977) and is defined as $$\ind(\mathfrak{g})=\min_{\varphi\in\mathfrak{g}^*}\dim(\ker(d\varphi)),$$ where a one-form $\varphi$ for which $\ind(\mathfrak{g})=\dim(\ker(d\varphi))$ is called \textit{regular}. Contact Lie algebras have index one,\footnote{The converse is not true in general. For example, the Lie algebra $\mathfrak{g}=\langle e_1,e_2,e_3\rangle$ with relations $[e_1,e_2]=e_2$ and $[e_1,e_3]=e_3$ has index one but is not contact.} and a contact form is necessarily regular. Consequently, our approach involves first identifying Lie algebras with index one and then determining which of those admits a contact form. In order to find index-one algebras, we restrict our attention to a family of subalgebras of $A_{n-1}=\mathfrak{sl}(n)$ called ``type-A Lie poset algebras," for which combinatorial index formulas have been of topical interest.

Formally, type-A Lie poset algebras are Lie subalgebras of $\mathfrak{sl}(n)$ which lie between a Cartan and associated Borel subalgebra. Such algebras can also be defined in terms of an associated poset (see Section~\ref{sec:poset}). In \textbf{\cite{Binary}}, the authors provide a combinatorial recipe -- in terms of ``building-block" posets and ``gluing rules" -- for constructing so-called ``toral" posets whose type-A Lie poset algebras have index given by a combinatorial formula. The authors' primary motivation for the introduction of toral posets was to characterize type-A Lie poset algebras that are \textit{Frobenius}, i.e., have index equal to zero. It is conjectured in \textbf{\cite{Binary}} that every Frobenius, type-A Lie poset algebra corresponds to a toral poset. However, as is, toral posets cannot correspond to contact, type-A Lie poset algebras. Having said that, results of \textbf{\cite{ContactLiePoset}} suggest that a slight modification to the notion of toral poset can allow for the associated algebra to be contact.

Using combinatorial index formulas for type-A Lie poset algebras introduced in (\textbf{\cite{SeriesA}}, 2021), the authors of \textbf{\cite{ContactLiePoset}} give a complete characterization of those posets with chains of cardinality at most three whose associated type-A Lie poset algebras are contact. Interestingly, such posets are defined via a combinatorial recipe tantamount to that of toral posets with one additional ``building-block" poset required. Inspired by this observation, here we extend the notion of toral poset via the introduction of ``contact building blocks," effectively combining the construction paradigms of \textbf{\cite{Binary}} and \textbf{\cite{ContactLiePoset}}. As a result, we are able to generate toral posets (with chains of arbitrary size) associated to contact, type-A Lie poset algebras. Further, we conjecture that all contact, type-A Lie poset algebras are either associated to posets of this form or associated to a direct sum of posets of this form.

The paper is organized as follows. In Section~\ref{sec:poset}, we give the necessary preliminaries from the theory of posets and Lie algebras. In Section~\ref{sec:ntp}, we introduce the notion of toral-pair, which forms the original set of building blocks of toral posets as defined in \textbf{\cite{Binary}}, and provide some new examples. In Section~\ref{sec:ctp}, we define the contact analogue of a toral-pair, leading to the extended notion of toral poset given in Section~\ref{sec:contoral}. In Section~\ref{sec:contactforms}, we characterize toral posets corresponding to contact, type-A Lie poset algebras. Directions for further research involving ``combinatorial contactization" are discussed in Section~\ref{sec:epilogue}.

\section{Preliminaries}\label{sec:poset}

A \textit{finite poset} $(\mathcal{P}, \preceq_{\mathcal{P}})$ consists of a finite set $\mathcal{P}=\{1,2,\hdots,n\}$ together with a binary relation $\preceq_{\mathcal{P}}$ which is reflexive, anti-symmetric, and transitive. We tacitly assume that if $x\preceq_{\mathcal{P}}y$ for $x,y\in\mathcal{P}$, then $x\le y$, where $\le$ denotes the natural ordering on $\mathbb{Z}$. When no confusion will arise, we simply denote a poset $(\mathcal{P}, \preceq_{\mathcal{P}})$ by $\mathcal{P}$, and $\preceq_{\mathcal{P}}$ by $\preceq$. Two posets $\mathcal{P}$ and $\mathcal{Q}$ are \textit{isomorphic} if there exists an order-preserving bijection $\mathcal{P}\to\mathcal{Q}$.

Let $\mathcal{P}$ be a finite poset and take $x,y\in\mathcal{P}$. If $x\preceq y$ and $x\neq y$, then we call $x\preceq y$ a \textit{strict relation} and write $x\prec y$. Let $Rel(\mathcal{P})$, $Ext(\mathcal{P})$, and $Rel_E(\mathcal{P})$ denote, respectively, the set of strict relations between elements of $\mathcal{P}$, the set of minimal and maximal elements of $\mathcal{P}$, and the set of strict relations between elements of $Ext(\mathcal{P})$. If $x\prec y$ and there exists no $z\in \mathcal{P}$ satisfying $x\prec z\prec y$, then $x\prec y$ is a \textit{covering relation}. Covering relations are used to define a visual representation of $\mathcal{P}$ called the \textit{Hasse diagram} -- a graph whose vertices correspond to elements of $\mathcal{P}$ and whose edges correspond to covering relations (see  Figure~\ref{fig:Hasse} (a)). A poset $\mathcal{P}$ is called \textit{connected} if the Hasse diagram of $\mathcal{P}$ is connected as a graph and called \textit{disconnected} otherwise. Extending the Hasse diagram of $\mathcal{P}$ by allowing all elements of $Rel(\mathcal{P})$ to define edges results in the \textit{comparability graph} of $\mathcal{P}$ (see Figure~\ref{fig:Hasse} (b)).

For a subset $S\subset\mathcal{P}$, the \textit{induced subposet generated by $S$} is the poset $\mathcal{P}_S$ on $S$, where $i\prec_{\mathcal{P}_S}j$ if and only if $i\prec_{\mathcal{P}}j$. A totally ordered subset $S\subset\mathcal{P}$ is called a \textit{chain}. The \textit{height} of $\mathcal{P}$ is one less than the largest cardinality of a chain in $\mathcal{P}$. One can define a simplicial complex $\Sigma(\mathcal{P})$ by having chains of cardinality $n$ in $\mathcal{P}$ define the $(n-1)-$dimensional faces of $\Sigma(\mathcal{P})$ (see Figure~\ref{fig:Hasse} (c)). A subset $I\subset\mathcal{P}$ is an \textit{order ideal} if $y\prec x\in I$ implies $y\in I$. Similarly, a subset $F\subset\mathcal{P}$ is a \textit{filter} if $F\ni x\prec y$ implies $y\in F$.


\begin{Ex}\label{ex:not}
Consider the poset $\mathcal{P}=\{1,2,3,4\}$ with $1\prec 2\prec 3,4$. We have $$Rel(\mathcal{P})=\{1\prec 2,1\prec 3,1\prec 4,2\prec 3,2\prec 4\},~~ Ext(\mathcal{P})=\{1,3,4\},~~\text{and}~~ Rel_E(\mathcal{P})=\{1\prec 3, 1\prec 4\}.$$

\fbox{\begin{minipage}{42em} 
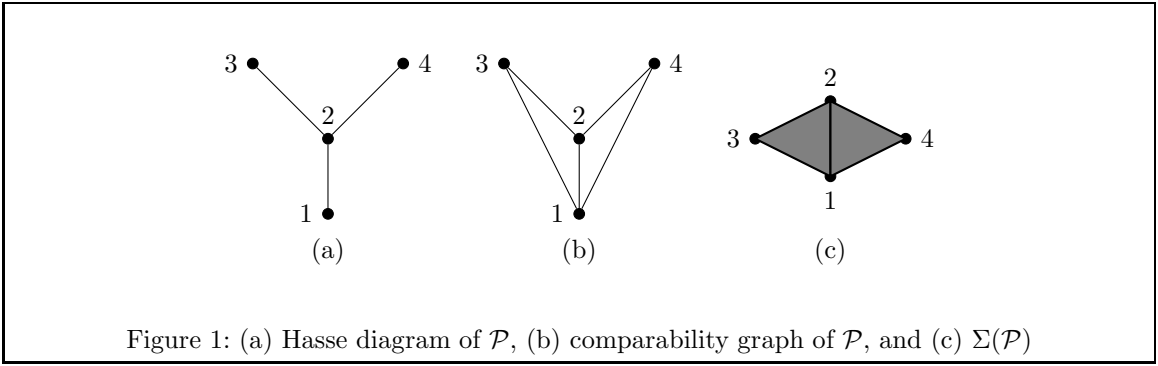
\begin{figure}[H]
$$\begin{tikzpicture}
	\node (1) at (0, 0) [circle, draw = black, fill = black, inner sep = 0.5mm, label=left:{$1$}]{};
	\node (2) at (0, 1)[circle, draw = black, fill = black, inner sep = 0.5mm, label=above:{$2$}] {};
	\node (3) at (-1, 2) [circle, draw = black, fill = black, inner sep = 0.5mm, label=left:{$3$}] {};
	\node (4) at (1, 2) [circle, draw = black, fill = black, inner sep = 0.5mm, label=right:{$4$}] {};
	\node (5) at (0,-0.5) {(a)};
    \draw (1)--(2);
    \draw (2)--(3);
    \draw (2)--(4);
\end{tikzpicture}\quad\begin{tikzpicture}
	\node (1) at (0, 0) [circle, draw = black, fill = black, inner sep = 0.5mm, label=left:{$1$}]{};
	\node (2) at (0, 1)[circle, draw = black, fill = black, inner sep = 0.5mm, label=above:{$2$}] {};
	\node (3) at (-1, 2) [circle, draw = black, fill = black, inner sep = 0.5mm, label=left:{$3$}] {};
	\node (4) at (1, 2) [circle, draw = black, fill = black, inner sep = 0.5mm, label=right:{$4$}] {};
	\node (5) at (0,-0.5) {(b)};
    \draw (1)--(2);
    \draw (2)--(3);
    \draw (2)--(4);
    \draw (1)--(3);
    \draw (1)--(4);
\end{tikzpicture}\quad\begin{tikzpicture}
	\node (1) at (0, 0) [circle, draw = black, fill = black, inner sep = 0.5mm, label=below:{$1$}]{};
	\node (2) at (-1, 0.5)[circle, draw = black, fill = black, inner sep = 0.5mm,label=left:{$3$}] {};
    \node (3) at (0, 1) [circle, draw = black, fill = black, inner sep = 0.5mm,label=above:{$2$}]{};
    \node (4) at (1, 0.5) [circle, draw = black, fill = black, inner sep = 0.5mm,label=right:{$4$}]{};
    \node (5) at (0,-1) {(c)};
    \draw [fill=gray,thick] (0, 0)--(-1, 0.5)--(0, 1)--(0, 0);
    \draw [fill=gray,thick] (0, 0)--(1, 0.5)--(0, 1)--(0, 0);
\end{tikzpicture}$$
\caption{(a) Hasse diagram of $\mathcal{P}$, (b) comparability graph of $\mathcal{P}$, and (c) $\Sigma(\mathcal{P})$}\label{fig:Hasse}
\end{figure} \end{minipage}}
\bigskip

\noindent
Note that $\{1,2\}\subset \mathcal{P}$ is both an order ideal and a chain of $\mathcal{P}$, but not a filter; meanwhile 
$\{3,4\}\subset \mathcal{P}$ is a filter, but neither an order ideal nor a chain. 
\end{Ex}

Let $\mathcal{P}$ be a finite poset. The (associative) \textit{incidence algebra} $A(\mathcal{P})=A(\mathcal{P}, \mathbb{C})$ is the span over $\mathbb{C}$ of elements $E_{p_i,p_j}$, for $p_i\preceq p_j$, with multiplication given by setting $E_{p_i,p_j}E_{p_k,p_l}=E_{p_i,p_l}$ if $p_j=p_k$ and $0$ otherwise.  The \textit{trace} of an element $\sum c_{p_i,p_j}E_{p_i,p_j}$ is $\sum c_{p_i,p_i}$.  We can equip $A$ with the commutator bracket $[a,b]=ab-ba$, where juxtaposition denotes the product in $A(\mathcal{P})$, to produce the \textit{Lie poset algebra} $\mathfrak{g}(\mathcal{P})
=\mathfrak{g}(\mathcal{P}, \mathbb{C})$. 

\begin{Ex}\label{ex:gstargate}
If $\mathcal{P}$ is the poset of Example~\ref{ex:not}, then $\mathfrak{g}(\mathcal{P})$ is the span over $\mathbb{C}$ of the elements of $$\{E_{1,1},E_{2,2},E_{3,3},E_{4,4},E_{1,2},E_{1,3},E_{1,4},E_{2,3},E_{2,4}\}.$$
\end{Ex}

\noindent
Restricting to the trace-zero elements of $\mathfrak{g}(\mathcal{P})$ results in the \textit{type-A Lie poset algebra} $\mathfrak{g}_A(\mathcal{P})$.

\begin{Ex}
If $\mathfrak{g}(\mathcal{P})$ is as in Example~\ref{ex:gstargate}, then $\mathfrak{g}_A(\mathcal{P})$ is the span over $\mathbb{C}$ of the elements of $$\{E_{1,1}-E_{2,2},E_{2,2}-E_{3,3},E_{3,3}-E_{4,4},E_{1,2},E_{1,3},E_{1,4},E_{2,3},E_{2,4}\}.$$
\end{Ex}

\begin{remark}
Isomorphic posets correspond to isomorphic \textup(type-A\textup) Lie poset algebras.
\end{remark}

Of principal interest here are type-A Lie poset algebras that are contact. Recall that a $(2k+1)-$dimensional Lie algebra $\mathfrak{g}$ is \textit{contact} if there exists $\varphi\in\mathfrak{g}^*$ such that $\varphi  \wedge (d\varphi)^k\ne 0$, where $d\varphi(x,y)=-\varphi([x,y])$, for all $x,y\in\mathfrak{g}.$ The one-form $\varphi$  is called a \textit{contact form}. Ongoing, we make use of the following alternative characterization of contact Lie algebras.

\begin{lemma}[\textbf{\cite{seaweedA}}, Lemma 5]\label{lem:kernel}
If $\ind(\mathfrak{g})=1,$ then a one-form $\varphi\in\mathfrak{g}^*$ is contact if and only if $\varphi$ is regular and there exists an element $B\in\mathfrak{g}$ with $\ker(d\varphi)=\textup{span}\{B\}$ and $\varphi(B)\neq 0.$
\end{lemma}

\begin{remark}
For a contact Lie algebra $\mathfrak{g}$, a fixed contact form $\varphi\in\mathfrak{g}^*$, and $B\in\mathfrak{g}$ as in Lemma~\ref{lem:kernel}, the element $\frac{1}{|\varphi(B)|}B\in\mathfrak{g}$ is called the Reeb vector of $\varphi,$ and it is unique with respect to $\varphi.$
\end{remark}

In our investigation of contact, type-A Lie poset algebras, certain Frobenius Lie algebras play a crucial role. A Lie algebra $\mathfrak{g}$ for which $\ind(\mathfrak{g})=0$ is called \textit{Frobenius}, and a regular one-form $\varphi\in\mathfrak{g}^*$ is likewise called \textit{Frobenius}. In what follows, we refer to posets $\mathcal{P}$ for which $\mathfrak{g}_A(\mathcal{P})$ is Frobenius (resp., contact) as \textit{Frobenius} \textup(resp., \textit{contact}\textup) \textit{posets}. 

As it will be needed below (see Definition~\ref{def:toralpair}), we note that to each Frobenius Lie algebra $\mathfrak{g}$ one can associate a sequence of values, called the ``spectrum," as follows. Given a Frobenius Lie algebra $\mathfrak{g}$ and a Frobenius one-form $\varphi\in\mathfrak{g}^*$, the map $\mathfrak{g}\to\mathfrak{g}^*$ defined by $x\mapsto \varphi([x,-])$ is an isomorphism. The inverse image of $\varphi$ under this isomorphism, denoted $\widehat{\varphi}$, is called a \textit{principal element} of $\mathfrak{g}$. In (\textbf{\cite{Ooms}}, 1980), Ooms shows that the eigenvalues (and multiplicities) of $ad(\widehat{\varphi})=[\widehat{\varphi},-]:\mathfrak{g}\to\mathfrak{g}$ do not depend on the choice of principal element $\widehat{\varphi}$ -- see also (\textbf{\cite{Prin}}, 2009). It follows that the spectrum of $ad(\widehat{\varphi})$ is an invariant of $\mathfrak{g}$, which we call the \textit{spectrum} of $\mathfrak{g}.$ When the spectrum of $\mathfrak{g}$ consists of an equal number of 0's and 1's, then $\mathfrak{g}$ is said to have a \textit{binary spectrum}.

Finally, since certain contact and Frobenius one-forms are heavily used in the sections to follow, we end this section by setting definitions and notations related to one-forms on (type-A) Lie poset algebras. Given a finite poset $\mathcal{P}$ and $B=\sum b_{p,q}E_{p,q}\in\mathfrak{g}_A(\mathcal{P})$, define $E_{p,q}^*\in(\mathfrak{g}_A(\mathcal{P}))^*$, for $p,q\in\mathcal{P}$ satisfying $p\preceq q$, by $E^*_{p,q}(B)=b_{p,q}$. From any set $S$ consisting of ordered pairs $(p,q)$ of elements $p,q\in\mathcal{P}$ satisfying $p\prec q$, i.e., $S\subset Rel(\mathcal{P})$, one can construct a one-form $\varphi_S=\sum_{(p,q)\in S}E_{p,q}^*\in(\mathfrak{g}_A(\mathcal{P}))^*$ as well as a directed subgraph $\Gamma_{\varphi_S}(\mathcal{P})$ of the comparability graph of $\mathcal{P}$. In \textbf{\cite{Prin}}, Gerstenhaber and Giaquinto refer to such a one-form as \textit{small} if $\Gamma_{\varphi_S}(\mathcal{P})$ is a spanning subtree of the comparability graph of $\mathcal{P}$. Note that if $\varphi_S\in(\mathfrak{g}_A(\mathcal{P}))^*$ is a small one-form, then $\Gamma_{\varphi_S}(\mathcal{P})$ naturally partitions the elements of $\mathcal{P}$ into the following disjoint subsets: 
\begin{itemize}
    \item $U_{{\varphi}_S}(\mathcal{P})$ consisting of all sinks in $\Gamma_{{\varphi}_S}(\mathcal{P})$, 
    \item $D_{{\varphi}_S}(\mathcal{P})$ consisting of all sources  in $\Gamma_{{\varphi}_S}(\mathcal{P})$, and 
    \item $O_{{\varphi}_S}(\mathcal{P})$ consisting of those vertices which are neither sinks nor sources in $\Gamma_{{\varphi}_S}(\mathcal{P})$.
\end{itemize}

\begin{remark}\label{rem:funct}
Throughout this article, all one-forms $\varphi\in (\mathfrak{g}_A(\mathcal{P}))^*$ can also be viewed as elements of $(\mathfrak{g}(\mathcal{P}))^*$ -- and we have occasion to consider both circumstances. Consequently, we set the following notational convention: we denote the kernel of $d\varphi$, for $\varphi\in (\mathfrak{g}(\mathcal{P}))^*$, by $\ker(d\varphi)$ and we denote the kernel of $d\varphi$, for $\varphi\in (\mathfrak{g}_A(\mathcal{P}))^*$, by $\ker_A(d\varphi)$. Clearly, $\ker_A(d\varphi)=\ker(d\varphi)\cap\mathfrak{g}_A(\mathcal{P})$.
\end{remark}

\section{Toral-pairs}\label{sec:ntp}

 In this section, we define toral-pairs as introduced in \textbf{\cite{Binary}}. These consist of a Frobenius poset $\mathcal{P}$ along with a corresponding Frobenius one-form on $\mathfrak{g}_A(\mathcal{P}).$ The posets of toral-pairs are the ``building blocks" of toral posets under the original definition given in \textbf{\cite{Binary}} and form a subcollection of the ``building blocks" used in the extended definition developed in Section~\ref{sec:contoral} (see Definition~\ref{def:toral}).

\begin{definition}\label{def:toralpair}
Given a Frobenius poset $\mathcal{P}$ and a corresponding Frobenius one-form $\varphi=\varphi_S\in(\mathfrak{g}_A(\mathcal{P}))^*$, for $S\subseteq Rel(\mathcal{P})$, we call $(\mathcal{P},\varphi)$ a toral-pair if $\mathcal{P}$ satisfies
\begin{trivlist}
    \item[\hskip 0.3in \textup{\textbf{(P1)}}] $|Ext(\mathcal{P})|=2$ or 3, and
    \item[\hskip 0.3in \textup{\textbf{(P2)}}] $\mathfrak{g}_A(\mathcal{P})$ has a binary spectrum,
\end{trivlist}
and $\varphi$ satisfies
\begin{trivlist}
    \item[\hskip 0.3in \textup{\textbf{(F1)}}] $\varphi$ is small,
     \item[\hskip 0.3in \textup{\textbf{(F2)}}] $U_{\varphi}(\mathcal{P})$ is a filter of $\mathcal{P}$, $D_{\varphi}(\mathcal{P})$ is an ideal of $\mathcal{P}$, and $O_{\varphi}(\mathcal{P})=\emptyset$,
    \item[\hskip 0.3in \textup{\textbf{(F3)}}] $\Gamma_{\varphi}$ contains all edges between elements of $Ext(\mathcal{P})$, and
    \item[\hskip 0.3in \textup{\textbf{(F4)}}] $B\in\ker(d\varphi)$ satisfies $E^*_{p,p}(B)=E^*_{q,q}(B)$, for all $p,q\in\mathcal{P}$, and $E^*_{p,q}(B)=0$, for all $p,q\in\mathcal{P}$ \\* $~~~~~~~~~~~~$ satisfying $p\prec q$.
\end{trivlist}
\end{definition}

\begin{remark}
In \textup{\textbf{\cite{Binary}}}, the authors include a third property of the poset $\mathcal{P}$, requiring that the poset's corresponding order complex $\Sigma(\mathcal{P})$ be contractible. This property is redundant, though, since requiring $\mathcal{P}$ to satisfy $|Ext(\mathcal{P})|=2$ or 3 forces $\mathcal{P}$ to have a unique minimal or maximal element. Thus, for such a poset, $\Sigma(\mathcal{P})$ is a cone, so $\Sigma(\mathcal{P})$ is contractible.
\end{remark}

\begin{Ex}\label{ex:toralpairs}
The posets illustrated in Figure~\ref{fig:bb} can be paired with an appropriate one-form to yield a toral-pair. See \textup{\textbf{\cite{Binary}}}.
\bigskip

\begin{center}
\fbox{\begin{minipage}{45em} 
\begin{figure}[H]
$$\begin{tikzpicture}[scale=0.7]
\node (v1) at (0,-0.5) [circle, draw = black, fill = black, inner sep = 0.5mm, label=left:{$1$}] {};
\node (v2) at (0,0.5) [circle, draw = black, fill = black, inner sep = 0.5mm, label=left:{$2$}] {};
\draw (v1) -- (v2);
\end{tikzpicture}\begin{tikzpicture}[scale=0.65]
	\node (1) at (0, 0) [circle, draw = black, fill = black, inner sep = 0.5mm, label=left:{$1$}]{};
	\node (2) at (0, 1)[circle, draw = black, fill = black, inner sep = 0.5mm, label=left:{$2$}] {};
	\node (3) at (-0.5, 2) [circle, draw = black, fill = black, inner sep = 0.5mm, label=left:{$3$}] {};
	\node (4) at (0.5, 2) [circle, draw = black, fill = black, inner sep = 0.5mm, label=right:{$4$}] {};
    \draw (1)--(2);
    \draw (2)--(3);
    \draw (2)--(4);
\end{tikzpicture}
\begin{tikzpicture}[scale=0.65]
\node (v1) at (1.5,0) [circle, draw = black, fill = black, inner sep = 0.5mm, label=left:{$1$}] {};
\node (v4) at (2.5,0) [circle, draw = black, fill = black, inner sep = 0.5mm, label=right:{$2$}] {};
\node (v2) at (2,1) [circle, draw = black, fill = black, inner sep = 0.5mm, label=right:{$3$}] {};
\node (v3) at (2,2) [circle, draw = black, fill = black, inner sep = 0.5mm, label=right:{$4$}] {};
\draw (v1) -- (v2) -- (v3);
\draw (v2) -- (v4);
\end{tikzpicture}\begin{tikzpicture}[scale=0.65]
\node (v4) at (-2,3) [circle, draw = black, fill = black, inner sep = 0.5mm, label=left:{$5$}] {};
\node (v6) at (-0.5,3) [circle, draw = black, fill = black, inner sep = 0.5mm, label=right:{$6$}] {};
\node (v5) at (-0.5,2) [circle, draw = black, fill = black, inner sep = 0.5mm, label=right:{$4$}] {};
\node (v3) at (-2,2) [circle, draw = black, fill = black, inner sep = 0.5mm, label=left:{$3$}] {};
\node (v2) at (-1.25,1) [circle, draw = black, fill = black, inner sep = 0.5mm, label=left:{$2$}] {};
\node (v1) at (-1.25,0) [circle, draw = black, fill = black, inner sep = 0.5mm, label=left:{$1$}] {};
\draw (v1) -- (v2) -- (v3) -- (v4);
\draw (v2) -- (v5) -- (v6);
\end{tikzpicture}
\begin{tikzpicture}[scale=0.65]
\node (v7) at (1,0)  [circle, draw = black, fill = black, inner sep = 0.5mm, label=left:{$1$}] {};
\node (v11) at (2.5,0)  [circle, draw = black, fill = black, inner sep = 0.5mm, label=right:{$2$}] {};
\node (v12) at (2.5,1)  [circle, draw = black, fill = black, inner sep = 0.5mm, label=right:{$4$}] {};
\node (v8) at (1,1)  [circle, draw = black, fill = black, inner sep = 0.5mm, label=left:{$3$}] {};
\node (v9) at (1.75,2)  [circle, draw = black, fill = black, inner sep = 0.5mm, label=right:{$5$}] {};
\node (v10) at (1.75,3)  [circle, draw = black, fill = black, inner sep = 0.5mm, label=right:{$6$}] {};
\draw (v7) -- (v8) -- (v9) -- (v10);
\draw (v11) -- (v12) -- (v9);
\end{tikzpicture}$$\medskip $$\begin{tikzpicture}[scale=0.8]
\node [circle, draw = black, fill = black, inner sep = 0.5mm, label=right:{$1$}] (v13) at (9,0.5) {};
\node [circle, draw = black, fill = black, inner sep = 0.5mm, label=right:{$\lfloor\frac{n}{2}\rfloor-1$}] (v15) at (9,1.25) {};
\node [circle, draw = black, fill = black, inner sep = 0.5mm, label=right:{$\lfloor\frac{n}{2}\rfloor$}] (v19) at (9,2) {};
\node [circle, draw = black, fill = black, inner sep = 0.5mm, label=right:{$\lfloor\frac{n}{2}\rfloor+1$}] (v16) at (9,2.75) {};
\node [circle, draw = black, fill = black, inner sep = 0.5mm, label=right:{$n-1$}] (v18) at (9,3.5) {};
\draw (v15) -- (v16) -- cycle;
\node [circle, draw = black, fill = black, inner sep = 0.5mm, label=left:{$n$}] (v20) at (8.5,2.75) {};
\draw (v19) -- (v20);
\node at (9,3.25) {$\vdots$};
\node at (9,1) {$\vdots$};
\end{tikzpicture}\quad\begin{tikzpicture}[scale=0.8]
\node [circle, draw = black, fill = black, inner sep = 0.5mm, label=right:{$1$}] (v13) at (11,0.5) {};
\node [circle, draw = black, fill = black, inner sep = 0.5mm, label=right:{$\lfloor\frac{n-1}{2}\rfloor$}] (v15) at (11,1.25) {};
\node [circle, draw = black, fill = black, inner sep = 0.5mm, label=right:{$\lfloor\frac{n-1}{2}\rfloor+2$}] (v19) at (11,2) {};
\node [circle, draw = black, fill = black, inner sep = 0.5mm, label=right:{$\lfloor\frac{n-1}{2}\rfloor+3$}] (v16) at (11,2.75) {};
\node [circle, draw = black, fill = black, inner sep = 0.5mm, label=right:{$n$}] (v18) at (11,3.5) {};
\draw (v15) -- (v16) -- cycle;
\node [circle, draw = black, fill = black, inner sep = 0.5mm, label=left:{$\lfloor\frac{n-1}{2}\rfloor+1$}] (v20) at (10.5,1.25) {};
\draw (v19) -- (v20);
\node at (11,3.25) {$\vdots$};
\node at (11,1) {$\vdots$};
\end{tikzpicture}\quad
\begin{tikzpicture}
	\node (1) at (0, 0) [circle, draw = black, fill = black, inner sep = 0.5mm, label=left:{$1$}]{};
	\node (2) at (-0.5, 0.5)[circle, draw = black, fill = black, inner sep = 0.5mm, label=left:{$2$}] {};
	\node (3) at (0.5, 0.5) [circle, draw = black, fill = black, inner sep = 0.5mm, label=right:{$3$}] {};
    \node (4) at (-0.5, 1) [circle, draw = black, fill = black, inner sep = 0.5mm, label=left:{$4$}] {};
    \node (5) at (0.5, 1) [circle, draw = black, fill = black, inner sep = 0.5mm, label=right:{$5$}] {};
    \node (6) at (0, 1.5) {$\vdots$};
    \node (7) at (-0.5, 2)[circle, draw = black, fill = black, inner sep = 0.5mm, label=left:{$2n-2$}] {};
	\node (8) at (0.5, 2) [circle, draw = black, fill = black, inner sep = 0.5mm, label=right:{$2n-1$}] {};
    \node (9) at (-0.5, 2.5) [circle, draw = black, fill = black, inner sep = 0.5mm, label=left:{$2n$}] {};
    \node (10) at (0.5, 2.5) [circle, draw = black, fill = black, inner sep = 0.5mm, label=right:{$2n+1$}] {};
    \draw (4)--(3)--(1)--(2)--(5);
    \draw (5)--(3);
    \draw (2)--(4);
    \draw (7)--(9)--(8);
    \draw (7)--(10)--(8);
\end{tikzpicture}\quad\begin{tikzpicture}
	\node (1) at (0, 2.5) [circle, draw = black, fill = black, inner sep = 0.5mm, label=left:{$2n+1$}]{};
	\node (2) at (-0.5, 2)[circle, draw = black, fill = black, inner sep = 0.5mm, label=left:{$2n-1$}] {};
	\node (3) at (0.5, 2) [circle, draw = black, fill = black, inner sep = 0.5mm, label=right:{$2n$}] {};
    \node (4) at (-0.5, 1.5) [circle, draw = black, fill = black, inner sep = 0.5mm, label=left:{$2n-3$}] {};
    \node (5) at (0.5, 1.5) [circle, draw = black, fill = black, inner sep = 0.5mm, label=right:{$2n-2$}] {};
    \node (6) at (0, 1) {$\vdots$};
    \node (7) at (-0.5, 0.5)[circle, draw = black, fill = black, inner sep = 0.5mm, label=left:{$3$}] {};
	\node (8) at (0.5, 0.5) [circle, draw = black, fill = black, inner sep = 0.5mm, label=right:{$4$}] {};
    \node (9) at (-0.5, 0) [circle, draw = black, fill = black, inner sep = 0.5mm, label=left:{$1$}] {};
    \node (10) at (0.5, 0) [circle, draw = black, fill = black, inner sep = 0.5mm, label=right:{$2$}] {};
    \draw (4)--(3)--(1)--(2)--(5);
    \draw (5)--(3);
    \draw (2)--(4);
    \draw (7)--(9)--(8);
    \draw (7)--(10)--(8);
\end{tikzpicture}$$
\caption{Known posets of toral-pairs}\label{fig:bb}
\end{figure}\end{minipage}}
\end{center}
\end{Ex}
\bigskip

Before defining ``contact" toral-pairs in Section~\ref{sec:ctp}, we extend the list of examples of toral-pairs. In particular, Theorem~\ref{thm:ntp} below states that, in addition to the posets of Example~\ref{ex:toralpairs}, the posets illustrated in Figure~\ref{fig:bb2} can also be paired with an appropriate one-form to yield a toral-pair.
\bigskip

\begin{center}
\fbox{\begin{minipage}{45em} 
\begin{figure}[H]
$$\begin{tikzpicture}[scale=0.65]
	\node (1) at (0, 0) [circle, draw = black, fill = black, inner sep = 0.5mm, label=left:{$1$}]{};
	\node (2) at (-1, 1)[circle, draw = black, fill = black, inner sep = 0.5mm, label=left:{$2$}] {};
	\node (3) at (1, 1) [circle, draw = black, fill = black, inner sep = 0.5mm, label=right:{$3$}] {};
	\node (4) at (0, 2) [circle, draw = black, fill = black, inner sep = 0.5mm, label=left:{$4$}] {};
	\node (5) at (-1, 3)[circle, draw = black, fill = black, inner sep = 0.5mm, label=left:{$5$}] {};
	\node (6) at (1, 3) [circle, draw = black, fill = black, inner sep = 0.5mm, label=right:{$6$}] {};
	\node at (0,-0.75) {$\mathcal{P}_1$};
    \draw (1)--(2)--(4)--(5);
    \draw (1)--(3)--(4)--(6);
\end{tikzpicture}\quad\quad
\begin{tikzpicture}[scale=0.65]
	\node (1) at (0, 3) [circle, draw = black, fill = black, inner sep = 0.5mm, label=left:{$6$}]{};
	\node (2) at (-1, 2)[circle, draw = black, fill = black, inner sep = 0.5mm, label=left:{$4$}] {};
	\node (3) at (1, 2) [circle, draw = black, fill = black, inner sep = 0.5mm, label=right:{$5$}] {};
	\node (4) at (0, 1) [circle, draw = black, fill = black, inner sep = 0.5mm, label=left:{$3$}] {};
	\node (5) at (-1, 0)[circle, draw = black, fill = black, inner sep = 0.5mm, label=left:{$1$}] {};
	\node (6) at (1, 0) [circle, draw = black, fill = black, inner sep = 0.5mm, label=right:{$2$}] {};
	\node at (0,-0.75) {$\mathcal{P}^*_1$};
    \draw (1)--(2)--(4)--(5);
    \draw (1)--(3)--(4)--(6);
\end{tikzpicture}\quad\quad\begin{tikzpicture}[scale=0.65]
	\node (1) at (0, 0) [circle, draw = black, fill = black, inner sep = 0.5mm, label=left:{$1$}]{};
	\node (2) at (-1, 1)[circle, draw = black, fill = black, inner sep = 0.5mm, label=left:{$2$}] {};
	\node (3) at (1, 1) [circle, draw = black, fill = black, inner sep = 0.5mm, label=right:{$3$}] {};
	\node (4) at (0, 2) [circle, draw = black, fill = black, inner sep = 0.5mm, label=left:{$4$}] {};
	\node (6) at (2, 2)[circle, draw = black, fill = black, inner sep = 0.5mm, label=left:{$6$}] {};
	\node (5) at (0, 3) [circle, draw = black, fill = black, inner sep = 0.5mm, label=right:{$5$}] {};
	\node at (0,-0.75) {$\mathcal{P}_2$};
    \draw (1)--(2)--(4)--(5);
    \draw (1)--(3)--(4);
    \draw (3)--(6);
\end{tikzpicture}\quad\quad\begin{tikzpicture}[scale=0.65]
	\node (1) at (0, 3) [circle, draw = black, fill = black, inner sep = 0.5mm, label=left:{$6$}]{};
	\node (2) at (-1, 2)[circle, draw = black, fill = black, inner sep = 0.5mm, label=left:{$4$}] {};
	\node (3) at (1, 2) [circle, draw = black, fill = black, inner sep = 0.5mm, label=right:{$5$}] {};
	\node (4) at (0, 1) [circle, draw = black, fill = black, inner sep = 0.5mm, label=left:{$3$}] {};
	\node (6) at (2, 1)[circle, draw = black, fill = black, inner sep = 0.5mm, label=left:{$2$}] {};
	\node (5) at (0, 0) [circle, draw = black, fill = black, inner sep = 0.5mm, label=left:{$1$}] {};
	\node at (0,-0.75) {$\mathcal{P}^*_2$};
    \draw (1)--(2)--(4)--(5);
    \draw (1)--(3)--(4);
    \draw (3)--(6);
\end{tikzpicture}$$\medskip $$\begin{tikzpicture}[scale=0.65]
	\node (1) at (0, 0) [circle, draw = black, fill = black, inner sep = 0.5mm, label=left:{$1$}]{};
	\node (2) at (0, 1)[circle, draw = black, fill = black, inner sep = 0.5mm, label=left:{$2$}] {};
	\node (3) at (-1, 2) [circle, draw = black, fill = black, inner sep = 0.5mm, label=left:{$3$}] {};
	\node (4) at (0.75, 2) [circle, draw = black, fill = black, inner sep = 0.5mm, label=left:{$4$}] {};
	\node (6) at (1.75, 2)[circle, draw = black, fill = black, inner sep = 0.5mm, label=right:{$6$}] {};
	\node (5) at (0, 3) [circle, draw = black, fill = black, inner sep = 0.5mm, label=left:{$5$}] {};
	\node at (0,-0.75) {$\mathcal{P}_3$};
    \draw (1)--(2)--(3)--(5);
    \draw (6)--(2)--(4)--(5);
\end{tikzpicture}\quad\quad\begin{tikzpicture}[scale=0.65]
	\node (1) at (0, 3) [circle, draw = black, fill = black, inner sep = 0.5mm, label=left:{$6$}]{};
	\node (2) at (0, 2)[circle, draw = black, fill = black, inner sep = 0.5mm, label=left:{$5$}] {};
	\node (3) at (-1, 1) [circle, draw = black, fill = black, inner sep = 0.5mm, label=left:{$3$}] {};
	\node (4) at (0.75, 1) [circle, draw = black, fill = black, inner sep = 0.5mm, label=left:{$4$}] {};
	\node (6) at (1.75, 1)[circle, draw = black, fill = black, inner sep = 0.5mm, label=right:{$2$}] {};
	\node (5) at (0, 0) [circle, draw = black, fill = black, inner sep = 0.5mm, label=left:{$1$}] {};
	\node at (0,-0.75) {$\mathcal{P}^*_3$};
    \draw (1)--(2)--(3)--(5);
    \draw (6)--(2)--(4)--(5);
\end{tikzpicture}\quad\quad\begin{tikzpicture}[scale=0.65]
	\node (1) at (0, 0) [circle, draw = black, fill = black, inner sep = 0.5mm, label=left:{$1$}]{};
	\node (2) at (-1, 1)[circle, draw = black, fill = black, inner sep = 0.5mm, label=left:{$2$}] {};
	\node (3) at (1, 1) [circle, draw = black, fill = black, inner sep = 0.5mm, label=right:{$3$}] {};
	\node (4) at (0, 2) [circle, draw = black, fill = black, inner sep = 0.5mm, label=below:{$4$}] {};
	\node (5) at (0, 3)[circle, draw = black, fill = black, inner sep = 0.5mm, label=left:{$5$}] {};
	\node (6) at (0, 4) [circle, draw = black, fill = black, inner sep = 0.5mm, label=left:{$6$}] {};
	\node at (0,-0.75) {$\mathcal{P}_4$};
    \draw (1)--(2)--(5)--(6)--(3)--(1);
    \draw (2)--(4)--(3);
\end{tikzpicture}\quad\quad\begin{tikzpicture}[scale=0.65]
	\node (1) at (0, 4) [circle, draw = black, fill = black, inner sep = 0.5mm, label=left:{$6$}]{};
	\node (2) at (-1, 3)[circle, draw = black, fill = black, inner sep = 0.5mm, label=left:{$4$}] {};
	\node (3) at (1, 3) [circle, draw = black, fill = black, inner sep = 0.5mm, label=right:{$5$}] {};
	\node (4) at (0, 2) [circle, draw = black, fill = black, inner sep = 0.5mm, label=above:{$3$}] {};
	\node (5) at (0, 1)[circle, draw = black, fill = black, inner sep = 0.5mm, label=left:{$2$}] {};
	\node (6) at (0, 0) [circle, draw = black, fill = black, inner sep = 0.5mm, label=left:{$1$}] {};
	\node at (0,-0.75) {$\mathcal{P}^*_4$};
    \draw (1)--(2)--(5)--(6)--(3)--(1);
    \draw (2)--(4)--(3);
\end{tikzpicture}$$
\caption{New posets of toral-pairs}\label{fig:bb2}
\end{figure} \end{minipage}}
\end{center}
\bigskip

\begin{theorem}\label{thm:ntp}
Each of the following pairs, consisting of a poset $\mathcal{P}$ and a one-form $\varphi_{\mathcal{P}}$, constitutes a toral-pair $(\mathcal{P},\varphi_{\mathcal{P}})$.
\begin{enumerate} [label=\textup(\roman*\textup)]
        \item $\mathcal{P}_1=\{1,2,3,4,5,6\}$ with $1\prec 2,3\prec 4\prec 5,6$, and $$\varphi_{\mathcal{P}_1}=E^*_{1,5}+E^*_{1,6}+E^*_{2,4}+E^*_{2,5}+E^*_{3,6},$$
        \item $\mathcal{P}_1^*=\{1,2,3,4,5,6\}$ with $1,2\prec 3\prec 4,5\prec 6$, and $$\varphi_{\mathcal{P}_1^*}=E^*_{1,6}+E^*_{2,6}+E^*_{1,4}+E^*_{3,4}+E^*_{2,5}.$$
        \item $\mathcal{P}_2=\{1,2,3,4,5,6\}$ with $1\prec 2,3\prec 4\prec 5$; $3\prec 6$, and $$\varphi_{\mathcal{P}_2}=E^*_{1,5}+E^*_{1,6}+E^*_{2,4}+E^*_{3,4}+E^*_{3,6},$$
        \item $\mathcal{P}_2^*=\{1,2,3,4,5,6\}$ with $1\prec 3\prec 4,5\prec 6$; $2\prec 5$, and $$\varphi_{\mathcal{P}_2^*}=E^*_{1,6}+E^*_{2,6}+E^*_{2,5}+E^*_{3,4}+E^*_{3,5}.$$
        \item $\mathcal{P}_3=\{1,2,3,4,5,6\}$ with $1\prec 2\prec 3, 4\prec 5$; $2\prec 6$, and $$\varphi_{\mathcal{P}_3}=E^*_{1,5}+E^*_{1,6}+E^*_{2,3}+E^*_{2,4}+E^*_{2,5},$$
        \item $\mathcal{P}_3^*=\{1,2,3,4,5,6\}$ with $1\prec 3, 4\prec 5\prec 6$; $2\prec 5$, and $$\varphi_{\mathcal{P}_3^*}=E^*_{1,6}+E^*_{2,6}+E^*_{1,5}+E^*_{3,5}+E^*_{4,5}.$$
        \item $\mathcal{P}_4=\{1,2,3,4,5,6\}$ with $1\prec 2, 3\prec 4$; $2\prec 5\prec 6$; $3\prec 6$, and $$\varphi_{\mathcal{P}_4}=E^*_{1,4}+E^*_{1,6}+E^*_{2,4}+E^*_{2,5}+E^*_{3,6},$$
        \item $\mathcal{P}_4^*=\{1,2,3,4,5,6\}$ with $1\prec 2\prec 4$; $1\prec 5$; $3\prec 4,5\prec 6$, and $$\varphi_{\mathcal{P}_4^*}=E^*_{1,5}+E^*_{1,6}+E^*_{2,4}+E^*_{3,4}+E^*_{3,6}.$$
\end{enumerate}
\end{theorem}
\begin{proof}
Appendix A.
\end{proof}

\section{Contact toral-pairs}\label{sec:ctp}

In this section, we introduce the contact analogues of toral-pairs, which we call ``contact toral-pairs." The posets of ``contact toral-pairs" will serve as the new ``building blocks" in our extended definition of toral poset given in Section~\ref{sec:contoral}.


\begin{definition}\label{def:contacttoral}
Given a contact poset $\mathcal{P}$ and a corresponding contact one-form $\varphi=\varphi_S\in(\mathfrak{g}_A(\mathcal{P}))^*,$ for $S\subseteq Rel(\mathcal{P})\cup\{(p,p)~|~p\in\mathcal{P}\}$, we call $(\mathcal{P},\varphi)$ a contact toral-pair if $\mathcal{P}$ satisfies

\begin{trivlist}
    \item[\hskip 0.3in \textup{\textbf{(CP1)}}] $\mathcal{P}$ is connected, and
    \item[\hskip 0.3in \textup{\textbf{(CP2)}}] $|Ext(\mathcal{P})|=2$ or $3,$
\end{trivlist}

\noindent and $\varphi$ satisfies

\begin{trivlist}
    \item[\hskip 0.3in \textup{\textbf{(CF1)}}] $E_{p,p}^*$ is a nonzero summand of $\varphi$ if and only if $p=1,$
    \item[\hskip 0.3in \textup{\textbf{(CF2)}}] $\varphi-E_{1,1}^*$ is small,
    \item[\hskip 0.3in \textup{\textbf{(CF3)}}]  $U_{\varphi-E_{1,1}^*}(\mathcal{P})$ is a filter of $\mathcal{P},$ $D_{\varphi-E_{1,1}^*}(\mathcal{P})$ is an order ideal of $\mathcal{P},$ and $O_{\varphi-E_{1,1}^*}(\mathcal{P})=\emptyset$, and
    \item[\hskip 0.3in \textup{\textbf{(CF4)}}] $\Gamma_{\varphi-E_{1,1}^*}$ contains all edges between elements of $Ext(\mathcal{P})$.
\end{trivlist}
\end{definition}

\begin{Ex}\label{ex:contacttoralpairs}
The posets illustrated in Figure~\ref{fig:ctp} can be paired with an appropriate one-form to yield a contact toral-pair \textup(see Theorems~\ref{thm:2chain},~\ref{thm:3chain}, and~\ref{thm:fork}\textup).
\bigskip

\begin{center}
\fbox{\begin{minipage}{45em} 
\begin{figure}[H]
$$\begin{tikzpicture}[scale=0.7]
\node (v1) at (0,-0.5) [circle, draw = black, fill = black, inner sep = 0.5mm, label=left:{$1$}] {};
\node (v2) at (0,0.5) [circle, draw = black, fill = black, inner sep = 0.5mm, label=left:{$2$}] {};
\node (v3) at (0,1.5) [circle, draw = black, fill = black, inner sep = 0.5mm, label=left:{$3$}] {};
\draw (v1) -- (v2)--(v3);
\node at (0,-1.5) {$\mathcal{P}_1$};
\end{tikzpicture}\quad\begin{tikzpicture}[scale=0.7]
\node (v1) at (0,-0.5) [circle, draw = black, fill = black, inner sep = 0.5mm, label=left:{$1$}] {};
\node (v2) at (0,0.5) [circle, draw = black, fill = black, inner sep = 0.5mm, label=left:{$2$}] {};
\node (v3) at (0,1.5) [circle, draw = black, fill = black, inner sep = 0.5mm, label=left:{$3$}] {};
\node (v4) at (0,2.5) [circle, draw = black, fill = black, inner sep = 0.5mm, label=left:{$4$}] {};
\draw (v1) -- (v2)--(v3)--(v4);
\node at (0,-1.5) {$\mathcal{P}_2$};
\end{tikzpicture}\quad\begin{tikzpicture}[scale=0.7]
\node (v1) at (0,-0.5) [circle, draw = black, fill = black, inner sep = 0.5mm, label=left:{$1$}] {};
\node (v2) at (0,0.5) [circle, draw = black, fill = black, inner sep = 0.5mm, label=left:{$2$}] {};
\node (v3) at (0,1.5) [circle, draw = black, fill = black, inner sep = 0.5mm, label=left:{$3$}] {};
\node (v4) at (-0.5,2.5) [circle, draw = black, fill = black, inner sep = 0.5mm, label=left:{$4$}] {};
\node (v5) at (0.5,2.5) [circle, draw = black, fill = black, inner sep = 0.5mm, label=right:{$5$}] {};
\draw (v1)--(v2)--(v3)--(v4);
\draw (v3)--(v5);
\node at (0,-1.5) {$\mathcal{P}_3$};
\end{tikzpicture}\quad\begin{tikzpicture}[scale=0.7]
\node (v1) at (0,2.5) [circle, draw = black, fill = black, inner sep = 0.5mm, label=left:{$5$}] {};
\node (v2) at (0,1.5) [circle, draw = black, fill = black, inner sep = 0.5mm, label=left:{$4$}] {};
\node (v3) at (0,0.5) [circle, draw = black, fill = black, inner sep = 0.5mm, label=left:{$3$}] {};
\node (v4) at (-0.5,-0.5) [circle, draw = black, fill = black, inner sep = 0.5mm, label=left:{$1$}] {};
\node (v5) at (0.5,-0.5) [circle, draw = black, fill = black, inner sep = 0.5mm, label=right:{$2$}] {};
\draw (v1)--(v2)--(v3)--(v4);
\draw (v3)--(v5);
\node at (0,-1.5) {$\mathcal{P}^*_3$};
\end{tikzpicture}$$

$$\begin{tikzpicture}[scale=0.8]
\node [circle, draw = black, fill = black, inner sep = 0.5mm, label=right:{$1$}] (v13) at (9,0.5) {};
\node [circle, draw = black, fill = black, inner sep = 0.5mm, label=right:{$\lfloor\frac{n}{2}\rfloor-1$}] (v15) at (9,1.25) {};
\node [circle, draw = black, fill = black, inner sep = 0.5mm, label=right:{$\lfloor\frac{n}{2}\rfloor$}] (v19) at (9,2) {};
\node [circle, draw = black, fill = black, inner sep = 0.5mm, label=right:{$\lfloor\frac{n}{2}\rfloor+1$}] (v16) at (9,2.75) {};
\node [circle, draw = black, fill = black, inner sep = 0.5mm, label=right:{$n-1$}] (v18) at (9,3.5) {};
\draw (v15) -- (v16) -- cycle;
\node [circle, draw = black, fill = black, inner sep = 0.5mm, label=left:{$n$}] (v20) at (8.5,3.5) {};
\draw (v16) -- (v20);
\node at (9,3.25) {$\vdots$};
\node at (9,1) {$\vdots$};
\node at (9,-0.5) {$\mathcal{P}_{4,n}$};
\end{tikzpicture}\quad\begin{tikzpicture}[scale=0.8]
\node [circle, draw = black, fill = black, inner sep = 0.5mm, label=right:{$1$}] (v13) at (11,0.5) {};
\node [circle, draw = black, fill = black, inner sep = 0.5mm, label=right:{$\lfloor\frac{n-1}{2}\rfloor$}] (v15) at (11,1.25) {};
\node [circle, draw = black, fill = black, inner sep = 0.5mm, label=right:{$\lfloor\frac{n-1}{2}\rfloor+2$}] (v19) at (11,2) {};
\node [circle, draw = black, fill = black, inner sep = 0.5mm, label=right:{$\lfloor\frac{n-1}{2}\rfloor+3$}] (v16) at (11,2.75) {};
\node [circle, draw = black, fill = black, inner sep = 0.5mm, label=right:{$n$}] (v18) at (11,3.5) {};
\draw (v15) -- (v16) -- cycle;
\node [circle, draw = black, fill = black, inner sep = 0.5mm, label=left:{$\lfloor\frac{n-1}{2}\rfloor+1$}] (v20) at (10.5,0.5) {};
\draw (v15) -- (v20);
\node at (11,3.25) {$\vdots$};
\node at (11,1) {$\vdots$};
\node at (11,-0.5) {$\mathcal{P}^*_{4,n}$};
\end{tikzpicture}\quad\begin{tikzpicture}[scale=0.8]
\node [circle, draw = black, fill = black, inner sep = 0.5mm, label=right:{$1$}] (v13) at (9,0.5) {};
\node [circle, draw = black, fill = black, inner sep = 0.5mm, label=right:{$\lfloor\frac{n}{2}\rfloor-1$}] (v15) at (9,1.25) {};
\node [circle, draw = black, fill = black, inner sep = 0.5mm, label=right:{$\lfloor\frac{n}{2}\rfloor$}] (v19) at (9,2) {};
\node [circle, draw = black, fill = black, inner sep = 0.5mm, label=right:{$\lfloor\frac{n}{2}\rfloor+1$}] (v16) at (9,2.75) {};
\node [circle, draw = black, fill = black, inner sep = 0.5mm, label=right:{$n-1$}] (v18) at (9,3.5) {};
\draw (v15) -- (v16) -- cycle;
\node [circle, draw = black, fill = black, inner sep = 0.5mm, label=left:{$n$}] (v20) at (8.5,2) {};
\draw (v15) -- (v20);
\node at (9,3.25) {$\vdots$};
\node at (9,1) {$\vdots$};
\node at (9,-0.5) {$\mathcal{P}_{5,n}$};
\end{tikzpicture}\quad\begin{tikzpicture}[scale=0.8]
\node [circle, draw = black, fill = black, inner sep = 0.5mm, label=right:{$1$}] (v13) at (11,0.5) {};
\node [circle, draw = black, fill = black, inner sep = 0.5mm, label=right:{$\lfloor\frac{n-1}{2}\rfloor$}] (v15) at (11,1.25) {};
\node [circle, draw = black, fill = black, inner sep = 0.5mm, label=right:{$\lfloor\frac{n-1}{2}\rfloor+2$}] (v19) at (11,2) {};
\node [circle, draw = black, fill = black, inner sep = 0.5mm, label=right:{$\lfloor\frac{n-1}{2}\rfloor+3$}] (v16) at (11,2.75) {};
\node [circle, draw = black, fill = black, inner sep = 0.5mm, label=right:{$n$}] (v18) at (11,3.5) {};
\draw (v15) -- (v16) -- cycle;
\node [circle, draw = black, fill = black, inner sep = 0.5mm, label=left:{$\lfloor\frac{n-1}{2}\rfloor+1$}] (v20) at (10.5,2) {};
\draw (v16) -- (v20);
\node at (11,3.25) {$\vdots$};
\node at (11,1) {$\vdots$};
\node at (11,-0.5) {$\mathcal{P}^*_{5,n}$};
\end{tikzpicture}$$
\caption{Posets of contact toral-pairs}\label{fig:ctp}
\end{figure}
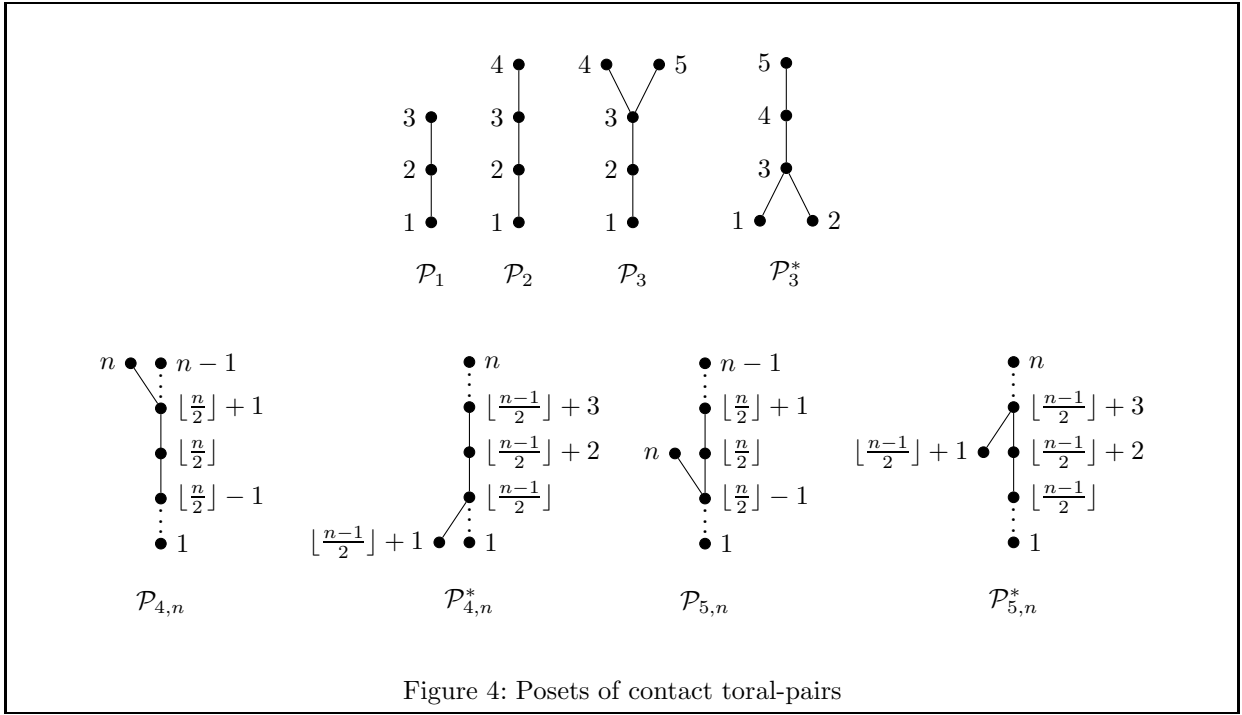 \end{minipage}}
\end{center}
\end{Ex}
\bigskip

\begin{remark}
By Theorem 17 of \textup{\textbf{\cite{ContactLiePoset}}}, all contact posets $\mathcal{P}$ of height zero or one are disconnected; hence, there are no contact toral-pairs $(\mathcal{P},\varphi)$ with $\mathcal{P}$ having height zero or one.
\end{remark}

\begin{theorem}\label{thm:2chain}
If $\mathcal{P}_1=\{1,2,3\}$ with $1\prec 2\prec 3$ and $\varphi_{\mathcal{P}_1}=E_{1,1}^*+E_{1,3}^*+E_{2,3}^*,$ then $(\mathcal{P}_1,\varphi_{\mathcal{P}_1})$ constitutes a contact toral-pair.
\end{theorem}
\begin{proof}
Let $\mathcal{P}=\mathcal{P}_1$ and $\varphi=\varphi_{\mathcal{P}_1}.$ Clearly, $\mathcal{P}$ is connected and $|Ext(\mathcal{P})|=2,$ so \textbf{(CP1)} and \textbf{(CP2)} are satisfied. Further, $E_{p,p}^*$ is a nonzero summand of $\varphi$ precisely when $p=1,$ $\varphi-E_{1,1}^*$ is clearly small, $U_{\varphi-E_{1,1}^*}(\mathcal{P})=\{3\}$ is a filter of $\mathcal{P},$ $D_{\varphi-E_{1,1}^*}(\mathcal{P})=\{1,2\}$ is an ideal of $\mathcal{P},$ $O_{\varphi-E_{1,1}^*}(\mathcal{P})=\emptyset,$ and $\Gamma_{\varphi-E_{1,1}^*}$ contains the edge $(1,3),$ which is the only edge between elements of $Ext(\mathcal{P}).$ Therefore, all that remains to show is that $\varphi$ is a contact form, and this fact follows from the proof of Theorem 28 in \textbf{\cite{ContactLiePoset}}.
\end{proof}

\begin{theorem}\label{thm:3chain}
If $\mathcal{P}_2=\{1,2,3,4\}$ with $1\prec 2\prec 3\prec 4$ and $\varphi_{\mathcal{P}_2}=E_{1,1}^*+E_{1,4}^*+E_{2,3}^*+E_{2,4}^*,$ then $(\mathcal{P}_2,\varphi_{\mathcal{P}_2})$ constitutes a contact toral-pair.
\end{theorem}
\begin{proof}
Let $\mathcal{P}=\mathcal{P}_2$ and $\varphi=\varphi_{\mathcal{P}_2}.$ Clearly, $\mathcal{P}$ is connected and $|Ext(\mathcal{P})|=2$, so \textbf{(CP1)} and \textbf{(CP2)} are satisfied. Further, $E_{p,p}^*$ is a nonzero summand of $\varphi$ precisely when $p=1$, $\varphi-E_{1,1}^*$ is clearly small, $U_{\varphi-E_{1,1}^*}(\mathcal{P})=\{3,4\}$ is a filter of $\mathcal{P},$ $D_{\varphi-E_{1,1}^*}(\mathcal{P})=\{1,2\}$ is an ideal of $\mathcal{P},$ $O_{\varphi-E_{1,1}^*}(\mathcal{P})=\emptyset,$ and $\Gamma_{\varphi-E_{1,1}^*}$ contains the edge $(1,4),$ which is the only edge between elements of $Ext(\mathcal{P}).$ Therefore, all that remains to show is that $\varphi$ is a contact form. We proceed by explicitly computing the kernel of $d\varphi$ and then applying Lemma~\ref{lem:kernel}.

Let $B=\sum k_{i,j}E_{i,j}\in\ker(d\varphi),$ and consider the following groups of conditions $B$ must satisfy:

\bigskip
\noindent
\textbf{Group 1:}
\begin{itemize}
    \item $\varphi([E_{1,1},B])=k_{1,4}=0,$
    \item $\varphi([E_{2,2},B])=k_{2,3}+k_{2,4}=0,$
    \item $\varphi([E_{3,3},B])=-k_{2,3}=0,$
    \item $\varphi([E_{4,4},B])=-k_{1,4}-k_{2,4}=0,$
    \item $\varphi([E_{1,2},B])=k_{2,4}=0,$
    \item $\varphi([E_{1,3},B])=k_{3,4}=0,$
    \item $\varphi([E_{3,4},B])=-k_{1,3}-k_{2,3}=0.$
\end{itemize}
\textbf{Group 2:}
\begin{itemize}
    \item $\varphi([E_{1,4},B])=k_{4,4}-k_{1,1}=0,$
    \item $\varphi([E_{2,3},B])=k_{3,3}-k_{2,2}=0,$
    \item $\varphi([E_{2,4},B])=-k_{1,2}+k_{4,4}-k_{2,2}=0.$
\end{itemize}
\noindent
The restrictions of Group 1 immediately imply that $B$ has the form $$B=k_{1,1}E_{1,1}+k_{2,2}E_{2,2}+k_{3,3}E_{3,3}+k_{4,4}E_{4,4}+k_{1,2}E_{1,2},$$ where the $k_{i,j}$ are potentially nonzero. The restrictions of Group 2 then yield the relations \begin{itemize}
    \item $k_{1,1}=k_{4,4}=k_{2,2}+k_{1,2}$ and
    \item $k_{2,2}=k_{3,3}.$
\end{itemize}
\noindent Combining this with the type-A trace condition $\sum_{i=1}^4k_{i,i}=0$ leads to the relations 
\begin{itemize}
    \item $k_{1,2}=2k_{1,1}$ and
    \item $k_{2,2}=-k_{1,1}$;
\end{itemize}
that is $$\ker(d\varphi)=\{k_{1,1}E_{1,1}-k_{1,1}E_{2,2}-k_{1,1}E_{3,3}+k_{1,1}E_{4,4}+2k_{1,1}E_{1,2}~|~k_{1,1}\in\mathbb{C}\}.$$ Then clearly, there exists a choice of $k_{1,1}$ for which $$k_{1,1}\varphi(E_{1,1}-E_{2,2}-E_{3,3}+E_{4,4}+2E_{1,2})=k_{1,1}\neq 0,$$ so the result follows from Lemma~\ref{lem:kernel}.
\end{proof}

\begin{theorem}\label{thm:fork}
Each of the following pairs, consisting of a poset $\mathcal{P}$ and a one-form $\varphi_{\mathcal{P}},$ constitutes a contact toral-pair $(\mathcal{P},\varphi_{\mathcal{P}}).$
\begin{enumerate}
    \item[\textup{(i)}] $\mathcal{P}_3=\{1,2,3,4,5\}$ with $1\prec 2\prec 3\prec 4,5,$ and $$\varphi_{\mathcal{P}_3}=E_{1,1}^*+E_{1,4}^*+E_{1,5}^*+E_{2,3}^*+E_{2,5}^*,$$
    \item[\textup{(ii)}] $\mathcal{P}_3^*=\{1,2,3,4,5\}$ with $1,2\prec 3\prec 4\prec 5$ and $$\varphi_{\mathcal{P}_3^*}=E_{1,1}^*+E_{1,4}^*+E_{1,5}^*+E_{2,5}^*+E_{3,4}^*.$$
\end{enumerate}
\end{theorem}
\begin{proof}
We prove (i), as (ii) follows via a symmetric argument. Let $\mathcal{P}=\mathcal{P}_3$ and $\varphi=\varphi_{\mathcal{P}_3}.$ Clearly, $\mathcal{P}$ is connected and $|Ext(\mathcal{P})|=3,$ so \textbf{(CP1)} and \textbf{(CP2)} are satisfied. Further, $E_{p,p}^*$ is a nonzero summand of $\varphi$ precisely when $p=1$, $\varphi-E_{1,1}^*$ is clearly small, $U_{\varphi-E_{1,1}^*}(\mathcal{P})=\{3,4,5\}$ is a filter of $\mathcal{P},$ $D_{\varphi-E_{1,1}^*}(\mathcal{P})=\{1,2\}$ is an ideal of $\mathcal{P},$ $O_{\varphi-E_{1,1}^*}(\mathcal{P})=\emptyset$, and $\Gamma_{\varphi-E_{1,1}^*}$ contains edges $(1,4)$ and $(1,5),$ which are all of the edges between elements of $Ext(\mathcal{P}).$ Therefore, all that remains to show is that $\varphi$ is a contact form. We proceed by explicitly computing the kernel of $d\varphi$ and then applying Lemma~\ref{lem:kernel}.

Let $B=\sum k_{i,j}E_{i,j}\in\ker(d\varphi),$ and consider the following groups of conditions $B$ must satisfy:

\bigskip
\noindent
\textbf{Group 1:}
\begin{itemize}
    \item $\varphi([E_{3,3},B])=-k_{2,3}=0,$
    \item $\varphi([E_{4,4},B])=-k_{1,4}=0,$
    \item $\varphi([E_{2,4},B])=-k_{1,2}=0,$
    \item $\varphi([E_{3,4},B])=-k_{1,3}=0.$
\end{itemize}
\textbf{Group 2:}
\begin{itemize}
    \item $\varphi([E_{1,1},B])=k_{1,4}+k_{1,5}=0,$
    \item $\varphi([E_{2,2},B])=k_{2,3}+k_{2,5}=0,$
    \item $\varphi([E_{5,5},B])=-k_{1,5}-k_{2,5}=0,$
    \item $\varphi([E_{1,2},B])=k_{2,4}+k_{2,5}=0,$
    \item $\varphi([E_{3,5},B])=-k_{1,3}-k_{2,3}=0.$
\end{itemize}
\textbf{Group 3:}
\begin{itemize}
    \item $\varphi([E_{1,3},B])=k_{3,4}+k_{3,5}=0,$
    \item $\varphi([E_{1,4},B])=k_{4,4}-k_{1,1}=0,$
    \item $\varphi([E_{1,5},B])=k_{5,5}-k_{1,1}=0,$
    \item $\varphi([E_{2,3},B])=k_{3,3}-k_{2,2}+k_{3,5}=0,$
    \item $\varphi([E_{2,5},B])=k_{5,5}-k_{2,2}-k_{1,2}=0.$
\end{itemize}
\noindent
The restrictions of Groups 1 and 2 immediately imply that $B$ has the form $$B=k_{1,1}E_{1,1}+k_{2,2}E_{2,2}+k_{3,3}E_{3,3}+k_{4,4}E_{4,4}+k_{5,5}E_{5,5}+k_{3,4}E_{3,4}+k_{3,5}E_{3,5},$$ where the $k_{i,j}$ are potentially nonzero. The restrictions of Group 3 then yield the relations $$k_{1,1}=k_{2,2}=k_{4,4}=k_{5,5}=k_{3,3}-k_{3,4}=k_{3,3}+k_{3,5}.$$ Combining this with the type-A trace condition $\sum_{i=1}^5 k_{i,i}=0$ leads to the relations 
\begin{itemize}
    \item $k_{3,3}=-4k_{1,1}$ and
    \item $k_{3,4}=-k_{3,5}=-5k_{1,1};$
\end{itemize}
that is,
\begin{equation*}
\ker(d\varphi)=\{k_{1,1}E_{1,1}+k_{1,1}E_{2,2}-4k_{1,1}E_{3,3}+k_{1,1}E_{4,4}+k_{1,1}E_{5,5}-5k_{1,1}E_{3,4}+5k_{1,1}E_{3,5}~|~k_{1,1}\in\mathbb{C}\}.
\end{equation*}
Then clearly, there exists a choice of $k_{1,1}$ for which $$k_{1,1}\varphi(E_{1,1}+E_{2,2}-4E_{3,3}+E_{4,4}+E_{5,5}-5E_{3,4}+5E_{3,5})=k_{1,1}\neq 0,$$ so the result follows from Lemma~\ref{lem:kernel}.
\end{proof}

\begin{theorem}\label{thm:ctpappendB}
Each of the following pairs, consisting of a poset $\mathcal{P}$ and a one-form $\varphi_{\mathcal{P}},$ constitutes a contact toral-pair $(\mathcal{P},\varphi_{\mathcal{P}}).$
\begin{enumerate}
    \item[\textup{(i)}] $\mathcal{P}_{4,n}=\{1,\dots,n\}$ with $1\preceq 2\preceq \dots\preceq n-1$ as well as $1\preceq 2\preceq \dots\preceq \lfloor\frac{n}{2}\rfloor+1\preceq n,$ and $$\varphi_{\mathcal{P}_{4,n}}=E_{1,1}^*+\sum_{i=1}^{\lfloor\frac{n-1}{2}\rfloor}E_{i,n-i}^*+\sum_{i=1}^{\lfloor\frac{n}{2}\rfloor}E_{i,n}^*$$
    \item[\textup{(ii)}] $\mathcal{P}_{4,n}^*=\{1,\dots,n\}$ with $2\preceq 3\preceq \dots\preceq n$ as well as $1\preceq\lceil\frac{n}{2}\rceil\preceq \lceil\frac{n}{2}\rceil+1\preceq\dots\preceq n,$ and $$\varphi_{\mathcal{P}_{4,n}^*}=E_{1,1}^*+\sum_{i=2}^{\lceil\frac{n}{2}\rceil}E_{i,n-i+2}^*+\sum_{i=\lceil\frac{n}{2}\rceil+1}^{n}E_{1,i}^*.$$
    \item[\textup{(iii)}] $\mathcal{P}_{5,n}=\{1,\dots,n\}$ with $1\preceq 2\preceq\dots\preceq n-1$ as well as $1\preceq 2\preceq\dots\preceq \lfloor\frac{n}{2}\rfloor-1\preceq n,$ and $$\varphi_{\mathcal{P}_{5,n}}=E_{1,1}^*+\sum_{i=1}^{\lfloor\frac{n-1}{2}\rfloor}E_{i,n-i}^*+\sum_{i=1}^{\lfloor\frac{n}{2}\rfloor-1}E_{i,n}^*+E_{\lfloor\frac{n}{2}\rfloor,n-1}^*.$$
    \item[\textup{(iv)}] $\mathcal{P}_{5,n}^*=\{1,\dots,n\}$ with $2\preceq 3\preceq\dots\preceq n$ as well as $1\preceq \lceil\frac{n}{2}\rceil+2\preceq\lceil\frac{n}{2}\rceil+3\preceq\dots\preceq n,$ and $$\varphi_{\mathcal{P}_{5,n}^*}=E_{1,1}^*+\sum_{i=2}^{\lceil\frac{n}{2}\rceil}E_{i,n-i+2}^*+\sum_{i=\lceil\frac{n}{2}\rceil+2}^{n}E_{1,i}^*+E^*_{2,\lceil\frac{n}{2}\rceil+1}.$$
\end{enumerate}
\end{theorem}
\begin{proof}
Appendix B.
\end{proof}

\section{Toral posets}\label{sec:contoral}
In this section, we extend the definition of toral poset given in \textbf{\cite{Binary}} so that posets of contact-toral pairs are included in the collection of ``building blocks," thus allowing toral posets to be contact (see Remark~\ref{rem:con}). First, we recall the ``gluing rules" used in \textbf{\cite{Binary}} to combine the posets of toral-pairs to yield toral posets.

Let $\mathcal{Q}$ be an arbitrary poset and $\mathcal{S}$ be a poset of a (contact) toral-pair. We define twelve ways of ``combining" the posets $\mathcal{S}$ and $\mathcal{Q}$ by identifying minimal (resp., maximal) elements of $\mathcal{S}$ with minimal (resp., maximal) elements of $\mathcal{Q}$. If $Ext(\mathcal{S})=2$, then $Ext(\mathcal{S})=\{a_1,c\}$ with $c\prec_{\mathcal{S}}a_1$; and if $Ext(\mathcal{S})=3$, then $Ext(\mathcal{S})=\{a_1,a_2,c\}$ with either $c\prec_{\mathcal{S}} a_1,a_2$ or $a_1,a_2\prec_{\mathcal{S}} c$. Further, assume $x,y,z\in Ext(\mathcal{Q})$. Since the construction rules are defined by identifying minimal elements and maximal elements of $\mathcal{S}$ and $\mathcal{Q}$, assume that if $c,a_1$, or $a_2$ are identified with elements of $\mathcal{Q}$, then those elements are $x,y,$ or $z$, respectively. To ease notation, let $\sim_{\mathcal{P}}$ denote that two elements of a poset $\mathcal{P}$ are related, and let $\nsim_{\mathcal{P}}$ denote that two elements are not related; that is, for $i,j\in\mathcal{P}$, $i\sim_{\mathcal{P}} j$ denotes that $i\preceq_{\mathcal{P}}j$ or $j\preceq_{\mathcal{P}}i$, and $i\nsim_{\mathcal{P}} j$ denotes that both $i\npreceq_{\mathcal{P}}j$ and $j\npreceq_{\mathcal{P}}i$. The following Table~\ref{tab:gluingrules} lists all possible ways (gluing rules) of identifying the elements $c,a_1,a_2\in\mathcal{S}$ with the elements $x,y,z\in\mathcal{Q}$. See Example~\ref{ex:toralposet} for an illustration of a subset of the gluing rules.

\begin{table}[H]
\centering
\begin{tabular}{c|c|c|c}
Gluing Rule & $c$                       & $a_1$                         & $a_2$                          \\ \hline
$A_1$ & $c\neq x$ & $a_1=y$         & $a_2\neq z$                        \\
$A_2$ & $c\neq x$ & $a_1\neq y$             & $a_2=z$                          \\
$B$ & $c\neq x$ & $a_1=y$       & $a_2=z$                                          \\
$C$ & $c=x$ & $a_1\neq y$                 & $a_2\neq z$                                          \\
$D_1$ & $c=x$                 & $a_1=y$, $y\sim x$ & $a_2\neq z$                           \\
$D_2$ & $c=x$                 & $a_1\neq y$  & $a_2= z$, $z\sim x$  \\
$E_1$ & $c=x$                 & $a_1= y$, $y\nsim x$  & $a_2\neq z$    \\
$E_2$ & $c=x$                 & $a_1\neq y$  & $a_2= z$, $z\nsim x$   \\
$F$ & $c=x$                 & $a_1= y$, $y\sim x$  & $a_2= z$, $z\sim x$ \\
$G_1$ & $c=x$                 & $a_1= y$, $y\sim x$  & $a_2= z$, $z\nsim x$   \\
$G_2$ & $c=x$                 & $a_1= y$, $y\nsim x$  & $a_2= z$, $z\sim x$   \\
$H$ & $c=x$                 & $a_1= y$, $y\nsim x$  & $a_2= z$, $z\nsim x$   \\
\end{tabular}
\caption{Gluing rules}\label{tab:gluingrules}
\end{table}

\begin{definition}\label{def:toral}
A poset $\mathcal{P}$ is called toral if there exists a sequence $\{(\mathcal{S}_i,\varphi_{\mathcal{S}_i})\}_{i=1}^n$ consisting of \textup(contact\textup) toral-pairs, and a sequence of posets $\mathcal{S}_1=\mathcal{Q}_1\subset\mathcal{Q}_2\subset\dots\subset\mathcal{Q}_n=\mathcal{P}$ such that $\mathcal{Q}_i$ is formed from $\mathcal{Q}_{i-1}$ and $\mathcal{S}_i$ by applying a rule from the set $\{A_1,A_2,B,C,D_1,D_2,E_1,E_2,F,G_1,G_2,H\}$, for $i=2,\dots,n.$ Such a sequence $\mathcal{S}_1=\mathcal{Q}_1\subset\mathcal{Q}_2\subset\dots\subset\mathcal{Q}_n=\mathcal{P}$ is called a construction sequence for $\mathcal{P}.$
\end{definition}

\begin{Ex}\label{ex:toralposet}
Let $\mathcal{P}$ be the toral poset constructed from the toral-pairs $\{(\mathcal{S}_i,F_i)\}_{i=1}^5$ with attendant construction sequence $\mathcal{S}_1=\mathcal{Q}_1\subset \mathcal{Q}_2\subset\mathcal{Q}_3\subset\mathcal{Q}_4\subset\mathcal{Q}_5=\mathcal{P}$, where
\begin{itemize}
    \item $\mathcal{S}_1=\mathcal{Q}_1$ is isomorphic to the poset on $\{1,2,3\}$ with $1\prec2\prec3$,
    \item $\mathcal{S}_2$ is isomorphic to the poset on $\{1,2,3,4\}$ with $1\prec2\prec3,4$,
    \item $\mathcal{S}_3$ is isomorphic to the poset on $\{1,2,3,4,5,6\}$ with $1\prec2\prec3,4$; $3\prec 5$; and $4\prec 6$,
    \item $\mathcal{S}_4$ is isomorphic to the poset on $\{1,2,3,4\}$ with $1\prec2\prec3\prec 4$,
    \item $\mathcal{S}_5$ is isomorphic to the poset on $\{1,2,3,4\}$ with $1,2\prec3\prec4$,
    \item $\mathcal{Q}_2$ is formed from $\mathcal{S}_1=\mathcal{Q}_1$ and $\mathcal{S}_2$ by applying rule $A_1$,
    \item $\mathcal{Q}_3$ is formed from $\mathcal{Q}_2$ and $\mathcal{S}_3$ by applying rule $C$,
    \item $\mathcal{Q}_4$ is formed from $\mathcal{Q}_3$ and $\mathcal{S}_4$ by applying rule $D_1$, and 
    \item $\mathcal{Q}_5=\mathcal{P}$ is formed from $\mathcal{Q}_4$ and $\mathcal{S}_5$ by applying rule $F$.
\end{itemize}
\noindent
In Figure~\ref{fig:toral} below we illustrate the posets $\mathcal{Q}_i$ of the construction sequence described above, identifying $a_1,a_2,$ and $c$ of $\mathcal{S}_i$, for $i=1,\dots,5$.
\bigskip

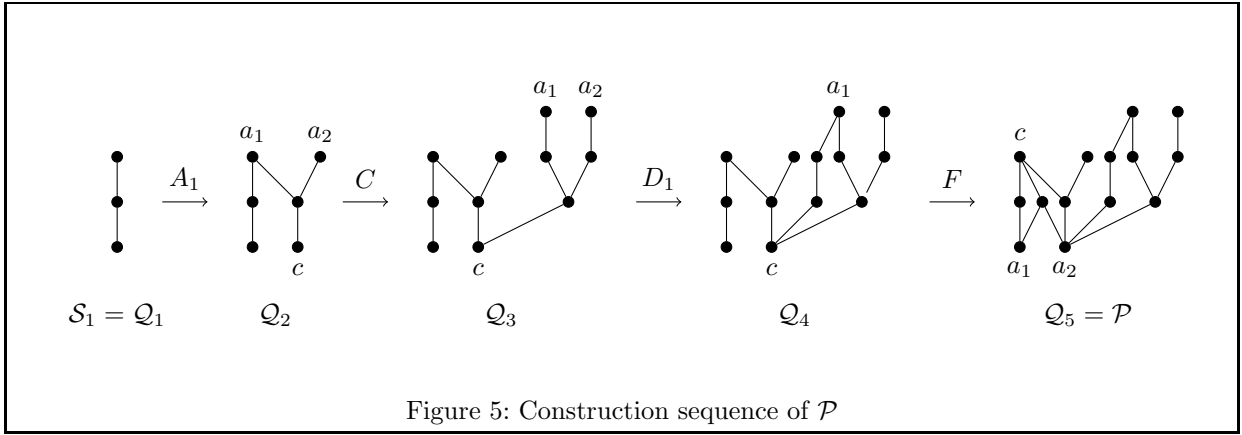
\begin{center}
\fbox{\begin{minipage}{45em}
\begin{figure}[H]
$$\begin{tikzpicture}[scale=0.6]
\node [circle, draw = black, fill = black, inner sep = 0.5mm] (v1) at (-2.5,0.5) {};
\node [circle, draw = black, fill = black, inner sep = 0.5mm] (v2) at (-2.5,1.5) {};
\node [circle, draw = black, fill = black, inner sep = 0.5mm] (v3) at (-2.5,2.5) {};
\draw (v1) -- (v2) -- (v3);
\draw[->] (-1.5,1.5) -- (-0.5,1.5);
\node [circle, draw = black, fill = black, inner sep = 0.5mm] (v8) at (0.5,2.5) {};
\node [circle, draw = black, fill = black, inner sep = 0.5mm] (v11) at (2,2.5) {};
\node [circle, draw = black, fill = black, inner sep = 0.5mm] (v10) at (1.5,1.5) {};
\node [circle, draw = black, fill = black, inner sep = 0.5mm] (v6) at (0.5,1.5) {};
\node [circle, draw = black, fill = black, inner sep = 0.5mm] (v5) at (0.5,0.5) {};
\node [circle, draw = black, fill = black, inner sep = 0.5mm] (v9) at (1.5,0.5) {};
\node at (-1,2) {$A_1$};
\draw (v5) -- (v6);
\draw (v6) -- (v8);
\draw (v9) -- (v10) -- (v8);
\draw (v10) -- (v11);
\draw [->] (2.5,1.5) -- (3.5,1.5);
\node at (3,2) {$C$};
\node [circle, draw = black, fill = black, inner sep = 0.5mm] (v15) at (4.5,2.5) {};
\node [circle, draw = black, fill = black, inner sep = 0.5mm] (v18) at (6,2.5) {};
\node [circle, draw = black, fill = black, inner sep = 0.5mm] (v13) at (4.5,1.5) {};
\node [circle, draw = black, fill = black, inner sep = 0.5mm] (v17) at (5.5,1.5) {};
\node [circle, draw = black, fill = black, inner sep = 0.5mm] (v12) at (4.5,0.5) {};
\node [circle, draw = black, fill = black, inner sep = 0.5mm] (v16) at (5.5,0.5) {};
\node [circle, draw = black, fill = black, inner sep = 0.5mm] (v14) at (7,3.5) {};
\node [circle, draw = black, fill = black, inner sep = 0.5mm] (v20) at (8,3.5) {};
\node [circle, draw = black, fill = black, inner sep = 0.5mm] (v7) at (7,2.5) {};
\node [circle, draw = black, fill = black, inner sep = 0.5mm] (v19) at (8,2.5) {};
\node [circle, draw = black, fill = black, inner sep = 0.5mm] (v4) at (7.5,1.5) {};
\draw (v16) -- (v4) -- (v7) -- (v14);
\draw (v4) -- (v19) -- (v20);
\draw (v12) -- (v13);
\draw (v13) -- (v15);
\draw (v16) -- (v17) -- (v15);
\draw (v17) -- (v18);
\draw[->] (9,1.5) -- (10,1.5);
\node at (9.5,2) {$D_1$};

\node at (1.5,0) {$c$};
\node at (0.5,3) {$a_1$};
\node at (2,3) {$a_2$};
\node at (5.5,0) {$c$};
\node at (7,4) {$a_1$};
\node at (8,4) {$a_2$};
\node at (12,0) {$c$};
\node at (13.5,4) {$a_1$};
\node at (17.5,3) {$c$};
\node at (17.5,0) {$a_1$};
\node at (18.5,0) {$a_2$};
\node at (-2.5,-1) {$\mathcal{S}_1=\mathcal{Q}_1$};
\node at (1,-1) {$\mathcal{Q}_2$};
\node at (6,-1) {$\mathcal{Q}_3$};
\node at (12.5,-1) {$\mathcal{Q}_4$};
\node at (19,-1) {$\mathcal{Q}_5=\mathcal{P}$};
\node [circle, draw = black, fill = black, inner sep = 0.5mm] (v21) at (11,0.5) {};
\node [circle, draw = black, fill = black, inner sep = 0.5mm] (v22) at (11,1.5) {};
\node [circle, draw = black, fill = black, inner sep = 0.5mm] (v23) at (11,2.5) {};
\node [circle, draw = black, fill = black, inner sep = 0.5mm] (v24) at (12,1.5) {};
\node [circle, draw = black, fill = black, inner sep = 0.5mm] (v25) at (12.5,2.5) {};
\node [circle, draw = black, fill = black, inner sep = 0.5mm] (v26) at (12,0.5) {};
\node [circle, draw = black, fill = black, inner sep = 0.5mm] at (14,1.5) {};
\node [circle, draw = black, fill = black, inner sep = 0.5mm] (v27) at (13.5,2.5) {};
\node [circle, draw = black, fill = black, inner sep = 0.5mm] (v28) at (13.5,3.5) {};
\node [circle, draw = black, fill = black, inner sep = 0.5mm] (v31) at (14.5,3.5) {};
\node [circle, draw = black, fill = black, inner sep = 0.5mm] (v30) at (14.5,2.5) {};
\draw (v21) -- (v22) -- (v23) -- (v24) -- (v25);
\draw (v24) -- (v26) -- (14,1.5) node (v29) {} -- (v27) -- (v28);
\draw (v29) -- (v30) -- (v31);
\node [circle, draw = black, fill = black, inner sep = 0.5mm] (v32) at (13,1.5) {};
\node [circle, draw = black, fill = black, inner sep = 0.5mm] (v33) at (13,2.5) {};
\draw (v26) -- (v32) -- (v33) -- (v28);
\node (v36) at (17.5,0.5) [circle, draw = black, fill = black, inner sep = 0.5mm] {};
\node (v37) at (17.5,1.5) [circle, draw = black, fill = black, inner sep = 0.5mm] {};
\node (v38) at (17.5,2.5) [circle, draw = black, fill = black, inner sep = 0.5mm] {};
\node (v39) at (18.5,1.5) [circle, draw = black, fill = black, inner sep = 0.5mm] {};
\node (v48) at (19,2.5) [circle, draw = black, fill = black, inner sep = 0.5mm] {};
\node (v40) at (18.5,0.5) [circle, draw = black, fill = black, inner sep = 0.5mm] {};
\node (v41) at (19.5,1.5) [circle, draw = black, fill = black, inner sep = 0.5mm] {};
\node (v45) at (20.5,1.5) [circle, draw = black, fill = black, inner sep = 0.5mm] {};
\node (v42) at (19.5,2.5) [circle, draw = black, fill = black, inner sep = 0.5mm] {};
\node (v44) at (20,2.5) [circle, draw = black, fill = black, inner sep = 0.5mm] {};
\node (v46) at (21,2.5) [circle, draw = black, fill = black, inner sep = 0.5mm] {};
\node (v43) at (20,3.5) [circle, draw = black, fill = black, inner sep = 0.5mm] {};
\node (v47) at (21,3.5) [circle, draw = black, fill = black, inner sep = 0.5mm] {};
\draw[->] (15.5,1.5) -- (16.5,1.5);
\node at (16,2) {$F$};
\draw (v36) -- (v37) -- (v38) -- (v39) -- (v40) -- (v41) -- (v42) -- (v43) -- (v44) -- (v45) -- (v46) -- (v47);
\draw (v45) -- (v40);
\draw (v39) -- (v48);
\node [circle, draw = black, fill = black, inner sep = 0.5mm] (v49) at (18,1.5) {};
\draw (v36) -- (v49) -- (v40);
\draw (v49) -- (v38);
\end{tikzpicture}$$
\caption{Construction sequence of $\mathcal{P}$}\label{fig:toral}
\end{figure}
\end{minipage}}
\end{center}
\end{Ex}
\bigskip

\begin{remark}\label{rem:con}
To retrieve the original definition of toral poset given in \textup{\textbf{\cite{Binary}}}, one merely has to omit mention of contact toral-pairs from Definition~\ref{def:toral}. Combining Theorems~\ref{lem:table} and~\ref{lem:cycle} below, one can deduce that toral posets, as they were originally defined, could not be contact. However, since the posets of contact toral-pairs are toral under our extended definition, it follows immediately that toral posets can now be contact.
\end{remark}

\begin{remark}\label{rem:h2}
Our extended definition of toral poset is motivated by results of \textup{\textbf{\cite{ContactLiePoset}}}, where the authors find that if a connected poset $\mathcal{P}$ of height at most two is contact, then it is toral. In particular, the authors show that a connected poset $\mathcal{P}$ of height at most two is contact if and only if $\mathcal{P}$ is toral and constructed from the \textup(contact\textup) toral-pairs $\{(\mathcal{S}_i,F_{\mathcal{S}_i})\}_{i=1}^n$ with construction sequence $\mathcal{S}_1=\mathcal{Q}_1\subset\mathcal{Q}_2\subset\cdots\subset\mathcal{Q}_n=\mathcal{P}$, where
\begin{enumerate}
    \item $(\mathcal{S}_i, F_{S_i})$ is the contact toral-pair of Theorem~\ref{thm:2chain} for exactly one value of $i\in\{1,\hdots,n\}$,
    \item the remaining pairs are toral-pairs, and
    \item $\mathcal{Q}_i$ is formed from $\mathcal{Q}_{i-1}$ and $\mathcal{S}_i$ by applying rules from the set $\{A_1,A_2,C,D_1,D_2,F\}$, for $i=2,\hdots,n$.
\end{enumerate}
The main result of this paper extends the above characterization to include posets of arbitrary height.
\end{remark}

Having extended the notion of toral poset, it is worth documenting how certain results of \textbf{\cite{Binary}} translate in this more general setting. To do so, we require the following theorem, which establishes a recursive relationship between the index of a type-A Lie poset algebra associated to a toral poset and the index of the type-A Lie poset algebras associated to each of the toral poset's building blocks.

\begin{theorem}\label{lem:table}
Let $\mathcal{Q}$ be a toral poset, $\mathcal{S}$ be the poset of a \textup(contact\textup) toral-pair, and $\mathcal{P}$ be a toral poset formed from $\mathcal{Q}$ and $\mathcal{S}$ by applying a rule from the set $\{A_1,A_2,B,C,D_1,D_2,E_1,E_2,F,G_1,G_2,H\}$. Then for each possible gluing rule applied in forming $\mathcal{P}$ from $\mathcal{Q}$ and $\mathcal{S}$, the following table gives the value $\ind(\mathfrak{g}_A(\mathcal{P}))-\ind(\mathfrak{g}_A(\mathcal{Q}))$ in terms of $\ind(\mathfrak{g}_A(\mathcal{S}))$.

\begin{table}[H]
\centering
\begin{tabular}{c|c}
Gluing Rule                     & $\ind(\mathfrak{g}_A(\mathcal{P}))-\ind(\mathfrak{g}_A(\mathcal{Q}))$ \\ \hline
$A_1$ &  $\ind(\mathfrak{g}_A(\mathcal{S}))$                         \\
$A_2$ &  $\ind(\mathfrak{g}_A(\mathcal{S}))$                         \\
$B$ &  $\ind(\mathfrak{g}_A(\mathcal{S}))+1$                         \\
$C$ &  $\ind(\mathfrak{g}_A(\mathcal{S}))$                         \\
$D_1$ & $\ind(\mathfrak{g}_A(\mathcal{S}))$                         \\
$D_2$ &  $\ind(\mathfrak{g}_A(\mathcal{S}))$                         \\
$E_1$ &  $\ind(\mathfrak{g}_A(\mathcal{S}))+1$                         \\
$E_2$ &  $\ind(\mathfrak{g}_A(\mathcal{S}))+1$                         \\
$F$ &  $\ind(\mathfrak{g}_A(\mathcal{S}))$                         \\
$G_1$ &  $\ind(\mathfrak{g}_A(\mathcal{S}))+1$                         \\
$G_2$ &  $\ind(\mathfrak{g}_A(\mathcal{S}))+1$                         \\
$H$ &  $\ind(\mathfrak{g}_A(\mathcal{S}))+2$                         \\
\end{tabular}
\caption{Gluing rules and index}\label{tab:gluingrulesind}
\end{table}
\end{theorem}
\begin{proof}
Note that for posets $\mathcal{S}$ of (contact) toral pairs, one has $|Rel_E(\mathcal{S})|-|Ext(\mathcal{S})|+1=0$. Combining Remark 6 and Theorem 20 of \textbf{\cite{SeriesA}}, the result follows.
\end{proof}

\begin{remark}
Note that Theorem~\ref{lem:table} is an extension of Theorem 22 of \textup{\textbf{\cite{ContactLiePoset}}}.
\end{remark}

\noindent
Considering Theorem~\ref{lem:table} above, we rephrase Theorems 6, 7, and 8 of \textbf{\cite{Binary}} below, for which the proofs remain unchanged.

\begin{theorem}\label{thm:indform}
If $\mathcal{P}$ is a toral poset and $|C(\mathcal{P})|$ denotes the number of contact toral-pairs in a construction sequence of $\mathcal{P}$, then $\ind(\mathfrak{g}_A(\mathcal{P}))=|Rel_E(\mathcal{P})|-|Ext(\mathcal{P})|+|C(\mathcal{P})|+1$.
\end{theorem}

\begin{theorem}\label{thm:wedge}
If $\mathcal{P}$ is a toral poset, then $\Sigma(\mathcal{P})$ is homotopic to a wedge sum of $|Rel_E(\mathcal{P})|-|Ext(\mathcal{P})|+1$ one-spheres.
\end{theorem}

\begin{theorem}\label{thm:Frob}
Let $\mathcal{P}$ be a toral poset constructed from the \textup(contact\textup) toral-pairs $\{(\mathcal{S}_i,\varphi_{\mathcal{S}_i})\}_{i=1}^n$ with construction sequence $\mathcal{S}_1=\mathcal{Q}_1\subset\mathcal{Q}_2\subset\cdots\subset\mathcal{Q}_n=\mathcal{P}$. Then, $\mathcal{P}$ is Frobenius if and only if
\begin{enumerate}
    \item[\textup1.] $(\mathcal{S}_i,\varphi_{\mathcal{S}_i})$ is a toral-pair, for $i=1,\hdots,n$, and
    \item[\textup2.] $\mathcal{Q}_i$ is formed from $\mathcal{Q}_{i-1}$ and $\mathcal{S}_i$ by applying rules from the set $\{A_1,A_2,C,D_1,D_2,F\}$, for $i=2,\hdots,n$.
\end{enumerate}
\end{theorem}

In \textbf{\cite{Binary}}, it is conjectured that all Frobenius posets are toral, i.e., are of the form described in Theorem~\ref{thm:Frob}. Since a type-A Lie poset algebra corresponding to a disjoint sum of two posets has nontrivial center,\footnote{Given two posets $\mathcal{P}$ and $\mathcal{Q}$, the \textit{disjoint sum} of $\mathcal{P}$ and $\mathcal{Q}$ is the poset $\mathcal{P}+\mathcal{Q}$ on the disjoint sum of $\mathcal{P}$ and $\mathcal{Q}$, where $s\preceq_{\mathcal{P}+\mathcal{Q}} t$ if either 
\begin{itemize}
    \item $s,t\in \mathcal{P}$ and $s\preceq_{\mathcal{P}} t$, or
    \item $s,t\in \mathcal{Q}$ and $s\preceq_{\mathcal{Q}} t$.
\end{itemize} The element $$|\mathcal{Q}|\sum_{i\in\mathcal{P}} E_{i,i}-|\mathcal{P}|\sum_{j\in\mathcal{Q}} E_{j,j}$$ is contained in the center of $\mathfrak{g}_A(\mathcal{P}+\mathcal{Q})$.} it follows from Theorem 7 of \textbf{\cite{SeriesA}} and Theorem 13 of \textbf{\cite{ContactLiePoset}} that a disconnected poset is contact if and only if it is a disjoint sum of exactly two Frobenius posets. Therefore, if the conjecture of \textbf{\cite{Binary}} is proven true, the following conjecture will also be established.

\begin{conj}
Let $\mathcal{P}$ be a disconnected poset. Then $\mathcal{P}$ is contact if and only if $\mathcal{P}$ is a disjoint sum of two Frobenius, toral posets.
\end{conj}

\begin{Ex}\label{ex:disjcontact}
The poset $\mathcal{P}$ whose Hasse diagram is illustrated in Figure~\ref{fig:discont} below is a disjoint sum of Frobenius, toral posets and, thus, is contact.
\bigskip

\begin{center}
\fbox{\begin{minipage}{45em}
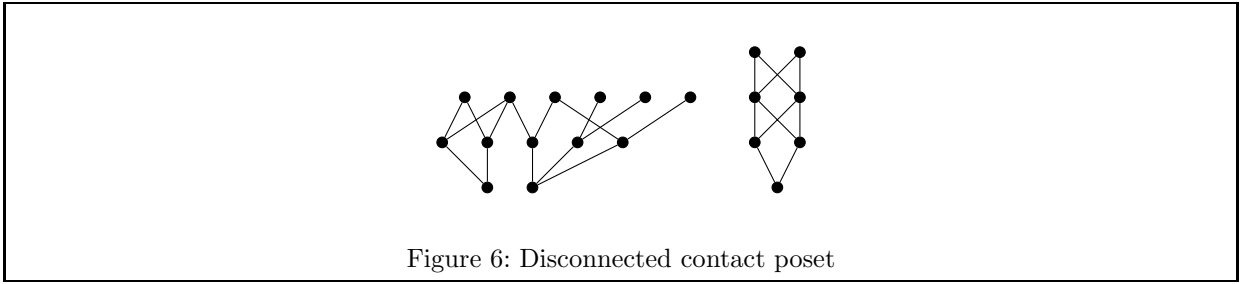
\begin{figure}[H]
$$\begin{tikzpicture}[scale=0.6]
\node [circle, draw = black, fill = black, inner sep = 0.5mm] (v46) at (17,1.5) {};
\node [circle, draw = black, fill = black, inner sep = 0.5mm] (v36) at (17.5,2.5) {};
\node [circle, draw = black, fill = black, inner sep = 0.5mm] (v37) at (18.5,2.5) {};
\node [circle, draw = black, fill = black, inner sep = 0.5mm] (v40) at (19.5,2.5) {};
\node [circle, draw = black, fill = black, inner sep = 0.5mm] (v42) at (20.5,2.5) {};
\node [circle, draw = black, fill = black, inner sep = 0.5mm] (v43) at (21.5,2.5) {};
\node [circle, draw = black, fill = black, inner sep = 0.5mm] (v45) at (22.5,2.5) {};
\node [circle, draw = black, fill = black, inner sep = 0.5mm] (v35) at (18,1.5) {};
\node [circle, draw = black, fill = black, inner sep = 0.5mm] (v39) at (19,1.5) {};
\node [circle, draw = black, fill = black, inner sep = 0.5mm] (v41) at (20,1.5) {};
\node [circle, draw = black, fill = black, inner sep = 0.5mm] (v44) at (21,1.5) {};
\node [circle, draw = black, fill = black, inner sep = 0.5mm] (v34) at (18,0.5) {};
\node [circle, draw = black, fill = black, inner sep = 0.5mm] (v38) at (19,0.5) {};
\draw (v34) -- (v35) -- (v36);
\draw (v35) -- (v37);
\draw (v38) -- (v39) -- (v37);
\draw (v39) -- (v40);
\draw (v38) -- (v41) -- (v42);
\draw (v41) -- (v43);
\draw (v38) -- (v44) -- (v40);
\draw (v44) -- (v45);
\draw (v34) -- (v46) -- (v36);
\draw (v46) -- (v37);
\end{tikzpicture}\quad\quad \begin{tikzpicture}[scale=0.6]
\node [circle, draw = black, fill = black, inner sep = 0.5mm] (v1) at (0,0) {};
\node [circle, draw = black, fill = black, inner sep = 0.5mm] (v2) at (-0.5, 1) {};
\node [circle, draw = black, fill = black, inner sep = 0.5mm] (v3) at (0.5, 1) {};
\node [circle, draw = black, fill = black, inner sep = 0.5mm] (v4) at (-0.5, 2) {};
\node [circle, draw = black, fill = black, inner sep = 0.5mm] (v5) at (0.5, 2) {};
\node [circle, draw = black, fill = black, inner sep = 0.5mm] (v6) at (-0.5, 3) {};
\node [circle, draw = black, fill = black, inner sep = 0.5mm] (v7) at (0.5, 3) {};
\draw (v1)--(v2)--(v5)--(v7)--(v4)--(v3)--(v1);
\draw (v3)--(v5)--(v6)--(v4)--(v2);
\end{tikzpicture}$$
\caption{Disconnected contact poset}\label{fig:discont}
\end{figure}
\end{minipage}}
\end{center}
\end{Ex}
\bigskip

Moving forward, our focus is characterizing contact, toral posets (which are necessarily connected). We proceed by first giving some necessary conditions for a type-A Lie poset algebra associated to a toral poset to admit a contact form, and in Section~\ref{sec:contactforms}, we show that such conditions are also sufficient. The following lemma places an upper bound on the number of contact toral-pairs that can appear in a construction sequence of a contact, toral poset.

\begin{lemma}\label{lem:additive}
Let $\mathcal{P}$ be a toral poset constructed from the \textup(contact\textup) toral-pairs $\{(\mathcal{S}_i,\varphi_{\mathcal{S}_i})\}_{i=1}^n$ with construction sequence $\mathcal{S}_1=\mathcal{Q}_1\subset\mathcal{Q}_2\subset\dots\mathcal{Q}_n=\mathcal{P}.$ If $\mathcal{P}$ is contact, then $(\mathcal{S}_i,\varphi_{\mathcal{S}_i})$ is a contact toral-pair for at most one value of $i\in\{1,\dots,n\}.$
\end{lemma}
\begin{proof}
Follows immediately from Theorem~\ref{lem:table} and the fact that contact Lie algebras have index one.
\end{proof}

\noindent Just as Lemma~\ref{lem:additive} above placed restrictions on the building blocks of a contact, toral poset, Lemma~\ref{lem:noFrobglue} below restricts the set of gluing rules that may be used to construct a contact, toral poset. First, we recall the following result from \textbf{\cite{ContactLiePoset}}.

\begin{theorem}[\textbf{\cite{ContactLiePoset}}, Theorem 15]\label{lem:cycle}
If the Hasse diagram of $\mathcal{P}_{Ext(\mathcal{P})}$ contains a cycle, then $\mathfrak{g}_A(\mathcal{P})$ is not contact.
\end{theorem}

\begin{Ex}\label{ex:cycle}
Consider the poset $\mathcal{P}=\{1,2,3,4,5,6,7\}$ with relations $1\prec 3,4\prec 6$; $2\prec 4\prec 7$; and $2\prec 5$ whose Hasse diagram is illustrated in Figure \ref{fig:cycles} below \textup(left\textup). Since the Hasse diagram of $\mathcal{P}_{Ext(\mathcal{P})}$ \textup(Figure~\ref{fig:cycles} \textup(right\textup)\textup) contains the cycle $(1,7,2,6),$ it follows from Theorem~\ref{lem:cycle} that $\mathcal{P}$ is not contact.
\bigskip

\begin{center}
\fbox{\begin{minipage}{45em}
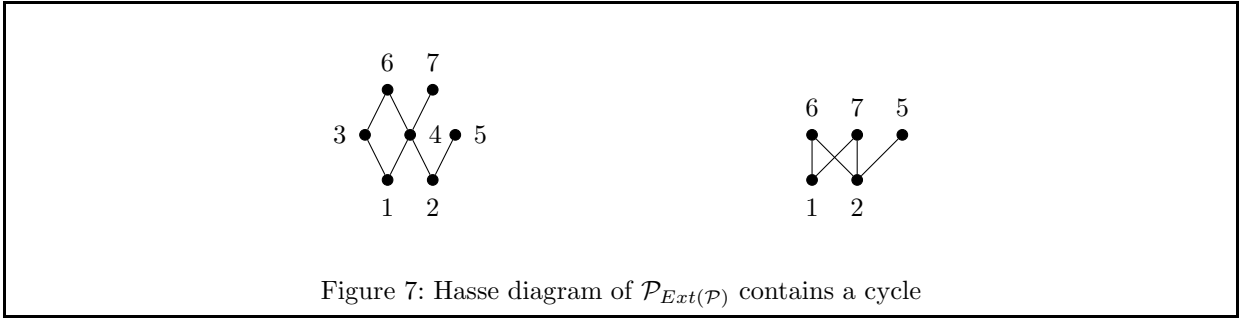
\begin{figure}[H]\label{Hasse}
$$\begin{tikzpicture}[scale=0.6]
\node (1) at (0,0) [circle, draw = black, fill = black, inner sep = 0.5mm] {};
\node (2) at (1,0) [circle,draw=black, fill=black, inner sep=0.5mm]{};
\node (3) at (-0.5,1) [circle,draw=black, fill=black, inner sep=0.5mm]{};
\node (4) at (0.5,1) [circle,draw=black, fill=black, inner sep=0.5mm]{};
\node (5) at (0,2) [circle,draw=black, fill=black, inner sep=0.5mm]{};
\node (6) at (1,2) [circle,draw=black, fill=black, inner sep=0.5mm]{};
\node (7) at (1.5, 1) [circle,draw=black, fill=black, inner sep=0.5mm]{};
\node[label=below:{$1$}] at (0,0){};
\node[label=below:{$2$}] at (1,0){};
\node[label=left:{$3$}] at (-0.5,1){};
\node[label=right:{$4$}] at (0.5,1){};
\node[label=right:{$5$}] at (1.5,1){};
\node[label=above:{$6$}] at (0,2){};
\node[label=above:{$7$}] at (1,2){};
\draw (1)--(3);
\draw (1)--(4);
\draw (3)--(5);
\draw (4)--(5);
\draw (4)--(6);
\draw (2)--(4);
\draw (2)--(7);
\end{tikzpicture}
\hspace{4cm}
\begin{tikzpicture}[scale=0.6]
\node (1) at (0,0) [circle, draw = black, fill = black, inner sep = 0.5mm] {};
\node (2) at (1,0) [circle,draw=black, fill=black, inner sep=0.5mm]{};
\node (3) at (0,1) [circle,draw=black, fill=black, inner sep=0.5mm]{};
\node (4) at (1,1) [circle,draw=black, fill=black, inner sep=0.5mm]{};
\node (5) at (2, 1) [circle,draw=black, fill=black, inner sep=0.5mm]{};
\node[label=above:{$7$}] at (1,1){};
\node[label=above:{$6$}] at (0,1){};
\node[label=above:{$5$}] at (2,1){};
\node[label=below:{$2$}] at (1,0){};
\node[label=below:{$1$}] at (0,0){};
\draw (1)--(3);
\draw (1)--(4);
\draw (2)--(3);
\draw (2)--(4);
\draw (2)--(5);
\end{tikzpicture}$$
    \caption{Hasse diagram of $\mathcal{P}_{Ext(\mathcal{P})}$ contains a cycle}
    \label{fig:cycles}
\end{figure}
\end{minipage}}
\end{center}
\end{Ex}
\bigskip

\begin{lemma}\label{lem:noFrobglue}
Let $\mathcal{P}$ be a toral poset constructed from the \textup(contact\textup) toral-pairs $\{(\mathcal{S}_i,\varphi_{\mathcal{S}_i})\}_{i=1}^n$, where $(\mathcal{S}_i,\varphi_{\mathcal{S}_i})$ is a contact toral-pair for at most one value of $i\in\{1,\dots,n\}.$ If $\mathcal{P}$ is contact, then $\mathcal{Q}_i$ is formed from $\mathcal{Q}_{i-1}$ and $\mathcal{S}_i$ by applying rules from the set $\{A_1,A_2,C,D_1,D_2,F\}$, for $i=2,\dots,n.$
\end{lemma}
\begin{proof}
Let $\mathcal{P}$ be a poset with construction sequence $\mathcal{S}_1=\mathcal{Q}_1\subset\mathcal{Q}_2\subset\dots\subset\mathcal{Q}_n=\mathcal{P},$ and assume that $\mathcal{Q}_i$ is formed from $\mathcal{Q}_{i-1}$ and $\mathcal{S}_i$ by applying a rule from the set $\{B,E_1,E_2,G_1,G_2,H\}$, for some $i\in\{2,\dots,n\}$. Note immediately that if rule $H$ is used to construct $\mathcal{P}$, then $\ind(\mathfrak{g}_A(\mathcal{P}))\geq 2,$ so $\mathcal{P}$ is not contact. Note also that each of the rules $B,E_1,E_2,G_1,$ and $G_2$ contributes $\ind(\mathfrak{g}_A(\mathcal{S}_i))+1$ to the index of $\mathfrak{g}_A(\mathcal{P})$, so if any $\mathcal{S}_i$ or $\mathcal{Q}_i$ is not Frobenius, then $\mathcal{P}$ is not contact. To ease notation without loss of generality, let $\mathcal{Q}=\mathcal{Q}_{i-1},$ $\mathcal{S}=\mathcal{S}_i,$ and $\mathcal{P}=\mathcal{Q}_i.$ 

Proceeding, we will show that applying each of rules $B,E_1,E_2,G_1,$ and $G_2$ to the posets $\mathcal{Q}$ and $\mathcal{S}$ results in a poset $\mathcal{P}$ with a cycle in the Hasse diagram of $\mathcal{P}_{Ext(\mathcal{P})}.$ If $|Ext(\mathcal{S})|=2,$ then let $Ext(\mathcal{S})=\{a_1,c\}$ with $c\prec a_1,$ and if $|Ext(\mathcal{S})|=3,$ then let $Ext(\mathcal{S})=\{a_1,a_2,c\}$ with $a_1,a_2\prec c$ or $c\prec a_1,a_2.$ If $|Ext(\mathcal{Q})|=2,$ then let $Ext(\mathcal{Q})=\{x,y\},$ and if $|Ext(\mathcal{Q})|>2,$ then let $\{x,y,z\}\subseteq Ext(\mathcal{Q}).$ As above, we assume that if $c,a_1,$ or $a_2$ are identified with elements of $\mathcal{Q}$, then those elements are $x,y$, and $z$, respectively. Since $\mathcal{Q}$ is connected, $\mathcal{Q}_{Ext(\mathcal{Q})}$ is also connected. Thus, there exists a path $P_{xy}$ from $x$ to $y$ in the Hasse diagram of $\mathcal{Q}_{Ext(\mathcal{Q})},$ and if $z\in Ext(\mathcal{Q}),$ then there exist paths $P_{yz}$ and $P_{xz}$ from $y$ to $z$ and from $x$ to $z,$ respectively, in the Hasse diagram of $\mathcal{Q}_{Ext(\mathcal{Q})}.$ Moreover, the Hasse diagrem of $\mathcal{Q}_{Ext(\mathcal{Q})}$ is a subgraph of the Hasse diagram of $\mathcal{P}_{Ext(\mathcal{P})},$ so $P_{xy}, P_{yz},$ and $P_{xz}$ are also paths in $\mathcal{P}_{Ext(\mathcal{P})}.$ Three cases arise.

\bigskip
\noindent
\textbf{Case 1:} $\mathcal{P}$ is formed by adjoining $\mathcal{S}$ to $\mathcal{Q}$ using rule $B.$ Combining the edges of $P_{yz}$ with the edges $\{y,c\}$ and $\{c,z\}$ yields a cycle in the Hasse diagram of $\mathcal{P}_{Ext(\mathcal{P})}.$

\bigskip
\noindent
\textbf{Case 2:} $\mathcal{P}$ is formed by adjoining $\mathcal{S}$ to $\mathcal{Q}$ using rule $E_1$ or $G_2.$ Combining the edges of $P_{xy}$ with the edge $\{x,y\}$ yields a cycle in the Hasse diagram of $\mathcal{P}_{Ext(\mathcal{P})}.$

\bigskip
\noindent
\textbf{Case 3:} $\mathcal{P}$ is formed by adjoining $\mathcal{S}$ to $\mathcal{Q}$ using rule $E_2$ or $G_1.$ Combining the edges of $P_{xz}$ with the edge $\{x,z\}$ yields a cycle in the Hasse diagram of $\mathcal{P}_{Ext(\mathcal{P})}.$
\bigskip

\noindent
In each case, the Hasse diagram $\mathcal{P}_{Ext(\mathcal{P})}$ contains a cycle, so $\mathfrak{g}_A(\mathcal{P})$ is not contact by Theorem~\ref{lem:cycle}. The result follows.
\end{proof}

\noindent Combining the formulas from Theorem~\ref{lem:table} with Lemmas~\ref{lem:additive} and \ref{lem:noFrobglue}, note that a contact, toral poset cannot be constructed using only toral-pairs; that is, at least one building block must be contact. Thus, we are led to the following theorem, which states that the upper bound of Lemma~\ref{lem:additive} is exact.

\begin{theorem}
Let $\mathcal{P}$ be a toral poset constructed from the \textup(contact\textup) toral pairs $\{(\mathcal{S}_i,\varphi_{\mathcal{S}_i})\}_{i=1}^n.$ If $\mathcal{P}$ is contact, then there is exactly one $i\in\{1,\dots,n\}$ such that $(\mathcal{S}_i,\varphi_{\mathcal{S}_i})$ is a contact toral-pair.
\end{theorem}

At this point, we have established the ``only if" direction of the following theorem. The other direction is the content of the next section.

\begin{theorem}\label{thm:big}
Let $\mathcal{P}$ be a toral poset constructed from the \textup(contact\textup) toral-pairs $\{(\mathcal{S}_i,F_{\mathcal{S}_i})\}_{i=1}^n$ with construction sequence $\mathcal{S}_1=\mathcal{Q}_1\subset\mathcal{Q}_2\subset\cdots\subset\mathcal{Q}_n=\mathcal{P}$. Then, $\mathcal{P}$ is contact if and only if
\begin{enumerate}
    \item[\textup1.] $(\mathcal{S}_i,F_{\mathcal{S}_i})$ is a contact toral-pair for exactly one value of $i\in\{1,\dots,n\}$, and
    \item[\textup2.] $\mathcal{Q}_i$ is formed from $\mathcal{Q}_{i-1}$ and $\mathcal{S}_i$ by applying rules from the set $\{A_1,A_2,C,D_1,D_2,F\}$, for $i=2,\hdots,n$.
\end{enumerate}
\end{theorem}

\noindent Theorem~\ref{thm:big} then motivates the following definition of a ``contact sequence."

\begin{definition}\label{def:contoral}
Let $\mathcal{P}$ be a toral poset constructed from the \textup(contact\textup) toral-pairs $\{(\mathcal{S}_i,F_{\mathcal{S}_i})\}_{i=1}^n$ with construction sequence $\mathcal{S}_1=\mathcal{Q}_1\subset\mathcal{Q}_2\subset\cdots\subset\mathcal{Q}_n=\mathcal{P}$. If $\mathcal{Q}_i$ is formed from $\mathcal{Q}_{i-1}$ and $\mathcal{S}_i$ by applying a rule from the set $\{A_1,A_2,C,D_1,D_2,F\}$, for $i=2,\dots,n$, and $\mathcal{S}_i$ is the poset of a contact toral-pair for exactly one value of $i\in\{1,\dots,n\},$ then we call the sequence $\mathcal{S}_1=\mathcal{Q}_1\subset\mathcal{Q}_2\subset\dots\subset\mathcal{Q}_n=\mathcal{P}$ a \textit{contact sequence}.
\end{definition}

\begin{Ex}\label{ex:conseq}
Let $\mathcal{P}$ be the toral poset constructed from the \textup(contact\textup) toral-pairs $\{(\mathcal{S}_i,\varphi_{\mathcal{S}_i})\}_{i=1}^5$ with attendant contact sequence $$\mathcal{S}_1=\mathcal{Q}_1\subset\mathcal{Q}_2\subset\mathcal{Q}_3\subset\mathcal{Q}_4\subset\mathcal{Q}_5=\mathcal{P}$$ such that
\begin{itemize}
    \item $\mathcal{S}_1=\mathcal{Q}_1$ is isomorphic to the poset on $\{1,2,3,4,5\}$ with $1\prec2\prec3\prec4,5$,
    \item $\mathcal{S}_2$ is isomorphic to the poset on $\{1,2,3,4,5,6\}$ with $1\prec2,3\prec4\prec5$ and $4\prec 6$,
    \item $\mathcal{S}_3$ is isomorphic to the poset on $\{1,2,3,4,5,6\}$ with $1\prec2\prec3,4\prec 5$ and $2\prec 6$,
    \item $\mathcal{S}_4$ is isomorphic to the poset on $\{1,2,3,4\}$ with $1,2\prec3\prec 4$,
    \item $\mathcal{S}_5$ is isomorphic to the poset on $\{1,2,3,4,5,6\}$ with $1\prec3\prec4,5\prec 6$ and $2\prec 5,$
    \item $\mathcal{Q}_2$ is formed from $\mathcal{Q}_1$ and $\mathcal{S}_2$ by applying rule $A_1,$
    \item $\mathcal{Q}_3$ is formed from $\mathcal{Q}_2$ and $\mathcal{S}_3$ by applying rule $D_1,$
    \item $\mathcal{Q}_4$ is formed from $\mathcal{Q}_3$ and $\mathcal{S}_4$ by applying rule $C,$ and
    \item $\mathcal{Q}_5$ is formed from $\mathcal{Q}_4$ and $\mathcal{S}_5$ by applying rule $F.$
\end{itemize}
In Figure~\ref{fig:conseq} below we illustrate the posets $\mathcal{Q}_i$ of the contact sequence described above.
\bigskip

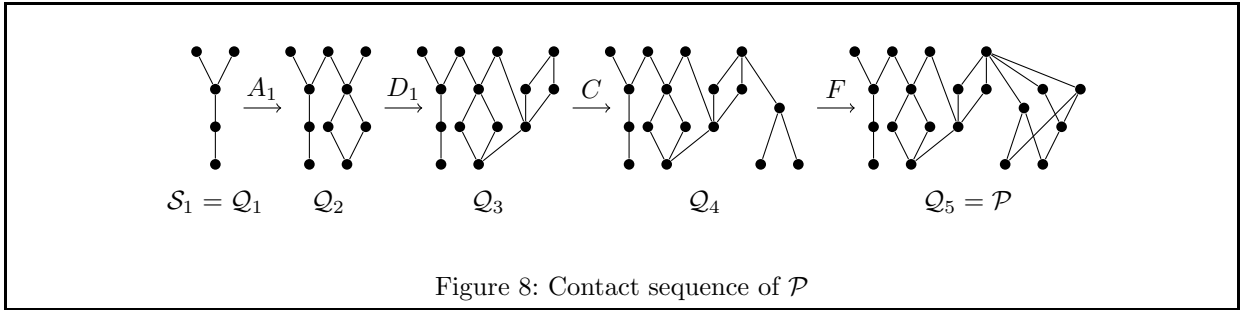
\begin{center}
\fbox{\begin{minipage}{45em}
\begin{figure}[H]
$$\begin{tikzpicture}[scale=0.5]
    \def\Node{\node [circle, fill, inner sep=0.5mm]}
    \Node (1) at (0,0){};
    \Node (2) at (0,1){};
    \Node (3) at (0,2){};
    \Node (4) at (0.5,3){};
    \Node (5) at (-0.5,3){};
    \draw (1)--(3)--(4);
    \draw (3)--(5);
    
    \node at (0,-1){$\mathcal{S}_1=\mathcal{Q}_1$};
    
    \draw[->] (0.75,1.5)--(1.75,1.5);
    \node at (1.25,2){$A_1$};
    
    \Node (6) at (2.5,0){};
    \Node (7) at (2.5,1){};
    \Node (8) at (2.5,2){};
    \Node (9) at (2,3){};
    \Node (10) at (3,3){};
    \draw (6)--(8)--(9);
    \draw (8)--(10);
    
    \Node (11) at (3.5,0){};
    \Node (12) at (3,1){};
    \Node (13) at (4,1){};
    \Node (14) at (3.5,2){};
    \Node (15) at (4,3){};
    \draw (11)--(12)--(14)--(15);
    \draw (11)--(13)--(14)--(10);
    
    \node at (3,-1){$\mathcal{Q}_2$};
    
    \draw[->] (4.5,1.5)--(5.5,1.5);
    \node at (5,2){$D_1$};
    
    \Node (16) at (6,0){};
    \Node (17) at (6,1){};
    \Node (18) at (6,2){};
    \Node (19) at (5.5,3){};
    \Node (20) at (6.5,3){};
    \draw (16)--(18)--(19);
    \draw (18)--(20);
    
    \Node (21) at (7,0){};
    \Node (22) at (6.5,1){};
    \Node (23) at (7.5,1){};
    \Node (24) at (7,2){};
    \Node (25) at (7.5,3){};
    \draw (21)--(22)--(24)--(25);
    \draw (21)--(23)--(24)--(20);
    
    \Node (26) at (8.25,1){};
    \Node (27) at (8.25,2){};
    \Node (28) at (9,3){};
    \Node (29) at (9,2){};
    \draw (21)--(26)--(27)--(28)--(29)--(26)--(25);
    
    \node at (7.25,-1){$\mathcal{Q}_3$};
    
    \draw[->] (9.5,1.5)--(10.5,1.5);
    \node at (10,2){$C$};
    
    \Node (30) at (11,0){};
    \Node (31) at (11,1){};
    \Node (32) at (11,2){};
    \Node (33) at (10.5,3){};
    \Node (34) at (11.5,3){};
    \draw (30)--(32)--(33);
    \draw (32)--(34);
    
    \Node (35) at (12,0){};
    \Node (36) at (11.5,1){};
    \Node (37) at (12.5,1){};
    \Node (38) at (12,2){};
    \Node (39) at (12.5,3){};
    \draw (35)--(36)--(38)--(39);
    \draw (35)--(37)--(38)--(34);
    
    \Node (40) at (13.25,1){};
    \Node (41) at (13.25,2){};
    \Node (42) at (14,3){};
    \Node (43) at (14,2){};
    \draw (35)--(40)--(41)--(42)--(43)--(40)--(39); 
    
    \Node (44) at (14.5,0){};
    \Node (45) at (15.5,0){};
    \Node (46) at (15,1.5){};
    \draw (44)--(46)--(42);
    \draw (45)--(46);
    
    \node at (13,-1){$\mathcal{Q}_4$};
    
    \draw[->] (16,1.5)--(17,1.5);
    \node at (16.5,2){$F$};
    
    \Node (47) at (17.5,0){};
    \Node (48) at (17.5,1){};
    \Node (49) at (17.5,2){};
    \Node (50) at (17,3){};
    \Node (51) at (18,3){};
    \draw (47)--(49)--(50);
    \draw (49)--(51);
    
    \Node (52) at (18.5,0){};
    \Node (53) at (18,1){};
    \Node (54) at (19,1){};
    \Node (55) at (18.5,2){};
    \Node (56) at (19,3){};
    \draw (52)--(53)--(55)--(56);
    \draw (52)--(54)--(55)--(51);
    
    \Node (57) at (19.75,1){};
    \Node (58) at (19.75,2){};
    \Node (59) at (20.5,3){};
    \Node (60) at (20.5,2){};
    \draw (52)--(57)--(58)--(59)--(60)--(57)--(56); 
    
    \Node (61) at (21,0){};
    \Node (62) at (22,0){};
    \Node (63) at (21.5,1.5){};
    \draw (61)--(63)--(59);
    \draw (62)--(63);
    
    \Node (64) at (22.5,1){};
    \Node (65) at (22,2){};
    \Node (66) at (23,2){};
    \draw (61)--(66);
    \draw (62)--(64)--(65)--(59)--(66)--(64);
    
    \node at (20,-1){$\mathcal{Q}_5=\mathcal{P}$};
\end{tikzpicture}$$
    \caption{Contact sequence of $\mathcal{P}$}
    \label{fig:conseq}
\end{figure}
\end{minipage}}
\end{center}
\end{Ex}

\begin{remark}\label{rem:order}
To ease discourse in the next section, we set the following conventions:
\begin{itemize}
    \item Given a contact sequence $\mathcal{S}_1=\mathcal{Q}_1\subset\mathcal{Q}_2\subset\dots\subset\mathcal{Q}_n$, assume that $\mathcal{S}_1=\mathcal{Q}_1$ is the unique poset from a contact toral-pair, as in Example~\ref{ex:conseq} above. That is, we adopt the convention that each contact sequence begins life as a poset from a contact toral-pair and grows by gluing on posets from toral-pairs, according to the gluing rules given in Lemma~\ref{lem:noFrobglue}. 
    \item A toral poset $\mathcal{P}$ for which there exists a contact sequence $\mathcal{Q}_1\subset\mathcal{Q}_2\subset\dots\subset\mathcal{Q}_n=\mathcal{P}$ will be called a contact toral poset. This naming convention will be justified upon the completion of the proof of Theorem~\ref{thm:big}.
\end{itemize}
\end{remark}

\section{Contact toral one-forms}\label{sec:contactforms}
In this section, given a contact toral poset $\mathcal{P}$ constructed from the (contact) toral-pairs $\{(\mathcal{S}_i,\varphi_{\mathcal{S}_i})\}_{i=1}^n$, we provide an inductive procedure for constructing a contact one-form $\varphi_{\mathcal{P}}\in\left(\mathfrak{g}_A(\mathcal{P})\right)^*$ from the (contact and Frobenius) one-forms $\varphi_{\mathcal{S}_i},$ for $i=1,\dots,n.$ Once complete, we will have  finished the proof of Theorem~\ref{thm:big}.

\begin{remark}
Let $\mathcal{P}$ be a toral poset constructed from the \textup(contact\textup) toral-pairs $\{(\mathcal{S}_i,\varphi_{\mathcal{S}_{i}})\}_{i=1}^n$ with contact sequence $\mathcal{S}_1=\mathcal{Q}_1\subset\mathcal{Q}_2\subset\hdots\subset\mathcal{Q}_n=\mathcal{P}$. Throughout this section, in the notation of Section~\ref{sec:contoral}, if $\mathcal{Q}=\mathcal{Q}_{i-1}$ and $\mathcal{S}=\mathcal{S}_i$, for $i=2,\hdots,n$, then we denote the elements of $\{a,x\}$ by $x_i$, $\{b,y\}$ by $y_i$, and $\{c,z\}$ by $z_i$.
\end{remark}

\begin{definition}\label{def:toralform}
If $\mathcal{P}$ is a contact toral poset constructed from the \textup(contact\textup) toral-pairs $\{(\mathcal{S}_i,\varphi_{\mathcal{S}_{i}})\}_{i=1}^n$ with contact sequence  $\mathcal{S}_1=\mathcal{Q}_1\subset\mathcal{Q}_2\subset\hdots\subset\mathcal{Q}_n=\mathcal{P}$, then define the contact toral one-form $\varphi_{\mathcal{Q}_i}\in(\mathfrak{g}_A(\mathcal{Q}_i))^*$, for $i=1,\hdots,n$, as follows:
\begin{itemize}[label={\large\textbullet}]
    \item $\varphi_{\mathcal{Q}_{1}}=\varphi_{\mathcal{S}_{1}}$;
    \item if $\mathcal{Q}_{i}$ is formed from $\mathcal{Q}_{i-1}$ and $\mathcal{S}_i$, for $1<i\le n$, by applying rule
    \begin{itemize}[label={\small\textbullet}]
        \item $\textup{A}_1,\textup{A}_2$, or $\textup{C}$, then $$\varphi_{\mathcal{Q}_i}=\varphi_{\mathcal{Q}_{i-1}}+\varphi_{\mathcal{S}_{i}}.$$
        \item $\textup{D}_1$, then
        \[\varphi_{\mathcal{Q}_i} =  \begin{cases} 
      \varphi_{\mathcal{Q}_{i-1}}+\varphi_{\mathcal{S}_{i}}-E^*_{x_i,y_i}, & \mathcal{S}_i\text{ has one minimal element}; \\
                                                &                        \\
       \varphi_{\mathcal{Q}_{i-1}}+\varphi_{\mathcal{S}_{i}}-E^*_{y_i,x_i}, & \mathcal{S}_i\text{ has one maximal element}.
   \end{cases}
\]
        \item $\textup{D}_2$, then
        \[\varphi_{\mathcal{Q}_i} =  \begin{cases} 
      \varphi_{\mathcal{Q}_{i-1}}+\varphi_{\mathcal{S}_{i}}-E^*_{x_i,z_i}, & \mathcal{S}_i\text{ has one minimal element}; \\
                                                &                        \\
       \varphi_{\mathcal{Q}_{i-1}}+\varphi_{\mathcal{S}_{i}}-E^*_{z_i,x_i}, & \mathcal{S}_i\text{ has one maximal element}.
   \end{cases}
\]
        \item $\textup{F}$, then
        \[\varphi_{\mathcal{Q}_i} =  \begin{cases} 
      \varphi_{\mathcal{Q}_{i-1}}+\varphi_{\mathcal{S}_{i}}-E^*_{x_i,y_i}-E^*_{x_i,z_i}, & \mathcal{S}_i\text{ has one minimal element}; \\
                                                &                        \\
       \varphi_{\mathcal{Q}_{i-1}}+\varphi_{\mathcal{S}_{i}}-E^*_{y_i,x_i}-E^*_{z_i,x_i}, & \mathcal{S}_i\text{ has one maximal element}.
   \end{cases}
\]
    \end{itemize}
\end{itemize}
\end{definition}
\medskip

\begin{remark}\label{rem:finf}
Note that $E^*_{p,q}$ is a summand of $\varphi_{\mathcal{Q}_i}$ if and only if $E^*_{p,q}$ is a summand of $\varphi_{\mathcal{Q}_{i-1}}$ or $\varphi_{\mathcal{S}_i}$.
\end{remark}

\begin{Ex}\label{ex:contoralform}
Consider the poset $\mathcal{P}$ and associated contact sequence $\mathcal{S}_1=\mathcal{Q}_1\subset\dots\subset\mathcal{Q}_5=\mathcal{P}$ of Example~\ref{ex:conseq}. If the Hasse diagram is labeled as in Figure~\ref{fig:contoralform} below, then we have the following construction of the contact toral one-form $\varphi_{\mathcal{P}}:$
\begin{itemize}
    \item $\varphi_{\mathcal{Q}_1}=\varphi_{\mathcal{S}_1}=E_{1,1}^*+E_{1,4}^*+E_{1,5}^*+E_{2,3}^*+E_{2,5}^*$ by Theorem~\ref{thm:fork},
    \item $\varphi_{\mathcal{Q}_2}=\varphi_{\mathcal{Q}_1}+E_{6,5}^*+E_{6,10}^*+E_{7,9}^*+E_{7,5}^*+E_{8,10}^*$ by Theorem~\ref{thm:ntp}\textup(i\textup) and the fact that gluing rule $A_1$ was applied,
    \item $\varphi_{\mathcal{Q}_3}=\varphi_{\mathcal{Q}_2}+E_{6,14}^*+E_{6,10}^*+E_{11,12}^*+E_{11,13}^*+E_{11,14}^*-E_{6,10}^*$ by Theorem~\ref{thm:ntp}\textup(v\textup) and the fact that gluing rule $D_1$ was applied,
    \item $\varphi_{\mathcal{Q}_4}=\varphi_{\mathcal{Q}_3}+E_{15,14}^*+E_{16,14}^*+E_{16,17}^*$ by Theorem 4\textup(ii\textup) of \textup{\textbf{\cite{Binary}}} and the fact that gluing rule $C$ was applied, and
    \item $\varphi_{\mathcal{Q}_5}=\varphi_{\mathcal{Q}_4}+E_{15,14}^*+E_{16,14}^*+E_{16,20}^*+E_{18,19}^*+E_{18,20}^*-E_{15,14}^*-E_{16,14}^*$ by Theorem~\ref{thm:ntp}\textup(iv\textup) and the fact that gluing rule $F$ was applied.
\end{itemize}
\bigskip

\begin{center}
\fbox{\begin{minipage}{45em}
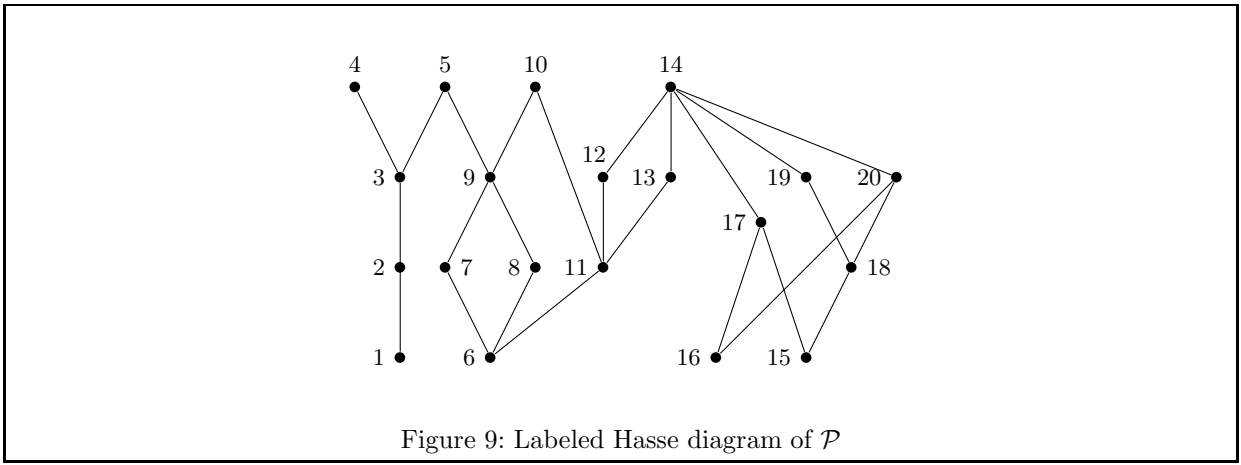
\begin{figure}[H]
$$\begin{tikzpicture}[scale=1.2]
\def\Node{\node [circle, fill, inner sep=0.5mm]}
    
    \Node[label=left:{\small 1}] (47) at (17.5,0){};
    \Node[label=left:{\small 2}] (48) at (17.5,1){};
    \Node[label=left:{\small 3}] (49) at (17.5,2){};
    \Node[label=above:{\small 4}] (50) at (17,3){};
    \Node[label=above:{\small 5}] (51) at (18,3){};
    \draw (47)--(49)--(50);
    \draw (49)--(51);
    
    \Node[label=left:{\small 6}] (52) at (18.5,0){};
    \Node[label=right:{\small 7}] (53) at (18,1){};
    \Node[label=left:{\small 8}] (54) at (19,1){};
    \Node[label=left:{\small 9}] (55) at (18.5,2){};
    \Node[label=above:{\small 10}] (56) at (19,3){};
    \draw (52)--(53)--(55)--(56);
    \draw (52)--(54)--(55)--(51);
    
    \Node[label=left:{\small 11}] (57) at (19.75,1){};
    \Node (58) at (19.75,2){};
    \node at (19.65,2.25){\small 12};
    \Node[label=above:{\small 14}] (59) at (20.5,3){};
    \Node[label=left:{\small 13}] (60) at (20.5,2){};
    \draw (52)--(57)--(58)--(59)--(60)--(57)--(56); 
    
    \Node[label=left:{\small 16}] (61) at (21,0){};
    \Node[label=left:{\small 15}] (62) at (22,0){};
    \Node[label=left:{\small 17}] (63) at (21.5,1.5){};
    \draw (61)--(63)--(59);
    \draw (62)--(63);
    
    \Node[label=right:{\small 18}] (64) at (22.5,1){};
    \Node[label=left:{\small 19}] (65) at (22,2){};
    \Node[label=left:{\small 20}] (66) at (23,2){};
    \draw (61)--(66);
    \draw (62)--(64)--(65)--(59)--(66)--(64);
\end{tikzpicture}$$
\caption{Labeled Hasse diagram of $\mathcal{P}$}
\label{fig:contoralform}
\end{figure}
\end{minipage}}
\end{center}
\end{Ex}

By definition, contact toral one-forms immediately satisfy \textbf{(CF1),(CF2),(CF3),} and \textbf{(CF4)} of Definition~\ref{def:contacttoral}. We document this in the following lemma.

\begin{lemma}\label{lem:toralUD}
If $\mathcal{P}$ is a contact toral poset constructed from the \textup(contact\textup) toral-pairs $\{(\mathcal{S}_i,\varphi_{\mathcal{S}_i})\}_{i=1}^n$ with contact sequence $\mathcal{S}_1=\mathcal{Q}_1\subset\mathcal{Q}_2\subset\hdots\subset\mathcal{Q}_n=\mathcal{P}$, then
\begin{itemize}
    \item $E_{p,p}^*$ is a nontrivial summand of $\varphi_{\mathcal{Q}_i}$ if and only if $p=1,$
    \item $\varphi_{\mathcal{Q}_i}-E_{1,1}^*$ is small,
    \item $U_{\varphi_{\mathcal{Q}_i}-E_{1,1}^*}(\mathcal{Q}_i)$ is a filter, $D_{\varphi_{\mathcal{Q}_i}-E_{1,1}^*}(\mathcal{Q}_i)$ is an order ideal of $\mathcal{Q}_i$,  $O_{\varphi_{\mathcal{Q}_i}-E_{1,1}^*}(\mathcal{Q}_i)=\emptyset$, and
    \item $\Gamma_{\varphi-E_{1,1}^*}$ contains all edges between elements of $Ext(\mathcal{Q}_i),$
\end{itemize} 
for $i=1,\hdots,n$.
\end{lemma}
\begin{proof}
Follows immediately from Definitions~\ref{def:toralpair}, \ref{def:contacttoral}, and \ref{def:toralform}.
\end{proof}

\begin{remark}
Recall that for a poset $\mathcal{P}$, elements of $\mathfrak{g}(\mathcal{P})\subset\mathfrak{gl}(|\mathcal{P}|)$ are of the form $$\sum_{(i, j)\in Rel(\mathcal{P})}k_{i,j}E_{i,j}+\sum_{i\in\mathcal{P}}k_{i,i}E_{i,i}.$$ Let $\mathcal{P}$ be a poset formed by combining the posets $\mathcal{S}$ and $\mathcal{Q}$ by identifying minimal elements or maximal elements. If $B\in\mathfrak{g}(\mathcal{P})$, then let $B|_{\mathcal{Q}}$ denote the restriction of $B$ to basis elements of $\mathfrak{g}(\mathcal{Q})$ and $B|_{\mathcal{S}}$ denote the restriction of $B$ to basis elements of $\mathfrak{g}(\mathcal{S})$.
\end{remark}

In order to show that a contact toral one-form $\varphi_{\mathcal{P}}$ is indeed contact, we use Lemma~\ref{lem:kernel}, which requires an investigation of the kernel of $d\varphi_{\mathcal{P}}$. Lemmas~\ref{lem:2}, \ref{lem:restrict}, and \ref{lem:extequal} below identify properties that elements of such a kernel must possess. Notice that Lemmas~\ref{lem:2} and \ref{lem:restrict} are ``extended" versions of Lemmas 2 and 3, respectively, of \textbf{\cite{Binary}}, and the proofs are unchanged.

\begin{lemma}\label{lem:2}
If $\mathcal{P}$ is a contact toral poset and $B\in\mathfrak{g}(\mathcal{P})$ satisfies $\varphi_{\mathcal{P}}([E_{p,p},B])=0,$ for all $p\in\mathcal{P},$ then $E_{p,q}^*(B)=0,$ for all $p\prec q$ such that $E_{p,q}^*$ is a nontrivial summand of $\varphi_{\mathcal{P}}.$
\end{lemma}

\noindent Lemma~\ref{lem:2} offers some insight into the structure of $\ker(d\varphi_{\mathcal{P}}),$ where $\varphi_{\mathcal{P}}$ is a contact toral one-form on a type-A Lie poset algebra $\mathfrak{g}_A(\mathcal{P}).$ In particular, we have the following corollary.

\begin{corollary}\label{cor:e11}
If $(\mathcal{P},\varphi_{\mathcal{P}})$ is a contact toral-pair with $B\in\ker_A(d\varphi_{\mathcal{P}}),$ then $E_{1,1}^*(B)\neq 0.$
\end{corollary}
\begin{proof}
First, notice that for a contact toral-pair $(\mathcal{P},\varphi_{\mathcal{P}}),$ condition \textbf{(CF1)} immediately implies that $$\varphi_{\mathcal{P}}=E_{1,1}^*+\sum_{(p,q)\in S}E_{p,q}^*,$$ for some subset $S\subset Rel(\mathcal{P}).$

Now, let $B\in\ker_A(d\varphi_{\mathcal{P}})\subset\ker(d\varphi_{\mathcal{P}}).$ Since $\mathcal{P}$ is a contact toral poset, $\varphi_{\mathcal{P}}$ is a contact toral one-form, and $\varphi_{\mathcal{P}}([E_{p,p},B])=0,$ for all $p\in\mathcal{P},$ we can apply Lemma~\ref{lem:2}; that is, $$\varphi_{\mathcal{P}}(B)=E_{1,1}^*(B)+\sum_{(p,q)\in S}E_{p,q}^*(B)=E_{1,1}^*(B).$$ Since $\varphi_{\mathcal{P}}$ is contact, Lemma~\ref{lem:kernel} implies that $\varphi_{\mathcal{P}}(B)\neq 0.$ The result follows.
\end{proof}

\begin{lemma}\label{lem:restrict}
Let $\mathcal{P}$ be a contact toral poset constructed from the \textup(contact\textup) toral-pairs $\{(\mathcal{S}_i,\varphi_{\mathcal{S}_i})\}_{i=1}^n$ with contact sequence $\mathcal{S}_1=\mathcal{Q}_1\subset\mathcal{Q}_2\subset\dots\subset\mathcal{Q}_n=\mathcal{P}.$ If $1<i\leq n,$ then $B\in\ker(d\varphi_{\mathcal{Q}_i})$ must satisfy $B|_{\mathcal{Q}_{i-1}}\in\ker(d\varphi_{\mathcal{Q}_{i-1}})$ and $B|_{\mathcal{S}_i}\in\ker(d\varphi_{\mathcal{S}_i}).$
\end{lemma}

\begin{lemma}\label{lem:extequal}
If $(\mathcal{P},\varphi_{\mathcal{P}})$ is a contact toral-pair with $B\in\ker(d\varphi_{\mathcal{P}}),$ then $E_{p,p}^*(B)=E_{q,q}^*(B)$ for all $p,q\in Ext(\mathcal{P})$ such that $p\prec q.$
\end{lemma}
\begin{proof}
Let $\varphi=\varphi_{\mathcal{P}},$ $B\in\ker(d\varphi),$ and $p,q\in Ext(\mathcal{P}),$ with $p\prec q.$ Consider the equation $\varphi([E_{p,q},B])=0.$ By property $\mathbf{(CF4)},$ $E_{p,q}^*$ is a summand of $\varphi,$ and moreover, since $p,q\in Ext(\mathcal{P}),$ we have that $[E_{p,q},E_{r,s}]=0$ for all $r,s\in\mathcal{P}$ with $r\prec s.$ Therefore, $\varphi([E_{p,q},B])=E_{q,q}^*(B)-E_{p,p}^*(B)=0,$ and the result follows.
\end{proof}

\begin{remark}\label{rem:decomp}
We are now in a position to prove the main result of this paper \textup(see Theorem~\ref{thm:main} below\textup). First, recall that $\mathfrak{gl}(n)=\mathfrak{sl}(n)\oplus span\{I_n\},$ where $I_n=\sum_{i=1}^nE_{i,i}.$ Since a type-A Lie poset algebra $\mathfrak{g}_A(\mathcal{P})$ contains all diagonal elements of its parent simple Lie algebra, $\mathfrak{sl}(|\mathcal{P}|),$ we also have that $$\mathfrak{g}(\mathcal{P})=\mathfrak{g}_A(\mathcal{P})\oplus span\{I_{|\mathcal{P}|}\}.$$ Moreover, since $I_{|\mathcal{P}|}$ is central in $\mathfrak{gl}(|\mathcal{P}|),$ we have that $$\ker(d\varphi)=\ker_A(d\varphi)\oplus span\{I_{|\mathcal{P}|}\},$$ for all one-forms $\varphi\in(\mathfrak{g}_A(\mathcal{P}))^*.$
\end{remark}

\begin{theorem}\label{thm:main}
If $\mathcal{P}$ is a contact toral poset constructed from the \textup(contact\textup) toral-pairs $\{(\mathcal{S}_i,\varphi_{\mathcal{S}_i})\}_{i=1}^n$ with contact sequence $\mathcal{S}_1=\mathcal{Q}_1\subset\mathcal{Q}_2\subset\dots\subset\mathcal{Q}_n=\mathcal{P},$ then $\varphi_{\mathcal{Q}_i}\in(\mathfrak{g}_A(\mathcal{Q}_i))^*$ is a contact form, for all $i=1,\dots,n.$
\end{theorem}
\begin{proof}
To show $\varphi_{\mathcal{Q}_i}$ is a contact form on $\mathfrak{g}_A(\mathcal{Q}_i),$ we use induction on $i$ to show that the following condition holds:
\begin{equation}\label{eqn:indhyp}
\text{there exists } B_i\in\mathfrak{g}_A(\mathcal{Q}_i) \text{ such that } span\{B_i\}=\ker_A(d\varphi_{\mathcal{Q}_i}) \text{ and } \varphi_{\mathcal{Q}_i}(B_i)\neq 0.
\end{equation}
An application of Lemma~\ref{lem:kernel} will then finish the proof.

The case of $i=1$ is clear since, by definition, $(\mathcal{Q}_1,\varphi_{\mathcal{Q}_1})$ is a contact toral-pair. Assume condition (\ref{eqn:indhyp}) holds for $B_{i-1}\in\ker_A(d\varphi_{\mathcal{Q}_{i-1}}),$ for some $i\in\{2,\dots,n\}.$ Let $0\neq B_i\in\ker_A(d\varphi_{\mathcal{Q}_i}),$ and note immediately that such a $B_i$ exists because $\ind(\mathfrak{g}_A(\mathcal{Q}_i))=1$. Since $\mathfrak{g}_A(\mathcal{Q}_i)\subset\mathfrak{g}(\mathcal{Q}_i)$ and $\ker_A(d\varphi_{\mathcal{Q}_i})\subset\ker(d\varphi_{\mathcal{Q}_i}),$  Lemma~\ref{lem:restrict} applies to $B_i,$ and we have that 
\begin{equation}\label{eqn:(a)}
B_i|_{\mathcal{Q}_{i-1}}\in\ker(d\varphi_{\mathcal{Q}_{i-1}})
\end{equation}
and
\begin{equation}\label{eqn:(b)}
    B_i|_{\mathcal{S}_i}\in\ker(d\varphi_{\mathcal{S}_i}).
\end{equation}
Combining (\ref{eqn:(a)}) with the inductive hypothesis and Remark~\ref{rem:decomp}, we have that $B_i|_{\mathcal{Q}_{i-1}}=k_1B_{i-1}+k_2I_{|\mathcal{Q}_{i-1}|}$, where $k_1,k_2\in\mathbb{C}$. Further, combining (\ref{eqn:(b)}) with property $\mathbf{(F4)}$ of toral-pairs, we must also have that $E_{p,p}^*(B_i|_{\mathcal{S}_i})=E_{q,q}^*(B_i|_{\mathcal{S}_i})$, for all $p,q\in\mathcal{S}_i,$ and $E_{p,q}^*(B_i|_{\mathcal{S}_i})=0$, for all $p,q\in\mathcal{S}_i$ satisfying $p\prec_{\mathcal{S}_i}q.$ Then since the gluing rules in the construction of $\mathcal{Q}_i$ involve only the identification of extremal elements -- and 1 is a minimal element of $\mathcal{Q}_i$ -- Lemma~\ref{lem:extequal} implies that $E_{p,p}^*(B_i|_{\mathcal{S}_i})=E_{1,1}^*(B_i|_{\mathcal{Q}_{i-1}})=k_1E_{1,1}^*(B_{i-1})+k_2,$ for all $p\in\mathcal{S}_i;$ that is, $$B_i=k_1B_{i-1}+k_2I_{|\mathcal{Q}_{i-1}|}+\sum_{p\in\mathcal{S}_i-\mathcal{Q}_{i-1}}(k_1E_{1,1}^*(B_{i-1})+k_2)E_{p,p}.$$
By assumption, $B_{i-1}\in\mathfrak{g}_A(\mathcal{Q}_{i-1}),$ so $tr(k_1B_{i-1})=0,$ and since $B_i\in\mathfrak{g}_A(\mathcal{Q}_i),$ it must be the case that $$tr(B_i)=k_2|\mathcal{Q}_{i-1}|+|\mathcal{S}_i-\mathcal{Q}_{i-1}|(k_1E_{1,1}^*(B_{i-1})+k_2)=0.$$ Solving for $k_2,$ we have that $$k_2=-\frac{|\mathcal{S}_i-\mathcal{Q}_{i-1}|}{|\mathcal{Q}_{i-1}|+|\mathcal{S}_i-\mathcal{Q}_{i-1}|}k_1E_{1,1}^*(B_{i-1}).$$ Therefore, a choice of $k_1$ determines $B_i$ completely, and so $span\{B_i\}=\ker_A(d\varphi_{\mathcal{Q}_i}).$ Further, by Lemma~\ref{lem:2} and Corollary~\ref{cor:e11}, $$\varphi_{\mathcal{Q}_i}(B_i)=E_{1,1}^*(B_i)=E_{1,1}^*(k_1B_{i-1}+k_2I)=k_1E_{1,1}^*(B_{i-1})\left(1-\frac{|\mathcal{S}_i-\mathcal{Q}_{i-1}|}{|\mathcal{Q}_{i-1}|+|\mathcal{S}_i-\mathcal{Q}_{i-1}|}\right)\neq 0.$$ The result follows from Lemma~\ref{lem:kernel}.
\end{proof}


\noindent
Considering Remark~\ref{rem:h2} and Theorem~\ref{thm:main}, we are led to the following conjecture.

\begin{conj}
If $\mathcal{P}$ is a connected poset for which $\mathfrak{g}_A(\mathcal{P})$ is contact, then $\mathcal{P}$ is toral.
\end{conj}

\section{Epilogue}\label{sec:epilogue}
At its genesis, contact geometry was mainly viewed as the odd-dimensional analogue of symplectic geometry (see \textbf{\cite{Geiges}}). Though the field has outgrown its reliance on symplectic geometry and taken on a life of its own, the intertwining of symplectic and contact manifolds is still a major focus. For example, a recent result of Barajas et al. (\textbf{\cite{Barajas}}, 2019) states that each $2n-$dimensional Frobenius, i.e., exact symplectic, Lie algebra $\mathfrak{f}$ can be constructed by a procedure which ``extends" a $(2n-1)-$dimensional contact ideal $\mathfrak{c}\subset\mathfrak{f}.$ As a consequence, the classification of Frobenius Lie algebras will immediately follow from a classification of contact Lie algebras. On the other hand, the same authors show that a similar procedure for constructing contact Lie algebras from Frobenius ones is not comprehensive; that is, there are contact Lie algebras that do not have codimension one, Frobenius ideals. Thus, we have the following question:
\begin{center}
    \textit{Which Frobenius Lie algebras can be augmented to yield contact Lie algebras?}
\end{center}

One of the first appearances of the question above was in Arnold's text \textbf{\cite{Arnold}}, where he described a ``contactification" functor from the category of exact symplectic manifolds to the category of contact manifolds. This notion was later extended by Diatta \textbf{\cite{Diatta}} and relabeled ``contactization." Effectively, the work done in the present article provides insight into a ``combinatorial contactization" upon adjustment of the convention given in Remark~\ref{rem:order}. In particular,
one can construct a Frobenius, toral poset and then ``contactize" the associated Lie algebra by gluing on a poset from a contact toral-pair -- according to the rules given in Theorem~\ref{thm:big}. The benefit of this approach is that it appears to be the first instance of its kind to use combinatorics, rather than algebra or geometry.

A potential direction for further research related to the notion of combinatorial contactization is in Lie poset algebras of other types. In general, a Lie poset subalgebra of a classical, simple Lie algebra $\mathfrak{g}$ is one which lies between a Cartan subalgebra of $\mathfrak{g}$ and an associated Borel subalgebra of $\mathfrak{g}$ (see \textbf{\cite{CG}}). The recent article (\textbf{\cite{BCD}}, 2021) serves as an initial investigation into the index and spectra of such algebras. Utilizing the identification in \textbf{\cite{BCD}} of restricted families of Frobenius, type-B, C, and D Lie poset algebras -- and their associated type-B, C, and D posets -- one may be able to introduce new building blocks and gluing rules in order to generate contact, type-B, C, and D Lie poset algebras, respectively.

Lie poset algebras are further generalized by removing the presence of a Borel subalgebra from the definition. That is, we define a \textit{Lie proset subalgebra} of a classical, simple Lie algebra $\mathfrak{g}$ to be any subalgebra which contains a Cartan subalgebra of $\mathfrak{g}.$ Clearly the family of Lie poset algebras is contained in this new family, and moreover, Lie proset algebras correspond to ``pre-orderings" in the same way that Lie poset algebras correspond to partial orderings. A \textit{proset} (pre-ordered set) is a set with a binary relation that satisfies reflexivity and transitivity, and each finite proset can be represented by the reachability of a directed graph, or \textit{quiver}. A Hasse diagram is an acyclic quiver, so combinatorially, the extension from Lie poset algebras to Lie proset algebras can be described as simply allowing (directed) cycles in the associated combinatorial object. We conclude with the following question:

\begin{center}
    \textit{Can a Frobenius Lie poset algebra be combinatorially contactized by gluing directed cycles to the associated Hasse diagram, thus yielding a contact Lie proset algebra?}
\end{center}

\section{Appendix A: New Toral Pairs}\label{sec:appendixA}

This appendix contains the proof of Theorem~\ref{thm:ntp}.

\begin{theorem}\label{thm:ntp1}
Each of the following pairs, consisting of a poset $\mathcal{P}$ and a one-form $\varphi_{\mathcal{P}}$, constitutes a toral-pair $(\mathcal{P},\varphi_{\mathcal{P}})$.
\begin{enumerate} [label=\textup(\roman*\textup)]
        \item $\mathcal{P}_1=\{1,2,3,4,5,6\}$ with $1\prec 2,3\prec 4\prec 5,6$, and $$\varphi_{\mathcal{P}_1}=E^*_{1,5}+E^*_{1,6}+E^*_{2,4}+E^*_{2,5}+E^*_{3,6},$$
        \item $\mathcal{P}_1^*=\{1,2,3,4,5,6\}$ with $1,2\prec 3\prec 4,5\prec 6$, and $$\varphi_{\mathcal{P}_1^*}=E^*_{1,6}+E^*_{2,6}+E^*_{1,4}+E^*_{3,4}+E^*_{2,5}.$$
\end{enumerate}
\end{theorem}
\begin{proof}
We prove (\textit{i}), as (\textit{ii}) follows via a symmetric argument. Let $\mathcal{P}=\mathcal{P}_1$ and $\varphi=\varphi_{\mathcal{P}_1}$. Clearly, $|Ext(\mathcal{P})|=3,$ so \textbf{(P1)} of Definition~\ref{def:toralpair} is satisfied. Further, $\varphi$ is clearly small, $U_\varphi(\mathcal{P})$ is a filter of $\mathcal{P},$ $D_\varphi(\mathcal{P})=\{1,2,3\}$ is an ideal of $\mathcal{P}$, $O_\varphi(\mathcal{P})=\emptyset$, and $\Gamma_\varphi$ contains the only edges, $(1,5)$ and $(1,6)$, between elements of $Ext(\mathcal{P})$. Thus, $\varphi$ satisfies $\mathbf{(F1)-(F3)}$ of Definition~\ref{def:toralpair}. Assuming $\varphi$ is Frobenius, applying Theorem 2 of \textbf{\cite{Binary}} and Theorem 3.3(1) of \textbf{\cite{Ooms}}, it follows that $\mathcal{P}$ satisfies \textbf{(P2)} of Definition~\ref{def:toralpair}. Therefore, all that remains to show is that $\varphi$ satisfies \textbf{(F4)} of Definition~\ref{def:toralpair} and is Frobenius on $\mathfrak{g}_A(\mathcal{P})$.

Let $B=\sum k_{i,j}E_{i,j}\in\ker(d\varphi)$ and consider the following groups of conditions $B$ must satisfy:

\bigskip
\noindent
\textbf{Group 1:}
\begin{itemize}
    \item $\varphi([E_{3,3},B])=k_{3,6}=0$,
    \item $\varphi([E_{4,4},B])=-k_{2,4}=0$,
    \item $\varphi([E_{2,6},B])=-k_{1,2}=0$,
    \item $\varphi([E_{3,4},B])=k_{4,6}=0$,
    \item $\varphi([E_{3,5},B])=-k_{1,3}=0$.
\end{itemize}
\textbf{Group 2:}
\begin{itemize}
    \item $\varphi([E_{2,2},B])=k_{2,4}+k_{2,5}=0$,
    \item $\varphi([E_{6,6},B])=-k_{1,6}-k_{3,6}=0$,
    \item $\varphi([E_{1,3},B])=k_{3,5}-k_{3,6}=0$,
    \item $\varphi([E_{1,4},B])=k_{4,5}-k_{4,6}=0$,
    \item $\varphi([E_{4,5},B])=-k_{1,4}-k_{2,4}=0$.
\end{itemize}
\textbf{Group 3:}
\begin{itemize}
    \item $\varphi([E_{1,1},B])=k_{1,5}+k_{1,6}=0$,
    \item $\varphi([E_{5,5},B])=-k_{1,5}-k_{2,5}=0$,
    \item $\varphi([E_{1,2},B])=k_{2,5}+k_{2,6}=0$,
    \item $\varphi([E_{4,6},B])=-k_{1,4}-k_{3,4}=0$.
\end{itemize}
\textbf{Group 4:}
\begin{itemize}
    \item $\varphi([E_{1,5},B])=k_{5,5}-k_{1,1}=0$,
    \item $\varphi([E_{1,6},B])=k_{6,6}-k_{1,1}=0$,
    \item $\varphi([E_{2,4},B])=k_{4,4}-k_{2,2}+k_{4,5}=0$,
    \item $\varphi([E_{2,5},B])=k_{5,5}-k_{2,2}-k_{1,2}=0$,
    \item $\varphi([E_{3,6},B])=k_{6,6}-k_{3,3}-k_{1,3}=0$.
\end{itemize}
The restrictions of the equations in Group 1 immediately imply that
\begin{equation}\label{eqn:p11}
k_{1,2}=k_{1,3}=k_{2,4}=k_{3,6}=k_{4,6}=0.
\end{equation} 
Combining the Group 1 restrictions to those of Group 2, we may conclude that
\begin{equation}\label{eqn:p12}
k_{1,4}=k_{1,6}=k_{2,5}=k_{3,5}=k_{4,5}=0.
\end{equation} 
Combining the Group 1 and 2 restrictions to those of Group 3, we may conclude that
\begin{equation}\label{eqn:p13}
k_{1,5}=k_{2,6}=k_{3,4}=0.
\end{equation} 
Finally, combining the restrictions of Groups 1, 2, and 3 to those of Group 4, we find that
\begin{equation}\label{eqn:p14}
k_{i,i}=k_{j,j},\text{ for all }i,j\in\mathcal{P}.
\end{equation}
Equations (\ref{eqn:p11}) -- (\ref{eqn:p14}) establish that $\varphi$ satisfies \textbf{(F4)} of Definition~\ref{def:toralpair}. Now, considering Remark~\ref{rem:funct}, combining the type-A trace condition $$\sum_{i=1}^6k_{i,i}=0$$ with (\ref{eqn:p11}) -- (\ref{eqn:p14}), we find that $\ker_A(d\varphi)=\{0\}$. Therefore, $\mathfrak{g}_A(\mathcal{P})$ is Frobenius with Frobenius one-form $\varphi$, and $(\mathcal{P},\varphi)$ is a toral-pair.
\end{proof}

\begin{theorem}\label{thm:ntp2}
Each of the following pairs, consisting of a poset $\mathcal{P}$ and a one-form $\varphi_{\mathcal{P}}$, constitutes a toral-pair $(\mathcal{P},\varphi_{\mathcal{P}})$.
\begin{enumerate} [label=\textup(\roman*\textup)]
        \item $\mathcal{P}_2=\{1,2,3,4,5,6\}$ with $1\prec 2,3\prec 4\prec 5$; $3\prec 6$, and $$\varphi_{\mathcal{P}_2}=E^*_{1,5}+E^*_{1,6}+E^*_{2,4}+E^*_{3,4}+E^*_{3,6},$$
        \item $\mathcal{P}_2^*=\{1,2,3,4,5,6\}$ with $1\prec 3\prec 4,5\prec 6$; $2\prec 5$, and $$\varphi_{\mathcal{P}_2^*}=E^*_{1,6}+E^*_{2,6}+E^*_{2,5}+E^*_{3,4}+E^*_{3,5}.$$
\end{enumerate}
\end{theorem}
\begin{proof}
We prove (\textit{i}), as (\textit{ii}) follows via a symmetric argument. Let $\mathcal{P}=\mathcal{P}_2$ and $\varphi=\varphi_{\mathcal{P}_2}$. Clearly, $|Ext(\mathcal{P})|=3,$ so \textbf{(P1)} of Definition~\ref{def:toralpair} is satisfied. Further, $\varphi$ is clearly small, $U_\varphi(\mathcal{P})=\{4,5,6\}$ is a filter of $\mathcal{P}$, $D_\varphi(\mathcal{P})=\{1,2,3\}$ is an ideal of $\mathcal{P}$, $O_\varphi(\mathcal{P})=\emptyset$, and $\Gamma_\varphi$ contains the only edges, $(1,5)$ and $(1,6)$, between elements of $Ext(\mathcal{P})$. Thus, $\varphi$ satisfies $\mathbf{(F1)-(F3)}$ of Definition~\ref{def:toralpair}. Assuming $\varphi$ is Frobenius, applying Theorem 2 of \textbf{\cite{Binary}} and Theorem 3.3(1) of \textbf{\cite{Ooms}}, it follows that $\mathcal{P}$ satisfies \textbf{(P2)} of Definition~\ref{def:toralpair}. Therefore, all that remains to show is that $\varphi$ satisfies \textbf{(F4)} of Definition~\ref{def:toralpair} and is Frobenius on $\mathfrak{g}_A(\mathcal{P})$.

Let $B=\sum k_{i,j}E_{i,j}\in\ker(d\varphi)$ and consider the following groups of conditions $B$ must satisfy:
\bigskip

\noindent
\textbf{Group 1:}
\begin{itemize}
    \item $\varphi([E_{2,2},B])=k_{2,4}=0$,
    \item $\varphi([E_{5,5},B])=-k_{1,5}=0$,
    \item $\varphi([E_{1,2},B])=k_{2,5}=0$,
    \item $\varphi([E_{1,4},B])=k_{4,5}=0$,
    \item $\varphi([E_{2,5},B])=-k_{1,2}=0$,
    \item $\varphi([E_{3,5},B])=-k_{1,3}=0$,
    \item $\varphi([E_{4,5},B])=-k_{1,4}=0$.
\end{itemize}
\textbf{Group 2:}
\begin{itemize}
    \item $\varphi([E_{1,1},B])=k_{1,5}+k_{1,6}=0$,
    \item $\varphi([E_{4,4},B])=-k_{2,4}-k_{3,4}=0$.
\end{itemize}
\textbf{Group 3:}
\begin{itemize}
    \item $\varphi([E_{3,3},B])=k_{3,4}+k_{3,6}=0$,
    \item $\varphi([E_{6,6},B])=-k_{1,6}-k_{3,6}=0$.
\end{itemize}
\textbf{Group 4:}
\begin{itemize}
    \item $\varphi([E_{1,3},B])=k_{3,5}+k_{3,6}=0$.
\end{itemize}
\textbf{Group 5:}
\begin{itemize}
    \item $\varphi([E_{1,5},B])=k_{5,5}-k_{1,1}=0$,
    \item $\varphi([E_{1,6},B])=k_{6,6}-k_{1,1}=0$,
    \item $\varphi([E_{2,4},B])=k_{4,4}-k_{2,2}=0$,
    \item $\varphi([E_{3,4},B])=k_{4,4}-k_{3,3}=0$,
    \item $\varphi([E_{3,6},B])=k_{6,6}-k_{3,3}-k_{1,3}=0$.
\end{itemize}
The restrictions of the equations in Group 1 immediately imply that
\begin{equation}\label{eqn:p21}
k_{1,2}=k_{1,3}=k_{1,4}=k_{1,5}=k_{2,4}=k_{2,5}=k_{4,5}=0.
\end{equation} 
Combining the Group 1 restrictions to those of Group 2, we may conclude that
\begin{equation}\label{eqn:p22}
k_{1,6}=k_{3,4}=0.
\end{equation} 
Combining the Group 1 and 2 restrictions to those of Group 3, we may conclude that
\begin{equation}\label{eqn:p23}
k_{3,6}=0.
\end{equation} 
Combining the Group 1, 2, and 3 restrictions to those of Group 4, we may conclude that
\begin{equation}\label{eqn:p24}
k_{3,5}=0.
\end{equation} 
Finally, combining the restrictions of Groups 1, 2, 3, and 4 to those of Group 5, we find that
\begin{equation}\label{eqn:p25}
k_{i,i}=k_{j,j},\text{ for all }i,j\in\mathcal{P}.
\end{equation}
Equations (\ref{eqn:p21}) -- (\ref{eqn:p25}) establish that $\varphi$ satisfies \textbf{(F4)} of Definition~\ref{def:toralpair}. Now, considering Remark~\ref{rem:funct}, combining the type-A trace condition $$\sum_{i=1}^6k_{i,i}=0$$ with (\ref{eqn:p21}) -- (\ref{eqn:p25}), we find that $\ker_A(d\varphi)=\{0\}$. Therefore, $\mathfrak{g}_A(\mathcal{P})$ is Frobenius with Frobenius one-form $\varphi$, and $(\mathcal{P},\varphi)$ is a toral-pair.
\end{proof}

\begin{theorem}\label{thm:ntp3}
Each of the following pairs, consisting of a poset $\mathcal{P}$ and a one-form $\varphi_{\mathcal{P}}$, constitutes a toral-pair $(\mathcal{P},\varphi_{\mathcal{P}})$.
\begin{enumerate} [label=\textup(\roman*\textup)]
        \item $\mathcal{P}_3=\{1,2,3,4,5,6\}$ with $1\prec 2\prec 3, 4\prec 5$; $2\prec 6$, and $$\varphi_{\mathcal{P}_3}=E^*_{1,5}+E^*_{1,6}+E^*_{2,3}+E^*_{2,4}+E^*_{2,5},$$
        \item $\mathcal{P}_3^*=\{1,2,3,4,5,6\}$ with $1\prec 3, 4\prec 5\prec 6$; $2\prec 5$, and $$\varphi_{\mathcal{P}_3^*}=E^*_{1,6}+E^*_{2,6}+E^*_{1,5}+E^*_{3,5}+E^*_{4,5}.$$
\end{enumerate}
\end{theorem}
\begin{proof}
We prove (\textit{i}), as (\textit{ii}) follows via a symmetric argument. Let $\mathcal{P}=\mathcal{P}_3$ and $\varphi=\varphi_{\mathcal{P}_3}$. Clearly, $|Ext(\mathcal{P})|=3,$ so \textbf{(P1)} of Definition~\ref{def:toralpair} is satisfied. Further, $\varphi$ is clearly small, $U_\varphi(\mathcal{P})=\{3,4,5,6\}$ is a filter of $\mathcal{P}$, $D_\varphi(\mathcal{P})=\{1,2\}$ is an ideal of $\mathcal{P}$, $O_\varphi(\mathcal{P})=\emptyset$, and $\Gamma_\varphi$ contains the only edges, $(1,5)$ and $(1,6)$, between elements of $Ext(\mathcal{P})$. Thus, $\varphi$ satisfies $\mathbf{(F1)-(F3)}$ of Definition~\ref{def:toralpair}. Assuming $\varphi$ is Frobenius, applying Theorem 2 of \textbf{\cite{Binary}} and Theorem 3.3(1) of \textbf{\cite{Ooms}}, it follows that $\mathcal{P}$ satisfies \textbf{(P2)} of Definition~\ref{def:toralpair}. Therefore, all that remains to show is that $\varphi$ satisfies \textbf{(F4)} of Definition~\ref{def:toralpair} and is Frobenius on $\mathfrak{g}_A(\mathcal{P})$.

Let $B=\sum k_{i,j}E_{i,j}\in\ker(d\varphi)$ and consider the following groups of conditions $B$ must satisfy:
\bigskip

\noindent
\textbf{Group 1:}
\begin{itemize}
    \item $\varphi([E_{3,3},B])=-k_{2,3}=0$,
    \item $\varphi([E_{4,4},B])=-k_{2,4}=0$,
    \item $\varphi([E_{6,6},B])=-k_{1,6}=0$,
    \item $\varphi([E_{1,3},B])=k_{3,5}=0$,
    \item $\varphi([E_{1,4},B])=k_{4,5}=0$,
    \item $\varphi([E_{2,6},B])=-k_{1,2}=0$.
\end{itemize}
\textbf{Group 2:}
\begin{itemize}
    \item $\varphi([E_{1,1},B])=k_{1,5}+k_{1,6}=0$, 
    \item $\varphi([E_{2,2},B])=k_{2,3}+k_{2,4}+k_{2,5}=0$,
    \item $\varphi([E_{3,5},B])=-k_{1,3}-k_{2,3}=0$,
    \item $\varphi([E_{4,5},B])=-k_{1,4}-k_{2,4}=0$,
\end{itemize}
\textbf{Group 3:}
\begin{itemize}
    \item $\varphi([E_{5,5},B])=-k_{1,5}-k_{2,5}=0$,
    \item $\varphi([E_{1,2},B])=k_{2,5}+k_{2,6}=0$,
\end{itemize}
\textbf{Group 4:}
\begin{itemize}
    \item $\varphi([E_{1,5},B])=k_{5,5}-k_{1,1}=0$,
    \item $\varphi([E_{1,6},B])=k_{6,6}-k_{1,1}=0$,
    \item $\varphi([E_{2,3},B])=k_{3,3}-k_{2,2}+k_{3,5}=0$,
    \item $\varphi([E_{2,4},B])=k_{4,4}-k_{2,2}+k_{4,5}=0$,
    \item $\varphi([E_{2,5},B])=k_{5,5}-k_{2,2}+k_{1,2}=0$,
\end{itemize}
The restrictions of the equations in Group 1 immediately imply that
\begin{equation}\label{eqn:p31}
k_{1,2}=k_{1,6}=k_{2,3}=k_{2,4}=k_{3,5}=k_{4,5}=0.
\end{equation} 
Combining the Group 1 restrictions to those of Group 2, we may conclude that
\begin{equation}\label{eqn:p32}
k_{1,3}=k_{1,4}=k_{1,5}=k_{2,5}=0.
\end{equation} 
Combining the Group 1 and 2 restrictions to those of Group 3, we may conclude that
\begin{equation}\label{eqn:p33}
k_{2,6}=0.
\end{equation} 
Finally, combining the restrictions of Groups 1, 2, and 3 to those of Group 4, we find that
\begin{equation}\label{eqn:p34}
k_{i,i}=k_{j,j},\text{ for all }i,j\in\mathcal{P}.
\end{equation}
Equations (\ref{eqn:p31}) -- (\ref{eqn:p34}) establish that $\varphi$ satisfies \textbf{(F4)} of Definition~\ref{def:toralpair}. Now, considering Remark~\ref{rem:funct}, combining the type-A trace condition $$\sum_{i=1}^6k_{i,i}=0$$ with (\ref{eqn:p31}) -- (\ref{eqn:p34}), we find that $\ker_A(d\varphi)=\{0\}$. Therefore, $\mathfrak{g}_A(\mathcal{P})$ is Frobenius with Frobenius one-form $\varphi$, and $(\mathcal{P},\varphi)$ is a toral-pair.
\end{proof}

\begin{theorem}\label{thm:ntp4}
Each of the following pairs, consisting of a poset $\mathcal{P}$ and a one-form $\varphi_{\mathcal{P}}$, constitutes a toral-pair $(\mathcal{P},\varphi_{\mathcal{P}})$.
\begin{enumerate} [label=\textup(\roman*\textup)]
        \item $\mathcal{P}_4=\{1,2,3,4,5,6\}$ with $1\prec 2, 3\prec 6$; $3\prec 4,5$; $2\prec 4$, and $$\varphi_{\mathcal{P}_4}=E^*_{1,4}+E^*_{1,6}+E^*_{2,4}+E^*_{2,5}+E^*_{3,6},$$
        \item $\mathcal{P}_4^*=\{1,2,3,4,5,6\}$ with $1\prec 4, 5\prec 6$; $2,3\prec 4$; $3\prec 5$, and $$\varphi_{\mathcal{P}_4^*}=E^*_{1,5}+E^*_{1,6}+E^*_{2,4}+E^*_{3,4}+E^*_{3,6}.$$
\end{enumerate}
\end{theorem}
\begin{proof}
We prove (\textit{i}), as (\textit{ii}) follows via a symmetric argument. Let $\mathcal{P}=\mathcal{P}_4$ and $\varphi=\varphi_{\mathcal{P}_4}$. Clearly, $|Ext(\mathcal{P})|=3,$ so \textbf{(P1)} of Definition~\ref{def:toralpair} is satisfied. Further, $\varphi$ is clearly small, $U_\varphi(\mathcal{P})=\{4,5,6\}$ is a filter of $\mathcal{P}$, $D_\varphi(\mathcal{P})=\{1,2,3\}$ is an ideal of $\mathcal{P}$, $O_\varphi(\mathcal{P})=\emptyset$, and $\Gamma_\varphi$ contains the only edges, $(1,4)$ and $(1,6)$, between elements of $Ext(\mathcal{P})$. Thus, $\varphi$ satisfies $\mathbf{(F1)-(F3)}$ of Definition~\ref{def:toralpair}. Assuming $\varphi$ is Frobenius, applying Theorem 2 of \textbf{\cite{Binary}} and Theorem 3.3(1) of \textbf{\cite{Ooms}}, it follows that $\mathcal{P}$ satisfies \textbf{(P2)} of Definition~\ref{def:toralpair}. Therefore, all that remains to show is that $\varphi$ satisfies \textbf{(F4)} of Definition~\ref{def:toralpair} and is Frobenius on $\mathfrak{g}_A(\mathcal{P})$.

Let $B=\sum k_{i,j}E_{i,j}\in\ker(d\varphi)$ and consider the following groups of conditions $B$ must satisfy:
\bigskip

\noindent
\textbf{Group 1:}
\begin{itemize}
    \item $\varphi([E_{3,3},B])=k_{3,6}=0$,
    \item $\varphi([E_{5,5},B])=-k_{2,5}=0$,
    \item $\varphi([E_{1,5},B])=k_{5,6}=0$,
    \item $\varphi([E_{2,6},B])=-k_{1,2}=0$,
    \item $\varphi([E_{3,4},B])=-k_{1,3}=0$,
    \item $\varphi([E_{5,6},B])=-k_{1,5}=0$,
\end{itemize}
\textbf{Group 2:}
\begin{itemize}
    \item $\varphi([E_{2,2},B])=k_{2,4}+k_{2,5}=0$, 
    \item $\varphi([E_{6,6},B])=-k_{1,6}-k_{3,6}=0$, 
    \item $\varphi([E_{1,3},B])=k_{3,4}+k_{3,6}=0$,
\end{itemize}
\textbf{Group 3:}
\begin{itemize}
    \item $\varphi([E_{1,1},B])=k_{1,4}+k_{1,6}=0$, 
    \item $\varphi([E_{4,4},B])=-k_{1,4}-k_{2,4}=0$,
    \item $\varphi([E_{1,2},B])=k_{2,4}+k_{2,6}=0$, 
\end{itemize}
\textbf{Group 4:}
\begin{itemize}
    \item $\varphi([E_{1,4},B])=k_{4,4}-k_{1,1}=0$,
    \item $\varphi([E_{1,6},B])=k_{6,6}-k_{1,1}=0$,
    \item $\varphi([E_{2,4},B])=k_{4,4}-k_{2,2}-k_{1,2}=0$,
    \item $\varphi([E_{2,5},B])=k_{5,5}-k_{2,2}-k_{1,2}=0$,
    \item $\varphi([E_{3,6},B])=k_{6,6}-k_{3,3}-k_{1,3}=0$,
\end{itemize}
The restrictions of the equations in Group 1 immediately imply that
\begin{equation}\label{eqn:p41}
k_{1,2}=k_{1,3}=k_{1,5}=k_{2,5}=k_{3,6}=k_{5,6}=0.
\end{equation} 
Combining the Group 1 restrictions to those of Group 2, we may conclude that
\begin{equation}\label{eqn:p42}
k_{1,6}=k_{2,4}=k_{3,4}=0.
\end{equation} 
Combining the Group 1 and 2 restrictions to those of Group 3, we may conclude that
\begin{equation}\label{eqn:p43}
k_{1,4}=k_{2,6}=0.
\end{equation} 
Finally, combining the restrictions of Groups 1, 2, and 3 to those of Group 4, we find that
\begin{equation}\label{eqn:p44}
k_{i,i}=k_{j,j},\text{ for all }i,j\in\mathcal{P}.
\end{equation}
Equations (\ref{eqn:p41}) -- (\ref{eqn:p44}) establish that $\varphi$ satisfies \textbf{(F4)} of Definition~\ref{def:toralpair}. Now, considering Remark~\ref{rem:funct}, combining the type-A trace condition $$\sum_{i=1}^6k_{i,i}=0$$ with (\ref{eqn:p41}) -- (\ref{eqn:p44}), we find that $\ker_A(d\varphi)=\{0\}$. Therefore, $\mathfrak{g}_A(\mathcal{P})$ is Frobenius with Frobenius one-form $\varphi$, and $(\mathcal{P},\varphi)$ is a toral-pair.
\end{proof}

\section{Appendix B}

This appendix contains the proof of Theorem~\ref{thm:ctpappendB}.

\begin{theorem}
Each of the following pairs, consisting of a poset $\mathcal{P}$ and a one-form $\varphi_{\mathcal{P}},$ constitutes a contact toral-pair $(\mathcal{P},\varphi_{\mathcal{P}}).$
\begin{enumerate}
    \item[\textup(i\textup)] $\mathcal{P}_{4,n}=\{1,\dots,n\}$ with $1\preceq 2\preceq \dots\preceq n-1$ as well as $1\preceq 2\preceq \dots\preceq \lfloor\frac{n}{2}\rfloor+1\preceq n,$ and $$\varphi_{\mathcal{P}_{4,n}}=E_{1,1}^*+\sum_{i=1}^{\lfloor\frac{n-1}{2}\rfloor}E_{i,n-i}^*+\sum_{i=1}^{\lfloor\frac{n}{2}\rfloor}E_{i,n}^*$$
    \item[\textup(ii\textup)] $\mathcal{P}_{4,n}^*=\{1,\dots,n\}$ with $2\preceq 3\preceq \dots\preceq n$ as well as $1\preceq\lceil\frac{n}{2}\rceil\preceq \lceil\frac{n}{2}\rceil+1\preceq\dots\preceq n,$ and $$\varphi_{\mathcal{P}_{4,n}^*}=E_{1,1}^*+\sum_{i=2}^{\lceil\frac{n}{2}\rceil}E_{i,n-i+2}^*+\sum_{i=\lceil\frac{n}{2}\rceil+1}^{n}E_{1,i}^*.$$
\end{enumerate}
\end{theorem}
\begin{proof}
We prove ($i$), as ($ii$) follows via a symmetric argument. Let $\mathcal{P}=\mathcal{P}_{4,n}$ and $\varphi=\varphi_{\mathcal{P}_{4,n}}.$ Note immediately that $\mathcal{P}$ is connected and $|Ext(\mathcal{P})|=3,$ so $\mathbf{(CP1)}$ and $\mathbf{(CP2)}$ are satisfied. Now, with regard to $\varphi,$ we have the following:
\begin{itemize}
    \item $E_{p,p}^*$ is a nonzero summand of $\varphi$ precisely when $p=1,$
    \item $\varphi-E_{1,1}^*$ is clearly small,
    \item $U_{\varphi-E_{1,1}^*}(\mathcal{P})=\left[\lfloor\frac{n}{2}\rfloor+1,n\right]\cap\mathbb{Z}$ is a filter of $\mathcal{P}$, $D_{\varphi-E_{1,1}^*}(\mathcal{P})=\left[1,\lfloor\frac{n}{2}\rfloor\right]\cap\mathbb{Z}$ is an order ideal of $\mathcal{P}$, and $O_{\varphi-E_{1,1}^*}(\mathcal{P})=\emptyset,$ and
    \item $\Gamma_{\varphi-E_{1,1}^*}$ contains the only edges, $(1,n-1)$ and $(1,n),$ between elements of $Ext(\mathcal{P}).$
\end{itemize}
Thus, $\mathbf{(CF1)-(CF4)}$ are satisfied, and all that remains to show is that $\varphi$ is a contact form. We claim that the kernel of $d\varphi$ is one-dimensional and generated by an element of the following form:
$$B=\begin{cases}
\displaystyle\sum_{i=1}^{n}E_{i,i} - \frac{n}{2}(E_{\frac{n}{2},\frac{n}{2}}+E_{\frac{n}{2}+1,\frac{n}{2}+1})+\frac{n}{2}E_{\frac{n}{2}-1,\frac{n}{2}}-\sum_{i=\frac{n}{2}+2}^{n-1}\frac{n}{2}E_{\frac{n}{2}+1,i}+\frac{n}{2}E_{\frac{n}{2}+1,n}, & n\text{\ even};\\
\displaystyle\sum_{i=1}^{n}E_{i,i}-nE_{\lfloor\frac{n}{2}\rfloor+1,\lfloor\frac{n}{2}\rfloor+1}-\sum_{i=\lfloor\frac{n}{2}\rfloor+2}^{n-1}nE_{\lfloor\frac{n}{2}\rfloor+1,i}+nE_{\lfloor\frac{n}{2}\rfloor+1,n}, & n\text{\ odd}.
\end{cases}$$
Note that $\varphi(B)\neq 0$ in both cases, so once the claim has been proven, an application of Lemma~\ref{lem:kernel} will complete the proof.

\medskip
\noindent
\textbf{Case 1:} $n$ even. Let $B=\sum k_{i,j}E_{i,j}\in\ker(d\varphi),$ and consider the following groups of conditions $B$ must satisfy:

\bigskip
\noindent
\textbf{Group 1:}
\begin{itemize}
    \item $\varphi([E_{\frac{n}{2},\frac{n}{2}},B])=k_{\frac{n}{2},n}=0,$
    \item $\varphi([E_{n-i,n-i},B])=-k_{i,n-i}=0,$ for $1\leq i\leq \frac{n}{2}-1,$
    \item $\varphi([E_{n-j,n-i},B])=-k_{i,n-j}=0,$ for $1\leq i<j\leq \frac{n}{2}-1,$
    \item $\varphi([E_{i,n-j},B])=k_{n-j,n-i}=0,$ for $1\leq i<j\leq \frac{n}{2}-2,$
    \item $\varphi([E_{i,n-j},B])=-k_{j,i}=0,$ for $1\leq j<i\leq \frac{n}{2}-1.$
\end{itemize}
    
\bigskip
\noindent
\textbf{Group 2:}
\begin{itemize}
    \item $\varphi([E_{n,n},B])=-\displaystyle\sum_{i=1}^{\frac{n}{2}}k_{i,n}=0,$
    \item $\varphi([E_{i,i},B])=k_{i,n-i}+k_{i,n}=0,$ for $1\leq i\leq \frac{n}{2}-1,$
    \item $\varphi([E_{\frac{n}{2}+1,n},B])=-\displaystyle\sum_{i=1}^{\frac{n}{2}}k_{i,\frac{n}{2}+1}=0,$
    \item $\varphi([E_{i,j},B])=k_{j,n-i}+k_{j,n}=0,$ for $1\leq i<j\leq \frac{n}{2}.$
\end{itemize}

\bigskip
\noindent
\textbf{Group 3:}
\begin{itemize}
    \item $\varphi([E_{i,n-i},B])=k_{n-i,n-i}-k_{i,i}=0,$ for $1\leq i\leq \frac{n}{2}-2,$
    \item $\varphi([E_{\frac{n}{2}-1,\frac{n}{2}+1},B])=k_{\frac{n}{2}+1,\frac{n}{2}+1}-k_{\frac{n}{2}-1,\frac{n}{2}-1}+k_{\frac{n}{2}+1,n}=0,$
    \item $\varphi([E_{i,\frac{n}{2}+1},B])=k_{\frac{n}{2}+1,n-i}+k_{\frac{n}{2}+1,n}=0,$ for $1\leq i\leq \frac{n}{2}-2,$
    \item $\varphi([E_{\frac{n}{2},\frac{n}{2}+1},B])=-k_{\frac{n}{2}-1,\frac{n}{2}}+k_{\frac{n}{2}+1,n}=0,$
    \item $\varphi([E_{i,n},B])=k_{n,n}-k_{i,i}-\displaystyle\sum_{j=1}^{i-1}k_{j,i}=0,$ for $1\leq i\leq \frac{n}{2}.$
\end{itemize}

The restrictions of Groups 1 and 2 imply that $B$ has the form $$B=\sum_{i=1}^nk_{i,i}E_{i,i}+\sum_{i=\frac{n}{2}+2}^nk_{\frac{n}{2}+1,i}E_{\frac{n}{2}+1,i}+k_{\frac{n}{2}-1,\frac{n}{2}}E_{\frac{n}{2}-1,\frac{n}{2}}.$$
Simplifying the restrictions of Group 3, we have that \begin{align}
k_{i,i}&=k_{n-i,n-i},\text{\ for\ }1\leq i\leq\frac{n}{2}-2,\\
k_{\frac{n}{2}+1,n-i}&=-k_{\frac{n}{2}+1,n},\text{\ for\ }1\leq i\leq\frac{n}{2}-2,\\
k_{n,n}=k_{i,i}&=k_{\frac{n}{2},\frac{n}{2}}+k_{\frac{n}{2}-1,\frac{n}{2}},\text{\ for\ }1\leq i\leq \frac{n}{2}-1,\\
k_{\frac{n}{2}-1,\frac{n}{2}-1}&=k_{\frac{n}{2}+1,\frac{n}{2}+1}+k_{\frac{n}{2}+1,n},\\
k_{\frac{n}{2}+1,n}&=k_{\frac{n}{2}-1,\frac{n}{2}}.
\end{align}
Equations (18) and (20) imply that $k_{i,i}=k_{\frac{n}{2},\frac{n}{2}}+k_{\frac{n}{2}-1,\frac{n}{2}},$ for all $i\in\left(\left[1,\frac{n}{2}-1\right]\cup\left[\frac{n}{2}+2,n\right]\right)\cap\mathbb{Z}.$\\ Equations (20), (21), and (22) imply that $k_{\frac{n}{2},\frac{n}{2}}=k_{\frac{n}{2}+1,\frac{n}{2}+1}.$\\ Equations (19) and (22) imply that $k_{\frac{n}{2},n}=k_{\frac{n}{2}-1,\frac{n}{2}}=-k_{\frac{n}{2}+1,n-i},$ for all $1\leq i\leq \frac{n}{2}-2.$\\ Combining these conditions with the type-A trace condition $\sum_{i=1}^nk_{i,i}=0,$ we choose $k_{i,i}=1,$ for all $i\in\left(\left[1,\frac{n}{2}-1\right]\cup\left[\frac{n}{2}+2,n\right]\right)\cap\mathbb{Z},$ which then determines $k_{\frac{n}{2},\frac{n}{2}}=k_{\frac{n}{2}+1,\frac{n}{2}+1}=-\frac{n}{2}+1$ and $k_{\frac{n}{2}-1,\frac{n}{2}}=-k_{\frac{n}{2}+1,i}=k_{\frac{n}{2},n}=\frac{n}{2},$ for all $\frac{n}{2}+2\leq i\leq n-1.$ Therefore, $B$ can be written as in the claim.

\bigskip
\noindent
\textbf{Case 2:} $n$ odd. Let $B=\sum k_{i,j}E_{i,j}\in\ker(d\varphi),$ and consider the following groups of conditions $B$ must satisfy:

\bigskip
\noindent
\textbf{Group 1:}
\begin{itemize}
    \item $\varphi([E_{n-i,n-i},B])=-k_{i,n-i}=0,$ for $1\leq i\leq\lfloor\frac{n}{2}\rfloor,$
    \item $\varphi([E_{n-j,n-i},B])=-k_{i,n-j}=0,$ for $1\leq i<j\leq\lfloor\frac{n}{2}\rfloor,$
    \item $\varphi([E_{i,n-j},B])=k_{n-j,n-i}=0,$ for $1\leq i<j\leq\lfloor\frac{n}{2}\rfloor-1,$
    \item $\varphi([E_{i,n-j},B])=-k_{j,i}=0,$ for $1\leq j<i\leq \lfloor\frac{n}{2}\rfloor.$
\end{itemize}

\bigskip
\noindent
\textbf{Group 2:}
\begin{itemize}
    \item $\varphi([E_{n,n},B])=-\displaystyle\sum_{i=1}^{\lfloor\frac{n}{2}\rfloor}k_{i,n}=0,$
    \item $\varphi([E_{i,i},B])=k_{i,n-i}+k_{i,n}=0,$ for $1\leq i\leq\lfloor\frac{n}{2}\rfloor,$
    \item $\varphi([E_{\lfloor\frac{n}{2}\rfloor+1,n},B])=-\displaystyle\sum_{i=1}^{\lfloor\frac{n}{2}\rfloor}k_{i,\lfloor\frac{n}{2}\rfloor+1}=0,$
    \item $\varphi([E_{i,j},B])=k_{j,n-i}+k_{j,n}=0,$ for $1\leq i<j\leq \lfloor\frac{n}{2}\rfloor.$
\end{itemize}

\bigskip
\noindent
\textbf{Group 3:}
\begin{itemize}
    \item $\varphi([E_{i,n-i},B])=k_{n-i,n-i}-k_{i,i}=0,$ for $1\leq i\leq \lfloor\frac{n}{2}\rfloor-1,$
    \item $\varphi([E_{\lfloor\frac{n}{2}\rfloor,\lfloor\frac{n}{2}\rfloor+1},B])=k_{\lfloor\frac{n}{2}\rfloor+1,\lfloor\frac{n}{2}\rfloor+1}-k_{\lfloor\frac{n}{2}\rfloor,\lfloor\frac{n}{2}\rfloor}+k_{\lfloor\frac{n}{2}\rfloor+1,n}=0,$
    \item $\varphi([E_{i,\lfloor\frac{n}{2}\rfloor+1},B])=k_{\lfloor\frac{n}{2}\rfloor+1,n-i}+k_{\lfloor\frac{n}{2}\rfloor+1,n}=0,$ for $1\leq i\leq \lfloor\frac{n}{2}\rfloor-1,$
    \item $\varphi([E_{i,n},B])=k_{n,n}-k_{i,i}-\displaystyle\sum_{j=1}^{i-1}k_{j,i}=0,$ for $1\leq i\leq \lfloor\frac{n}{2}\rfloor.$
\end{itemize}

The restrictions of Groups 1 and 2 imply that $B$ has the form $$B=\sum_{i=1}^nk_{i,i}E_{i,i}+\sum_{i=\lfloor\frac{n}{2}\rfloor+2}^nk_{\lfloor\frac{n}{2}\rfloor+1,i}E_{\lfloor\frac{n}{2}\rfloor+1,i}.$$ Simplifying the restrictions of Group 3, we have that
\begin{align}
    k_{i,i}&=k_{n-i,n-i},\text{\ for\ }1\leq i\leq\Big\lfloor\frac{n}{2}\Big\rfloor-1,\\
    k_{\lfloor\frac{n}{2}\rfloor+1,n-i}&=-k_{\lfloor\frac{n}{2}\rfloor+1,n},\text{\ for\ }1\leq i\leq \Big\lfloor\frac{n}{2}\Big\rfloor-1,\\
    k_{n,n}&=k_{i,i},\text{\ for\ }1\leq i\leq \Big\lfloor\frac{n}{2}\Big\rfloor,\\
    k_{\lfloor\frac{n}{2}\rfloor,\lfloor\frac{n}{2}\rfloor}&=k_{\lfloor\frac{n}{2}\rfloor+1,\lfloor\frac{n}{2}\rfloor+1}+k_{\lfloor\frac{n}{2}\rfloor+1,n}.
\end{align}

\noindent
Equations (23), (25), and (26) imply that $k_{i,i}=k_{\lfloor\frac{n}{2}\rfloor+1,\lfloor\frac{n}{2}\rfloor+1}+k_{\lfloor\frac{n}{2}\rfloor+1,n},$ for all $i\in\left(\left[1,\lfloor\frac{n}{2}\rfloor\right]\cup\left[\lfloor\frac{n}{2}\rfloor+2,n\right]\right)~\cap~\mathbb{Z}.$\\
Combining this condition with Equation (24) and the type-A trace condition $\sum_{i=1}^nk_{i,i}=0,$ we choose $k_{i,i}=1,$ for all $i\in\left(\left[1,\lfloor\frac{n}{2}\rfloor\right]\cup\left[\lfloor\frac{n}{2}\rfloor+2,n\right]\right)\cap\mathbb{Z},$ which then determines $k_{\lfloor\frac{n}{2}\rfloor+1,\lfloor\frac{n}{2}\rfloor+1}=-n+1$ and $k_{\lfloor\frac{n}{2}\rfloor+1,n}=-k_{\lfloor\frac{n}{2}\rfloor+1,i}=n,$ for $\lfloor\frac{n}{2}\rfloor+2\leq i\leq n-1.$ Therefore, $B$ can be written as in the claim, and this completes the proof.
\end{proof}

\begin{theorem}
Each of the following pairs, consisting of a poset $\mathcal{P}$ and a one-form $\varphi_{\mathcal{P}},$ constitutes a contact toral-pair $(\mathcal{P},\varphi_{\mathcal{P}}).$
\begin{enumerate}
    \item[\textup(i\textup)] $\mathcal{P}_{5,n}=\{1,\dots,n\}$ with $1\preceq 2\preceq\dots\preceq n-1$ as well as $1\preceq 2\preceq\dots\preceq \lfloor\frac{n}{2}\rfloor-1\preceq n,$ and $$\varphi_{\mathcal{P}_{5,n}}=E_{1,1}^*+\sum_{i=1}^{\lfloor\frac{n-1}{2}\rfloor}E_{i,n-i}^*+\sum_{i=1}^{\lfloor\frac{n}{2}\rfloor-1}E_{i,n}^*+E_{\lfloor\frac{n}{2}\rfloor,n-1}^*.$$
    \item[\textup(ii\textup)] $\mathcal{P}_{5,n}^*=\{1,\dots,n\}$ with $2\preceq 3\preceq\dots\preceq n$ as well as $1\preceq \lceil\frac{n}{2}\rceil+2\preceq\lceil\frac{n}{2}\rceil+3\preceq\dots\preceq n,$ and $$\varphi_{\mathcal{P}_{5,n}^*}=E_{1,1}^*+\sum_{i=2}^{\lceil\frac{n}{2}\rceil}E_{i,n-i+2}^*+\sum_{i=\lceil\frac{n}{2}\rceil+2}^{n}E_{1,i}^*+E_{2,\lceil\frac{n}{2}\rceil+1}.$$
\end{enumerate}
\end{theorem}
\begin{proof}
We prove ($i$), as ($ii$) follows via a symmetric argument. Let $\mathcal{P}=\mathcal{P}_{5,n}$ and $\varphi=\varphi_{\mathcal{P}_{5,n}}.$ Note immediately that $\mathcal{P}$ is connected and $|Ext(\mathcal{P})|=3,$ so $\mathbf{(CP1)}$ and $\mathbf{(CP2)}$ are satisfied. Now, with regard to $\varphi,$ we have the following:
\begin{itemize}
    \item $E_{p,p}^*$ is a nonzero summand of $\varphi$ precisely when $p=1,$
    \item $\varphi-E_{1,1}^*$ is clearly small,
    \item $U_{\varphi-E_{1,1}^*}(\mathcal{P})=\left[\lfloor\frac{n}{2}\rfloor+1,n\right]\cap\mathbb{Z}$ forms a filter of $\mathcal{P}$, $D_{\varphi-E_{1,1}^*}(\mathcal{P})=\left[1,\lfloor\frac{n}{2}\rfloor\right]\cap\mathbb{Z}$ forms an order ideal of $\mathcal{P}$, and $O_{\varphi-E_{1,1}^*}(\mathcal{P})=\emptyset,$ and
    \item $\Gamma_{\varphi-E_{1,1}^*}$ contains the only edges, $(1,n-1)$ and $(1,n),$ between elements of $Ext(\mathcal{P}).$
\end{itemize}
Thus, $\mathbf{(CF1)-(CF4)}$ are satisfied, and all that remains to show is that $\varphi$ is a contact form. We claim that the kernel of $d\varphi$ is one-dimensional and generated by an element of the following form:
$$B=\begin{cases}
\displaystyle\sum_{i=1}^{n}E_{i,i}-nE_{\frac{n}{2},\frac{n}{2}}+nE_{1,\frac{n}{2}}, & n\text{\ even};\\
\displaystyle\sum_{i=1}^{n}E_{i,i}-\frac{n}{2}(E_{\lfloor\frac{n}{2}\rfloor,\lfloor\frac{n}{2}\rfloor}+E_{\lfloor\frac{n}{2}\rfloor+1,\lfloor\frac{n}{2}\rfloor+1})+\frac{n}{2}E_{1,\lfloor\frac{n}{2}\rfloor}, & n\text{\ odd}.
\end{cases}$$
Note that $\varphi(B)\neq 0$ in both cases, so once the claim has been proven, an application of Lemma~\ref{lem:kernel} will complete the proof.

\medskip
\noindent
\textbf{Case 1:} $n$ even. Let $B=\sum k_{i,j}E_{i,j}\in\ker(d\varphi),$ and consider the following groups of conditions $B$ must satisfy:

\bigskip
\noindent
\textbf{Group 1:}
\begin{itemize}
    \item $\varphi([E_{\frac{n}{2},\frac{n}{2}},B])=k_{\frac{n}{2},n-1}=0,$
    \item $\varphi([E_{n-i,n-i},B])=-k_{i,n-i}=0,$ for $2\leq i\leq \frac{n}{2}-1,$
    \item $\varphi([E_{n-j,n-i},B])=-k_{i,n-j}=0,$ for $1\leq i<j\leq \frac{n}{2}-1,$
    \item $\varphi([E_{i,n-j},B])=k_{n-j,n-i}=0,$ for $1\leq i<j\leq \frac{n}{2},$
    \item $\varphi([E_{i,n-j},B])=-k_{j,i}=0,$ for $1\leq j<i\leq \frac{n}{2}-1.$
\end{itemize}

\bigskip
\noindent
\textbf{Group 2:}
\begin{itemize}
    \item $\varphi([E_{n-1,n-1},B])=-k_{1,n-1}-k_{\frac{n}{2},n-1}=0,$
    \item $\varphi([E_{n,n},B])=-\displaystyle\sum_{i=1}^{\frac{n}{2}-1}k_{i,n}=0,$
    \item $\varphi([E_{i,i},B])=k_{i,n-i}+k_{i,n}=0,$ for $1\leq i\leq \frac{n}{2}-1,$
    \item $\varphi([E_{\frac{n}{2},n-i},B])=-k_{i,\frac{n}{2}}+k_{n-i,n-1}=0,$ for $2\leq i\leq \frac{n}{2}-1,$
    \item $\varphi([E_{i,j},B])=k_{j,n-i}+k_{j,n}=0,$ for $1\leq i<j\leq \frac{n}{2}-1.$
\end{itemize}

\bigskip
\noindent
\textbf{Group 3:}
\begin{itemize}
    \item $\varphi([E_{i,n-i},B])=k_{n-i,n-i}-k_{i,i}=0,$ for $1\leq i\leq \frac{n}{2}-1,$
    \item $\varphi([E_{\frac{n}{2},n-1},B])=k_{n-1,n-1}-k_{\frac{n}{2},\frac{n}{2}}-k_{1,\frac{n}{2}}=0,$
    \item $\varphi([E_{i,n},B])=k_{n,n}-k_{i,i}-\displaystyle\sum_{j=1}^{i-1}k_{j,i}=0,$ for $1\leq i\leq \frac{n}{2}-1.$
\end{itemize}

The restrictions of Groups 1 and 2 imply that $B$ has the form $$B=\sum_{i=1}^nk_{i,i}E_{i,i}+k_{1,\frac{n}{2}}E_{1,\frac{n}{2}}.$$ Simplifying the restrictions of Group 3, we have that 
\begin{align}
    k_{i,i}&=k_{n-i,n-i},\text{\ for\ }1\leq i\leq \frac{n}{2}-1,\\
    k_{n,n}&=k_{i,i},\text{\ for\ }1\leq i\leq \frac{n}{2}-1\\
        k_{n-1,n-1}&=k_{\frac{n}{2},\frac{n}{2}}+k_{1,\frac{n}{2}}.
\end{align}
Combining these conditions with the type-A trace condition $\sum_{i=1}^nk_{i,i}=0,$ we choose $k_{i,i}=1,$ for all $i\in\left(\left[1,\frac{n}{2}-1\right]\cup\left[\frac{n}{2}+1,n\right]\right)\cap\mathbb{Z},$ which then determines $k_{\frac{n}{2},\frac{n}{2}}=-n+1$ and $k_{1,\frac{n}{2}}=n.$ Therefore, $B$ can be written as in the claim.

\bigskip
\noindent
\textbf{Case 2:} $n$ odd. Let $B=\sum k_{i,j}E_{i,j}\in\ker(d\varphi),$ and consider the following groups of conditions $B$ must satisfy:

\bigskip
\noindent
\textbf{Group 1:}
\begin{itemize}
    \item $\varphi([E_{n-i,n-i},B])=-k_{i,n-i}=0,$ for $1\leq i\leq \lfloor\frac{n}{2}\rfloor,$
    \item $\varphi([E_{n-j,n-i},B])=-k_{i,n-j},$ for $1\leq i<j\leq \lfloor\frac{n}{2}\rfloor,$
    \item $\varphi([E_{i,n-j},B])=k_{n-j,n-i}=0,$ for $1\leq i<j\leq \lfloor\frac{n}{2}\rfloor$
    \item $\varphi([E_{i,\lfloor\frac{n}{2}\rfloor},B])=k_{\lfloor\frac{n}{2}\rfloor,n-i}=0,$ for $1\leq i\leq \lfloor\frac{n}{2}\rfloor-1,$
    \item $\varphi([E_{i,n-j},B])=-k_{j,i}=0,$ for $1\leq j<i\leq \lfloor\frac{n}{2}\rfloor-1.$
\end{itemize}

\bigskip
\noindent
\textbf{Group 2:}
\begin{itemize}
    \item $\varphi([E_{n,n},B])=-\displaystyle\sum_{i=1}^{\lfloor\frac{n}{2}\rfloor-1}k_{i,n}=0,$
    \item $\varphi([E_{i,i},B])=k_{i,n-i}+k_{i,n}=0,$ for $1\leq i\leq \lfloor\frac{n}{2}\rfloor-1,$
    \item $\varphi([E_{\lfloor\frac{n}{2}\rfloor,\lfloor\frac{n}{2}\rfloor},B])=k_{\lfloor\frac{n}{2}\rfloor,\lfloor\frac{n}{2}\rfloor+1}+k_{\lfloor\frac{n}{2}\rfloor,n-1}=0,$
    \item $\varphi([E_{i,j},B])=k_{j,n-i}+k_{j,n}=0,$ for $1\leq i<j\leq \lfloor\frac{n}{2}\rfloor-1,$
    \item $\varphi([E_{\lfloor\frac{n}{2}\rfloor,n-i},B])=-k_{i,\lfloor\frac{n}{2}\rfloor}+k_{n-i,n-1}=0,$ for $2\leq i\leq \lfloor\frac{n}{2}\rfloor-1.$
\end{itemize}

\bigskip
\noindent
\textbf{Group 3:}
\begin{itemize}
    \item $\varphi([E_{i,n-i},B])=k_{n-i,n-i}-k_{i,i}=0,$ for $1\leq i\leq \lfloor\frac{n}{2}\rfloor-1,$
    \item $\varphi([E_{\lfloor\frac{n}{2}\rfloor,\lfloor\frac{n}{2}\rfloor+1},B])=k_{\lfloor\frac{n}{2}\rfloor+1,\lfloor\frac{n}{2}\rfloor+1}-k_{\lfloor\frac{n}{2}\rfloor,\lfloor\frac{n}{2}\rfloor}+k_{\lfloor\frac{n}{2}\rfloor+1,n-1}=0,$
    \item $\varphi([E_{\lfloor\frac{n}{2}\rfloor,n-1},B])=k_{n-1,n-1}-k_{\lfloor\frac{n}{2}\rfloor,\lfloor\frac{n}{2}\rfloor}-k_{1,\lfloor\frac{n}{2}\rfloor}=0,$
    \item $\varphi([E_{i,n},B])=k_{n,n}-k_{i,i}-\displaystyle\sum_{j=1}^{i-1}k_{j,i}=0,$ for $1\leq i\leq \lfloor\frac{n}{2}\rfloor-1.$
\end{itemize}

The restrictions of Groups 1 and 2 imply that $B$ has the form $$x=\sum_{i=1}^nk_{i,i}E_{i,i}+k_{1,\lfloor\frac{n}{2}\rfloor}E_{1,\lfloor\frac{n}{2}\rfloor}.$$
Simplifying the restrictions of Group 3, we have that
\begin{align}
    k_{i,i}&=k_{n-i,n-i},\text{\ for\ }1\leq i\leq\Big\lfloor\frac{n}{2}\Big\rfloor-1\\
    k_{\lfloor\frac{n}{2}\rfloor,\lfloor\frac{n}{2}\rfloor}&=k_{\lfloor\frac{n}{2}\rfloor+1,\lfloor\frac{n}{2}\rfloor+1}\\
    k_{n-1,n-1}&=k_{\lfloor\frac{n}{2}\rfloor,\lfloor\frac{n}{2}\rfloor}+k_{1,\lfloor\frac{n}{2}\rfloor}\\
    k_{n,n}&=k_{i,i},\text{\ for\ }1\leq i\leq\Big\lfloor\frac{n}{2}\Big\rfloor-1.
\end{align}
Combining these conditions with the type-A trace condition $\sum_{i=1}^nk_{i,i}=0,$ we choose $k_{i,i}=1,$ for all $i\in\left(\left[1,\lfloor\frac{n}{2}\rfloor-1\right]\cup\left[\lfloor\frac{n}{2}\rfloor+2,n\right]\right)\cap\mathbb{Z},$ which then determines $k_{\lfloor\frac{n}{2}\rfloor,\lfloor\frac{n}{2}\rfloor}=k_{\lfloor\frac{n}{2}\rfloor+1,\lfloor\frac{n}{2}\rfloor+1}=-\frac{n}{2}+1$ and $k_{1,\lfloor\frac{n}{2}\rfloor}=\frac{n}{2}.$ Therefore, $B$ can be written as in the claim, and this completes the proof.
\end{proof}

\end{document}